\documentclass[11pt]{amsart}
\usepackage{amssymb, amsmath, amscd, eucal, rotating}
\usepackage{pstricks,pstricks-add,pst-node,pst-plot,pst-grad,pst-math,pstcol,pst-3dplot,pst-solides3d}
\usepackage{graphicx}
\usepackage{times,epsf,epsfig,color}

\usepackage{mathrsfs}

\input xy
 \xyoption{all}

\numberwithin{equation}{section}

\theoremstyle{plain}

\newtheorem{proposition}{Proposition}[section]
\newtheorem{theorem}[proposition]{Theorem}		
\newtheorem{corollary}[proposition]{Corollary}
\newtheorem{lemma}[proposition]{Lemma}

\theoremstyle{definition}

\newtheorem{definition}[proposition]{Definition}
\newtheorem{remark}[proposition]{Remark}

\newcommand{\C}{\mathbb C}
\newcommand{\R}{\mathbb R}

\newcommand{\YMH}{\mathop{\rm YMH}\nolimits}

\newcommand{\End}{\mathop{\rm End}\nolimits}
\newcommand{\coker}{\mathop{\rm coker}\nolimits}
\newcommand{\Lie}{\mathop{\rm Lie}\nolimits}

\newcommand{\Hom}{\mathop{\rm Hom}\nolimits}
\newcommand{\Ext}{\mathop{\rm Ext}\nolimits}

\newcommand{\vol}{\mathop{\rm vol}\nolimits}

\newcommand{\ad}{\mathop{\rm ad}\nolimits}
\newcommand{\Ad}{\mathop{\rm Ad}\nolimits}

\newcommand{\tr}{\mathop{\rm Tr}\nolimits}

\newcommand{\GL}{\mathsf{GL}}

\def\overlaprel{\mathrel{\mkern-4mu}}

\newcommand{\longlonglongrightarrow}{\ensuremath{\relbar\overlaprel\relbar\overlaprel\longrightarrow}}

\DeclareMathOperator{\id}{id}
\DeclareMathOperator{\im}{im}
\DeclareMathOperator{\rank}{rank}
\DeclareMathOperator{\slope}{slope}

\DeclareMathOperator{\grad}{grad}

\DeclareMathOperator{\vecspan}{span}

\DeclareMathOperator{\Sec}{Sec}

\begin{document}

% Topmatter

\title{The reverse Yang-Mills-Higgs flow in a neighbourhood of a critical point}

\author{Graeme Wilkin}
\address{Department of Mathematics, National University of Singapore, Singapore 119076}
\email{graeme@nus.edu.sg}

\date{\today}

\begin{abstract}
The main result of this paper is a construction of solutions to the reverse Yang-Mills-Higgs flow converging in the $C^\infty$ topology to a critical point. The construction uses only the complex gauge group action, which leads to an algebraic classification of the isomorphism classes of points in the unstable set of a critical point in terms of a filtration of the underlying Higgs bundle. 

Analysing the compatibility of this filtration with the Harder-Narasimhan-Seshadri double filtration gives an algebraic criterion for two critical points to be connected by a flow line. As an application, we can use this to construct Hecke modifications of Higgs bundles via the Yang-Mills-Higgs flow. When the Higgs field is zero (corresponding to the Yang-Mills flow), this criterion has a geometric interpretation in terms of secant varieties of the projectivisation of the underlying bundle inside the unstable manifold of a critical point, which gives a precise description of broken and unbroken flow lines connecting two critical points. For non-zero Higgs field, at generic critical points the analogous interpretation involves the secant varieties of the spectral curve of the Higgs bundle.
\end{abstract}

\maketitle

%\tableofcontents

\thispagestyle{empty}

\baselineskip=16pt
%\thispagestyle{plain}
%\setcounter{page}{1}

%\setcounter{section}{-1}

%%%%%%%%%%%%%%%%%%%%%%%%%%%%%%%%%%%%%%%%%%%%%%%%

\section{Introduction}

There is a well-known relationship between the Yang-Mills heat flow on a Riemann surface and the notion of stability from algebraic geometry. This began with work of Atiyah and Bott \cite{AtiyahBott83} and continued with Donaldson's proof \cite{Donaldson83} of the Narasimhan-Seshadri theorem \cite{NarasimhanSeshadri65} and subsequent work of Daskalopoulos \cite{Daskal92} and Rade  \cite{Rade92}, which shows that the Yang-Mills flow converges to a unique critical point which is isomorphic to the graded object of the Harder-Narasimhan-Seshadri double filtration of the initial condition. In the setting of Higgs bundles, a theorem of Hitchin \cite{Hitchin87} and Simpson \cite{Simpson88} shows that a polystable Higgs bundle is gauge equivalent to the minimum of the Yang-Mills-Higgs functional and that this minimum is achieved by the heat flow on the space of metrics. The results of \cite{Wilkin08} show that the theorem of Daskalopoulos and Rade described above extends to the Yang-Mills-Higgs flow on the space of Higgs bundles over a compact Riemann surface. More generally, when the base manifold is compact and K\"ahler, then these results are due to \cite{Donaldson85}, \cite{Donaldson87-2}, \cite{UhlenbeckYau86}, \cite{Simpson88}, \cite{DaskalWentworth04}, \cite{Sibley15}, \cite{Jacob15} and \cite{LiZhang11}.

Continuing on from these results, it is natural to investigate flow lines between critical points. Naito, Kozono and Maeda \cite{NaitoKozonoMaeda90} proved the existence of an unstable manifold of a critical point for the Yang-Mills functional, however their method does not give information about the isomorphism classes in the unstable manifold, and their proof requires a manifold structure on the space of connections (which is not true for the space of Higgs bundles). Recent results of Swoboda \cite{Swoboda12} and Janner-Swoboda  \cite{JannerSwoboda15} count flow lines for a perturbed Yang-Mills functional, however these perturbations destroy the algebraic structure of the Yang-Mills flow, and so there does not yet exist an algebro-geometric description of the flow lines in the spirit of the results described in the previous paragraph. Moreover, one would also like to study flow lines for the Yang-Mills-Higgs functional, in which case the perturbations do not necessarily preserve the space of Higgs bundles, which is singular. 

The purpose of this paper is to show that in fact there is an algebro-geometric description of the flow lines connecting given critical points of the Yang-Mills-Higgs functional over a compact Riemann surface. As an application, we show that the Hecke correspondence for Higgs bundles studied by Witten in \cite{witten-hecke} has a natural interpretation in terms of gradient flow lines. Moreover, for the Yang-Mills flow, at a generic critical point there is a natural embedding of the projectivisation of the underlying bundle inside the unstable set of the critical point, and the results of this paper show that the isomorphism class of the limit of the downwards flow is determined if the initial condition lies in one of the secant varieties of this embedding, giving us a geometric criterion to distinguish between broken and unbroken flow lines. For the Yang-Mills-Higgs flow the analogous picture involves the secant varieties of the space of Hecke modifications compatible with the Higgs field. At generic critical points of the Yang-Mills-Higgs functional this space of Hecke modifications is the spectral curve of the Higgs bundle.

The basic setup for the paper is as follows. Let $E \rightarrow X$ be a smooth complex vector bundle over a compact Riemann surface with a fixed Hermitian metric and let $\mathcal{B}$ denote the space of Higgs pairs on $E$. The \emph{Yang-Mills-Higgs functional} is
\begin{equation*}
\YMH(\bar{\partial}_A, \phi) := \| F_A + [\phi, \phi^*] \|_{L^2}^2
\end{equation*}
and the \emph{Yang-Mills-Higgs flow} is the downwards gradient flow of $\YMH$ given by the equation \eqref{eqn:YMH-flow}. This flow is generated by the action of the complex gauge group $\mathcal{G}^\C$. Equivalently, one can fix a Higgs pair and allow the Hermitian metric on the bundle to vary in which case the flow becomes a nonlinear heat equation on the space of Hermitian metrics (cf. \cite{Donaldson85}, \cite{Simpson88}). At a critical point for this flow the Higgs bundle splits into Higgs subbundles and on each subbundle the Higgs structure minimises $\YMH$. The \emph{unstable set} of a critical point $(\bar{\partial}_A, \phi)$ consists of all Higgs pairs for which a solution to the $\YMH$ flow \eqref{eqn:YMH-flow} exists for all negative time and converges in the smooth topology to $(\bar{\partial}_A, \phi)$ as $t \rightarrow - \infty$. The first theorem of the paper gives an algebraic criterion for a complex gauge orbit to intersect the unstable set for the Yang-Mills-Higgs flow. 

\begin{theorem}[Criterion for convergence of reverse heat flow]\label{thm:unstable-set-intro}
Let $E$ be a complex vector bundle over a compact Riemann surface $X$, and let $(\bar{\partial}_A, \phi)$ be a Higgs bundle on $E$. Suppose that $E$ admits a filtration $(E^{(1)}, \phi^{(1)}) \subset \cdots \subset (E^{(n)}, \phi^{(n)}) = (E, \phi)$ by Higgs subbundles such that the quotients $(Q_k, \phi_k) := (E^{(k)}, \phi^{(k)}) / (E^{(k-1)}, \phi^{(k-1)})$ are Higgs polystable and $\slope(Q_k) < \slope(Q_j)$ for all $k < j$. Then there exists $g \in \mathcal{G}^\C$ and a solution to the reverse Yang-Mills-Higgs heat flow equation with initial condition $g \cdot (\bar{\partial}_A, \phi)$ which converges to a critical point isomorphic to $(Q_1, \phi_1) \oplus \cdots \oplus (Q_n, \phi_n)$.

Conversely, if there exists a solution of the reverse heat flow from the initial condition $(\bar{\partial}_A, \phi)$ converging to a critical point $(Q_1, \phi_1) \oplus \cdots \oplus (Q_n, \phi_n)$ where each $(Q_j, \phi_j)$ is polystable with $\slope(Q_k) < \slope(Q_j)$ for all $k < j$, then $(E, \phi)$ admits a filtration $(E^{(1)}, \phi^{(1)}) \subset \cdots \subset (E^{(n)}, \phi^{(n)}) = (E, \phi)$ whose graded object is isomorphic to $(Q_1, \phi_1) \oplus \cdots \oplus (Q_n, \phi_n)$.
\end{theorem}

A key difficulty in the construction is the fact that the space of Higgs bundles is singular, and so the existing techniques for constructing unstable sets (see for example \cite{NaitoKozonoMaeda90} for the Yang-Mills flow or \cite[Sec. 6]{Jost05} in finite dimensions) cannot be directly applied since they depend on the manifold structure of the ambient space. One possibility is to study the unstable set of the function $\| F_A + [\phi, \phi^*] \|_{L^2}^2 + \| \bar{\partial}_A \phi \|_{L^2}^2$ on the space of all pairs $(\bar{\partial}_A, \phi)$ without the Higgs bundle condition $\bar{\partial}_A \phi = 0$, however one would then need a criterion to determine when a point in this unstable set is a Higgs bundle and one would also need a method to determine the isomorphism classes of these points. 

The construction in the proof of Theorem \ref{thm:unstable-set-intro} is intrinsic to the singular space since it uses the action of the complex gauge group to map the unstable set for the linearised $\YMH$ flow (for which we can explicitly describe the isomorphism classes) to the unstable set for the Yang-Mills-Higgs flow. The method used here to compare the flow with its linearisation is called the ``scattering construction'' in \cite{Hubbard05} and \cite{Nelson69} since it originates in the study of wave operators in quantum mechanics (see \cite{ReedSimonVol3} for an overview). The method in this paper differs from \cite{Hubbard05} and \cite{Nelson69} in that (a) the construction here is done using the gauge group action in order to preserve the singular space and (b) the distance-decreasing formula for the flow on the space of metrics \cite{Donaldson85} is used here in order to avoid constructing explicit local coordinates as in \cite{Hubbard05} (the construction of \cite{Hubbard05} requires a manifold structure around the critical points). 

As a consequence of Theorem \ref{thm:unstable-set-intro}, we have an algebraic criterion for critical points to be connected by flow lines.

\begin{corollary}[Algebraic classification of flow lines]\label{cor:algebraic-flow-line-intro}
Let $x_u = (\bar{\partial}_{A_u}, \phi_u)$ and $x_\ell = (\bar{\partial}_{A_\ell}, \phi_\ell)$ be critical points with $\YMH(x_u) > \YMH(x_\ell)$. Then $x_u$ and $x_\ell$ are connected by a flow line if and only if there exists a Higgs pair $(E, \phi)$ which has Harder-Narasimhan-Seshadri double filtration whose graded object is isomorphic to $x_\ell$, and which also admits a filtration $(E^{(1)}, \phi^{(1)}) \subset \cdots \subset (E^{(n)}, \phi^{(n)}) = (E, \phi)$ by Higgs subbundles such that the quotients $(Q_k, \phi_k) := (E^{(k)}, \phi^{(k)}) / (E^{(k-1)} \phi^{(k-1)})$ are polystable and satisfy $\slope(Q_k) < \slope(Q_j)$ for all $k < j$ and the graded object $(Q_1, \phi_1) \oplus \cdots \oplus (Q_n, \phi_n)$ is isomorphic to $x_u$.
\end{corollary}

As an application of the previous theorem, we can construct Hecke modifications of Higgs bundles via Yang-Mills-Higgs flow lines. First consider the case of a Hecke modification at a single point (miniscule Hecke modifications in the terminology of \cite{witten-hecke}).

\begin{theorem}\label{thm:hecke-intro}

\begin{enumerate}

\item Let $0 \rightarrow (E', \phi') \rightarrow (E, \phi) \stackrel{v}{\rightarrow} \C_p \rightarrow 0$ be a Hecke modification such that $(E, \phi)$ is stable and $(E', \phi')$ is semistable, and let $L_u$ be a line bundle with $\deg L_u + 1 < \slope(E') < \slope(E)$. Then there exist sections $\phi_u, \phi_\ell \in H^0(K)$, a line bundle $L_\ell$ with $\deg L_\ell = \deg L_u + 1$ and a metric on $E \oplus L_u$ such that $x_u = (E, \phi) \oplus (L_u, \phi_u)$ and $x_\ell = (E_{gr}', \phi_{gr}') \oplus (L_\ell, \phi_\ell)$ are critical points connected by a $\YMH$ flow line, where $(E_{gr}', \phi_{gr}')$ is isomorphic to the graded object of the Seshadri filtration of $(E', \phi')$.

\item Let $x_u = (E, \phi) \oplus (L_u, \phi_u)$ and $x_\ell = (E', \phi') \oplus (L_\ell, \phi_\ell)$ be critical points connected by a $\YMH$ flow line such that $L_u, L_\ell$ are line bundles with $\deg L_\ell = \deg L_u + 1$, $(E, \phi)$ is stable and $(E', \phi')$ is polystable with $\deg L_u + 1 < \slope(E') < \slope(E)$. Then $(E', \phi')$ is the graded object of the Seshadri filtration of a Hecke modification of $(E, \phi)$. If $(E', \phi')$ is Higgs stable then it is a Hecke modification of $(E, \phi)$. 

\end{enumerate}

\end{theorem}

For Hecke modifications defined at multiple points, we can inductively apply the above theorem to obtain a criterion for two critical points to be connected by a broken flow line. For non-negative integers $m, n$, the definition of $(m, n)$ stability is given in Definition \ref{def:m-n-stable}. The space $\mathcal{N}_{\phi, \phi_u}$ denotes the space of Hecke modifications compatible with the Higgs fields $\phi$ and $\phi_u$ (see Definition \ref{def:Hecke-compatible}).

\begin{corollary}\label{cor:broken-hecke-intro}
Consider a Hecke modification $0 \rightarrow (E', \phi') \rightarrow (E, \phi) \rightarrow \oplus_{j=1}^n \C_{p_j} \rightarrow 0$ defined by $n > 1$ distinct points $\{ v_1, \ldots, v_n \} \in \mathbb{P} E^*$, where $(E, \phi)$ is $(0,n)$ stable. If there exists $\phi_u \in H^0(K)$ such that $v_1, \ldots, v_n \in \mathcal{N}_{\phi, \phi_u}$, then there is a broken flow line connecting $x_u = (E, \phi) \oplus (L_u, \phi_u)$ and $x_\ell = (E_{gr}', \phi_{gr}') \oplus (L_\ell, \phi_\ell)$, where $(E_{gr}', \phi_{gr}')$ is the graded object of the Seshadri filtration of the semistable Higgs bundle $(E', \phi')$.
\end{corollary}

For any gradient flow, given upper and lower critical sets $C_u$ and $C_\ell$, one can define the spaces $\mathcal{F}_{\ell, u}$ (resp. $\mathcal{BF}_{\ell,u}$) of unbroken flow lines (resp. broken or unbroken flow lines) connecting these sets, and the spaces $\mathcal{P}_{\ell, u}$ (resp. $\mathcal{BP}_{\ell, u}$) of pairs of critical points connected by an unbroken flow line (resp. broken or unbroken flow line). These spaces are correspondences with canonical projection maps to the critical sets given by the projection taking a flow line to its upper and lower endpoints.
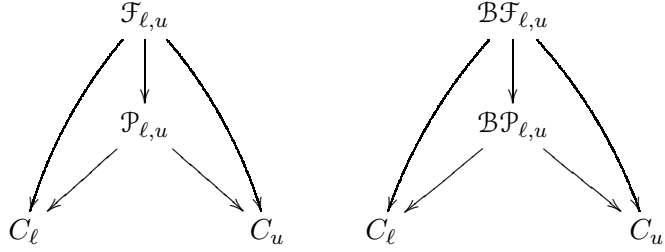
\begin{equation*}
\xymatrix{
 & \mathcal{F}_{\ell, u} \ar[d] \ar@/_/[ddl] \ar@/^/[ddr] &  &  & \mathcal{BF}_{\ell, u} \ar[d] \ar@/_/[ddl] \ar@/^/[ddr] \\
 & \mathcal{P}_{\ell, u} \ar[dl] \ar[dr] & &  & \mathcal{BP}_{\ell, u} \ar[dl] \ar[dr] & \\
C_\ell & & C_u & C_\ell & & C_u
}
\end{equation*}

In the setting of Theorem \ref{thm:hecke-intro}, let $d = \deg E$ and $r = \rank(E)$ and let $C_u$ and $C_\ell$ be the upper and lower critical sets. There are natural projection maps to the moduli space of semistable Higgs bundles $C_u \rightarrow \mathcal{M}_{ss}^{Higgs}(r, d)$ and $C_\ell \rightarrow \mathcal{M}_{ss}^{Higgs}(r, d-1)$. Suppose that $\gcd(r,d) = 1$ so that $\mathcal{M}_{ss}^{Higgs}(r, d)$ consists solely of stable Higgs pairs and hence any Hecke modification is semistable. Since the flow is $\mathcal{G}$-equivariant, then there is an induced correspondence variety, denoted $\mathcal{M}_{\ell, u}$ in the diagram below. 
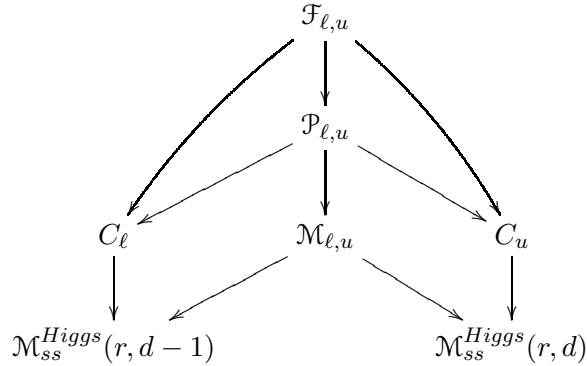
\begin{equation*}
\xymatrix{
 & \mathcal{F}_{\ell, u} \ar[d] \ar@/_/[ddl] \ar@/^/[ddr] &  \\
 & \mathcal{P}_{\ell, u} \ar[dl] \ar[dr] \ar[d] & \\
C_\ell \ar[d] & \mathcal{M}_{\ell,u} \ar[dl] \ar[dr] & C_u \ar[d] \\
\mathcal{M}_{ss}^{Higgs}(r, d-1) & & \mathcal{M}_{ss}^{Higgs}(r,d)
}
\end{equation*}

As a consequence of Theorem \ref{thm:hecke-intro}, we have the following result.

\begin{corollary}\label{cor:hecke-correspondence-intro}
$\mathcal{M}_{\ell,u}$ is the Hecke correspondence.
\end{corollary}

A natural question from Floer theory is to ask whether a pair of critical points connected by a broken flow line can also be connected by an unbroken flow line, i.e whether $\mathcal{BP}_{\ell, u} = \mathcal{P}_{\ell, u}$. The methods used to prove the previous theorems can be used to investigate this question using the geometry of secant varieties of the space of Hecke modifications inside the unstable set of a critical point. For critical points of the type studied in Theorem \ref{thm:hecke-intro}, generically this space of Hecke modifications is the spectral curve of the Higgs field, and so the problem reduces to studying secant varieties of the spectral curve. This is explained in detail in Section \ref{sec:secant-criterion}. In particular, Corollary \ref{cor:rank-2-classification} gives a complete classification of the unbroken flow lines on the space of rank $2$ Higgs bundles.

The paper is organised as follows. In Section \ref{sec:preliminaries} we set the notation for the paper, prove a slice theorem around the critical points and derive some preliminary estimates for the $\YMH$ flow near a critical point. Section \ref{sec:local-analysis} contains the main part of the analysis of the $\YMH$ flow around a critical point, which leads to the proof of Theorem \ref{thm:unstable-set-intro} and Corollary \ref{cor:algebraic-flow-line-intro}. In Section \ref{sec:hecke} we interpret the analytic results on flow lines in terms of the Hecke correspondence, leading to the proof of Theorem \ref{thm:hecke-intro}, Corollary \ref{cor:broken-hecke-intro} and Corollary \ref{cor:hecke-correspondence-intro}. Appendix \ref{sec:uniqueness} contains a proof that a solution to the reverse $\YMH$ flow with a given initial condition is necessarily unique.

{\bf Acknowledgements.} I would like to thank George Daskalopoulos, M.S. Narasimhan and Richard Wentworth for their interest in the project, as well as George Hitching for useful discussions about \cite{ChoeHitching10} and \cite{Hitching13}. 

\section{Preliminaries}\label{sec:preliminaries}

\subsection{The Yang-Mills-Higgs flow on a compact Riemann surface}

Fix a compact Riemann surface $X$ and a smooth complex vector bundle $E \rightarrow X$. Choose a normalisation so that $\vol(X) = 2\pi$. Fix $\bar{\partial}_{A_0} : \Omega^0(E) \rightarrow \Omega^{0,1}(E)$ such that $\bar{\partial}_{A_0}$ is $\C$-linear and satisfies the Leibniz rule $\bar{\partial}_{A_0}(fs) = (\bar{\partial} f) s + f (\bar{\partial}_{A_0} s)$ for all $f \in \Omega^0(X)$ and $s \in \Omega^0(E)$. Let $\mathcal{A}^{0,1}$ denote the affine space $\bar{\partial}_{A_0} + \Omega^{0,1}(\End(E))$. A theorem of Newlander and Nirenberg identifies $\mathcal{A}^{0,1}$ with the space of holomorphic structures on $E$. The \emph{space of Higgs bundles on $E$} is 
\begin{equation}
\mathcal{B} := \{ (\bar{\partial}_A, \phi) \in \mathcal{A}^{0,1} \times \Omega^{1,0}(\End(E)) \, : \, \bar{\partial}_A \phi = 0 \}
\end{equation}
The complex gauge group is denoted $\mathcal{G}^\C$ and acts on $\mathcal{B}$ by $g \cdot (\bar{\partial}_A, \phi) = (g \bar{\partial}_A g^{-1}, g \phi g^{-1})$. If $X$ is a complex manifold with $\dim_\C X > 1$ then we impose the extra integrability conditions $(\bar{\partial}_A)^2 = 0$ and $\phi \wedge \phi = 0$. Given a Hermitian metric on $E$, let $\mathcal{A}$ denote the space of connections on $E$ compatible with the metric, and let $\mathcal{G} \subset \mathcal{G}^\C$ denote the subgroup of unitary gauge transformations. The Chern connection construction defines an injective map $\mathcal{A}^{0,1} \hookrightarrow \mathcal{A}$ which is a diffeomorphism when $\dim_\C X = 1$. Given $\bar{\partial}_A \in \mathcal{A}^{0,1}$, let $F_A$ denote the curvature of the Chern connection associated to $\bar{\partial}_A$ via the Hermitian metric. The metric induces a pointwise norm $| \cdot | : \Omega^2(\End(E)) \rightarrow \Omega^0(X, \R)$ and together with the Riemannian structure on $X$ an $L^2$ norm $\| \cdot \|_{L^2} : \Omega^2(\End(E)) \rightarrow \R$. The \emph{Yang-Mills-Higgs functional} $\YMH : \mathcal{B} \rightarrow \mathbb{R}$ is defined by
\begin{equation}\label{eqn:YMH-def}
\YMH(\bar{\partial}_A, \phi) = \| F_A + [\phi, \phi^*] \|_{L^2}^2 = \int_X | F_A  + [ \phi, \phi^*] |^2 \, dvol
\end{equation}

When $\dim_\C = 1$, the Hodge star defines an isometry $* : \Omega^2(\End(E)) \rightarrow \Omega^0(\End(E)) \cong \Lie \mathcal{G}^\C$. For any initial condition $(A_0, \phi_0)$, the following equation for $g_t \in \mathcal{G}^\C$ has a unique solution on the interval $t \in [0, \infty)$ (cf. \cite{Donaldson85}, \cite{Simpson88})
\begin{equation}\label{eqn:gauge-flow}
\frac{\partial g}{\partial t} g_t^{-1} = - i * ( F_{g_t \cdot A_0} + [g_t \cdot \phi_0, (g_t \cdot \phi_0)^*]) , \quad g_0 = \id .
\end{equation}
This defines a unique curve $(A_t, \phi_t) = g_t \cdot (A_0, \phi_0) \in \mathcal{B}$ which is a solution to the downwards Yang-Mills-Higgs gradient flow equations 
\begin{align}\label{eqn:YMH-flow}
\begin{split}
\frac{\partial A}{\partial t} & = i \bar{\partial}_A * (F_A + [\phi, \phi^*]) \\
\frac{\partial \phi}{\partial t} & = i \left[ \phi, *(F_A + [\phi, \phi^*]) \right] .
\end{split}
\end{align}
for all $t \in [0, \infty)$. The result of \cite[Thm 3.1]{Wilkin08} shows that the solutions converge to a unique limit $(A_\infty, \phi_\infty)$ which is a critical point of $\YMH$. Moreover \cite[Thm. 4.1]{Wilkin08} shows that the isomorphism class of this limit is determined by the graded object of the Harder-Narasimhan-Seshadri double filtration of the initial condition $(A_0, \phi_0)$.

\begin{remark}
Since the space $\mathcal{B}$ of Higgs bundles is singular, then we define the gradient of $\YMH$ as the gradient of the function $\| F_A + [\phi, \phi^*] \|_{L^2}^2$ defined on the ambient smooth space $T^* \mathcal{A}^{0,1}$, which contains the space $\mathcal{B}$ as a singular subset. When the initial condition is a Higgs bundle, then a solution to \eqref{eqn:YMH-flow} is generated by the action of the complex gauge group $\mathcal{G}^\C$ which preserves $\mathcal{B}$. Therefore the solution to \eqref{eqn:YMH-flow} is contained in $\mathcal{B}$ and so from now on we can consider the flow \eqref{eqn:YMH-flow} as a well-defined gradient flow on the singular space $\mathcal{B}$. Throughout the paper we define a critical point to be a stationary point for the Yang-Mills-Higgs flow. 
\end{remark}

\begin{definition}\label{def:critical-point}
A \emph{critical point} for $\YMH$ is a pair $(A, \phi) \in \mathcal{B}$ such that 
\begin{equation}\label{eqn:critical-point}
\bar{\partial}_A * (F_A + [\phi, \phi^*]) = 0, \quad  \text{and} \quad \left[ \phi, *(F_A + [\phi, \phi^*]) \right] = 0 .
\end{equation}
%Let $\mathcal{C} \subset \mathcal{B}$ denote the space of all critical points.
\end{definition}

%{\bf Classification of critical points.}

The critical point equations \eqref{eqn:critical-point} imply that the bundle $E$ splits into holomorphic $\phi$-invariant sub-bundles $E_1 \oplus \cdots \oplus E_n$, such that the induced Higgs structure $(\bar{\partial}_{A_j}, \phi_j)$ on the bundle $E_j$ minimises the Yang-Mills-Higgs functional on the bundle $E_j$ (cf. \cite[Sec. 5]{AtiyahBott83} for holomorphic bundles and \cite[Sec. 4]{Wilkin08} for Higgs bundles). In particular, each Higgs pair $(\bar{\partial}_{A_j}, \phi_j)$ is polystable. The decomposition is not necessarily unique due to the possibility of polystable bundles with the same slope, however it is unique if we impose the condition that $(E_1, \phi_1) \oplus \cdots \oplus (E_n, \phi_n)$ is the graded object of the socle filtration of the Higgs bundle $(E, \phi)$ (see \cite{HuybrechtsLehn97} for holomorphic bundles and \cite[Sec. 4]{BiswasWilkin10} for Higgs bundles). With respect to this decomposition the curvature $*(F_A + [\phi, \phi^*]) \in \Omega^0(\ad(E)) \cong \Lie(\mathcal{G}$ has the following block-diagonal form
\begin{equation}\label{eqn:critical-curvature}
i * (F_A + [\phi, \phi^*]) = \left( \begin{matrix} \lambda_1 \cdot \id_{E_1} & 0 & \cdots & 0 \\ 0 & \lambda_2 \cdot \id_{E_2} & \cdots & 0 \\ \vdots & \vdots & \ddots & \vdots \\ 0 & 0 & \cdots & \lambda_n \cdot \id_{E_n} \end{matrix} \right) 
\end{equation}
where $\lambda_j =\slope(E_j)$ and we order the eigenvalues by $\lambda_j < \lambda_k$ for all $j < k$. 

\begin{definition}
A \emph{Yang-Mills-Higgs flow line} connecting an upper critical point $x_u = (\bar{\partial}_{A_u}, \phi_u)$ and a lower critical point $x_\ell = (\bar{\partial}_{A_\ell}, \phi_\ell)$ is a continuous map $\gamma : \mathbb{R} \rightarrow \mathcal{B}$ such that 
\begin{enumerate}

\item $\frac{d\gamma}{dt}$ satisfies the Yang-Mills-Higgs flow equations \eqref{eqn:YMH-flow}, and

\item $\lim_{t \rightarrow - \infty} \gamma(t) = x_u$ and $\lim_{t \rightarrow \infty} \gamma(t) = x_\ell$, where the convergence is in the $C^\infty$ topology on $\mathcal{B}$.

\end{enumerate}

\end{definition}

\begin{definition}\label{def:unstable-set}
The \emph{unstable set} $W_{x_u}^-$ of a non-minimal critical point $x_u = (\bar{\partial}_{A_u}, \phi_u)$ is defined as the set of all points $y_0 \in \mathcal{B}$ such that a solution $y_t$ to the Yang-Mills-Higgs flow equations \eqref{eqn:YMH-flow} exists for all $(-\infty, 0]$ and $y_t \rightarrow x$ in the $C^\infty$ topology on $\mathcal{B}$ as $t \rightarrow - \infty$.
\end{definition}

\subsection{A local slice theorem}\label{subsec:local-slice}

In this section we define local slices around the critical points and describe the isomorphism classes in the negative slice.

\begin{definition}\label{def:slice}
Let $x = (\bar{\partial}_A, \phi) \in \mathcal{B}$. The \emph{slice} through $x$ is the set of deformations orthogonal to the $\mathcal{G}^\C$ orbit at $x$.
\begin{equation}\label{eqn:slice-def}
S_x = \{ (a, \varphi) \in \Omega^{0,1}(\End(E)) \oplus \Omega^{1,0}(\End(E)) \mid \bar{\partial}_A^* a -*[*\phi^*, \varphi] = 0, (\bar{\partial}_A + a, \phi + \varphi) \in \mathcal{B} \} .
\end{equation}
If $x$ is a critical point of $\YMH$ with $\beta = *(F_A + [\phi, \phi^*])$, then the \emph{negative slice} $S_x^-$ is the subset
\begin{equation}\label{eqn:neg-slice-def}
S_x^- = \{ (a, \varphi) \in S_x \mid \lim_{t \rightarrow \infty} e^{i \beta t} \cdot (a, \varphi) = 0 \} .
\end{equation}
\end{definition}

To prove Lemma \ref{lem:slice-theorem} and Proposition \ref{prop:filtered-slice-theorem} below, one needs to first define the slice on the $L_1^2$ completion of the space of Higgs bundles with the action of the $L_2^2$ completion of the gauge group. The following lemma shows that if the critical point $x$ is $C^\infty$ then the elements in the slice $S_x$ are also $C^\infty$.

\begin{lemma}\label{lem:slice-smooth}
Let $x = (\bar{\partial}_A, \phi)$ be a critical point of $\YMH$ in the space of $C^\infty$ Higgs bundles, let $S_x$ be the set of solutions to the slice equations in the $L_1^2$ completion of the space of Higgs bundles and let $\delta x = (a, \varphi) \in S_x$. Then $\delta x$ is $C^\infty$. 
\end{lemma}

\begin{proof}
The slice equations are
\begin{align*}
\bar{\partial}_A \varphi + [a, \phi] + [a, \varphi] & = 0 \\
\bar{\partial}_A^* a - *[\phi^*, *\varphi] & = 0
\end{align*}
Since $(a, \varphi) \in L_1^2$ and $(\bar{\partial}_A, \phi)$ is $C^\infty$, then the second equation above implies that $\bar{\partial}_A^* a \in L_1^2$ and so $a \in L_2^2$ by elliptic regularity. After applying Sobolev multiplication $L_2^2 \times L_1^2 \rightarrow L^4$, then $[a, \varphi] \in L^4$ and so the first equation above implies that $\bar{\partial}_A \varphi \in L^4$, hence $\varphi \in L_1^4$. Repeating this again shows that $\varphi \in L_2^2$, and then one can repeat the process inductively to show that $\delta x = (a, \varphi)$ is $C^\infty$.
\end{proof}

The following result gives a local description of the space of Higgs bundles in terms of the slice. The infinitesimal action of $\mathcal{G}^\C$ at $x \in \mathcal{B}$ is denoted by $\rho_x : \Lie(\mathcal{G}^\C) \cong \Omega^0(\End(E)) \rightarrow \Omega^{0,1}(\End(E)) \oplus \Omega^{1,0}(\End(E))$. Explicitly, for $x = (\bar{\partial}_A, \phi)$ and $u \in \Omega^0(\End(E))$, we have $\rho_x(u) = -(\bar{\partial}_A u, [\phi, u])$. The $L^2$-orthogonal complement of $\ker \rho_x \subseteq \Omega^0(\End(E))$ is denoted $(\ker \rho_x)^\perp$. 

\begin{lemma}\label{lem:slice-theorem}
Fix $x \in \mathcal{B}$. Then the map $\psi : (\ker \rho_x)^\perp \times S_x \rightarrow \mathcal{B}$ given by $\psi(u, \delta x) = \exp(u) \cdot (x + \delta x)$ is a local homeomorphism.
\end{lemma}

\begin{proof}
The result of \cite[Prop. 4.12]{Wilkin08} shows that the statement is true for the $L_1^2$ completion of the space of Higgs bundles and the $L_2^2$ completion of the gauge group, and so it only remains to show that it remains true on restricting to the space of $C^\infty$ Higgs bundles with the action of the group of $C^\infty$ gauge transformations. The proof of this statement follows from elliptic regularity using the same method as \cite[Cor. 4.17]{Wilkin08}.
\end{proof}

Now let $x = (\bar{\partial}_A, \phi)$ be a critical point and let $\beta = \mu(x) := *(F_A + [\phi, \phi^*])$. The Lie algebra $\Lie(\mathcal{G}^\C) \cong \Omega^0(\End(E))$ decomposes into eigenbundles for the adjoint action of $e^{i \beta}$. We denote the positive, zero and negative eigenspaces respectively by $\Omega^0(\End(E)_+)$, $\Omega^0(\End(E)_0)$ and $\Omega^0(\End(E)_-)$. The positive and negative eigenspaces are nilpotent Lie algebras with associated unipotent groups $\mathcal{G}_+^\C$ and $\mathcal{G}_-^\C$. The subgroups of $\mathcal{G}$ and $\mathcal{G}^\C$ consisting of elements commuting with $e^{i \beta}$ are denoted $\mathcal{G}_\beta$ and $\mathcal{G}_\beta^\C$ respectively. Since $\Omega^0(\End(E)_0) \oplus \Omega^0(\End(E)_+)$ is also a Lie algebra then there is a corresponding subgroup denoted $\mathcal{G}_*^\C$.

Let $\mathcal{G}_x$ and $\mathcal{G}_x^\C$ denote the respective isotropy groups of $x$ in $\mathcal{G}$ and $\mathcal{G}^\C$. There is an inclusion $(\mathcal{G}_x)^\C \subseteq \mathcal{G}_x^\C$, however at a non-minimal critical point the two groups may not be equal (in the context of reductive group actions on finite-dimensional affine spaces, this question has been studied by Sjamaar in \cite[Prop. 1.6]{Sjamaar95}). At a general critical point, the Higgs bundle $(E, \phi)$ splits into polystable Higgs sub-bundles $(E_1, \phi_1) \oplus \cdots \oplus (E_n, \phi_n)$, where we order by increasing slope. Then a homomorphism $u \in \Hom(E_j, E_k)$ satisfying $u \phi_j = \phi_k u$ will be zero if $j > k$ since $(E_j, \phi_j)$ and $(E_k, \phi_k)$ are polystable and $\slope(E_j) > \slope(E_k)$, however if $j < k$ then the homomorphisms do not necessarily vanish in which case $(\mathcal{G}_x)^\C \subsetneq \mathcal{G}_x^\C$. Therefore $\ker \rho_x = \Lie(\mathcal{G}_x^\C) \subset \Omega^0(\End(E)_+) \oplus \Omega^0(\End(E)_0)$, and so $\mathcal{G}_x^\C \subset \mathcal{G}_*^\C$.

The result of \cite[Thm. 2.16]{Daskal92} shows that the $L_2^2$ completion of the gauge group satisfies $\mathcal{G}^\C \cong \mathcal{G}_*^\C \times_{\mathcal{G}_\beta} \mathcal{G}$. We will use $(\ker \rho_x)_*^\perp$ to denote $(\ker \rho_x)^\perp \cap (\Omega^0(\End(E))_+ \oplus \Omega^0(\End(E)_0)$. At a critical point $x$, the above argument shows that isotropy group $\mathcal{G}_x^\C$ is contained in $\mathcal{G}_*^\C$, and so we have the following refinement of Lemma \ref{lem:slice-theorem}.
\begin{proposition}\label{prop:filtered-slice-theorem}
Let $x \in \mathcal{B}$ be a critical point of $\YMH$. Then there exists a $\mathcal{G}$-invariant neighbourhood $U$ of $x$ and a neighbourhood $U'$ of $[\id, 0, 0]$ in $\mathcal{G} \times_{\mathcal{G}_\beta} \left( (\ker \rho_x)_*^\perp \times S_x \right)$ such that $\psi : U' \rightarrow U$ is a $\mathcal{G}$-equivariant homeomorphism.
\end{proposition}

The results of Section \ref{sec:local-analysis} show that the negative slice $S_x^-$ is complex gauge-equivalent to the unstable set $W_x^-$ of a critical point. The next lemma gives a complete classification of the isomorphism classes in $S_x^-$. Together with the results of Section \ref{sec:local-analysis}, this is used in Section \ref{sec:hecke} to classify critical points connected by flow lines.

\begin{lemma}\label{lem:classify-neg-slice}
Let $x = (E_1, \phi_1) \oplus \cdots \oplus (E_n, \phi_n)$ be a critical point of $\YMH$ with curvature as in \eqref{eqn:critical-curvature} with the Higgs polystable subbundles ordered so that $\slope(E_j) < \slope(E_k)$ iff $j < k$. If $\delta x \in S_x^- \cap U$ then $x + \delta x$ has a filtration $(E^{(1)}, \phi^{(1)}) \subset \cdots \subset (E^{(n)}, \phi^{(n)})$ by Higgs subbundles such that the successive quotients are $(E^{(k)}, \phi^{(k)}) / (E^{(k-1)}, \phi^{(k-1)}) = (E_k, \phi_k)$. Conversely, there exists a neighbourhood $U$ of $x$ such that if a Higgs bundle $y = (E, \phi) \in U$ admits such a filtration then it is gauge equivalent to $x + \delta x$ for some $\delta x \in S_x^-$. 
\end{lemma}

\begin{proof}
The first statement follows directly from the definition of the negative slice in \eqref{eqn:neg-slice-def}.

Let $\End(E)_-$ be the subbundle of $\End(E)$ corresponding to the negative eigenspaces of $i \beta$ and let $\rho_x^- : \Omega^0(\End(E)_-) \rightarrow \Omega^{0,1}(\End(E)_-) \oplus \Omega^{1,0}(\End(E)_-)$ be the restriction of the infinitesimal action to the negative eigenspaces. Then
\begin{equation*}
\im \rho_x^- = \im \rho_x \cap \Omega^{0,1}(\End(E)_-) \oplus \Omega^{1,0}(\End(E)_-)
\end{equation*}
and
\begin{equation}\label{eqn:negative-orthogonal}
\ker (\rho_x^-)^*  \supseteq \ker \rho_x^* \cap \Omega^{0,1}(\End(E)_-) \oplus \Omega^{1,0}(\End(E)_-)
\end{equation}
Since $\im \rho_x \oplus \ker \rho_x^* \cong \Omega^{0,1}(\End(E)) \oplus \Omega^{1,0}(\End(E))$ by \cite[Lem. 4.9]{Wilkin08} then
\begin{align*}
\Omega^{0,1}(\End(E)_-) \oplus \Omega^{1,0}(\End(E)_-) & = \left( \im \rho_x \oplus \ker \rho_x^* \right) \cap \left( \Omega^{0,1}(\End(E)_-) \oplus \Omega^{1,0}(\End(E)_-) \right) \\
 & \subseteq \left( \im \rho_x^- \oplus \ker (\rho_x^-)^* \right) \subseteq \Omega^{0,1}(\End(E)_-) \oplus \Omega^{1,0}(\End(E)_-)
\end{align*}
and so \eqref{eqn:negative-orthogonal} must be an equality, therefore $\Omega^{0,1}(\End(E)_-) \oplus \Omega^{1,0}(\End(E)_-) \cong \im \rho_x^- \oplus \ker (\rho_x^-)^*$. Therefore the function
\begin{align*}
\psi^- : (\ker \rho_x^-)^\perp \times \ker (\rho_x^-)^* & \rightarrow \Omega^{0,1}(\End(E)_-) \oplus \Omega^{1,0}(\End(E)_-)  \\
 (u, \delta x) & \mapsto e^u \cdot (x + \delta x)
\end{align*}
is a local diffeomorphism at $0$. If $\delta x \in S_x^-$ then $x + \delta x \in \mathcal{B}$, and so $e^u \cdot (x + \delta x) \in \mathcal{B}$, since the complex gauge group preserves the space of Higgs bundles. Conversely, if $e^u \cdot (x + \delta x) \in \mathcal{B}$ then $x + \delta x \in \mathcal{B}$ and so $\delta x \in S_x^-$. Therefore $\psi$ restricts to a local homeomorphism $(\ker \rho_x^-)^\perp \times S_x^- \rightarrow \mathcal{B} \cap \left( \Omega^{0,1}(\End(E)_-) \oplus \Omega^{1,0}(\End(E)_-) \right)$.
\end{proof}

The next two results concern a sequence of points $g_t \cdot z$ in a $\mathcal{G}^\C$ orbit which approach a critical point $x$ in the $L_k^2$ norm and for which $\YMH(z) < \YMH(x)$. Since $x$ is critical and $\YMH(z) < \YMH(x)$ then $x \in \overline{\mathcal{G}^\C \cdot z} \setminus \mathcal{G}^\C \cdot z$, and therefore $\| g_t \|_{L_{k+1}^2} \rightarrow \infty$. The result below shows that the $C^0$ norm of the function $\sigma(h_t) = \tr(h_t) + \tr(h_t^{-1}) - 2 \rank (E)$ must also blow up.

\begin{lemma}\label{lem:GIT-C0-norm-blows-up}
Let $x \in \mathcal{B}$ be a critical point of $\YMH$ and let $z \in \mathcal{B}$ be any point such that $\YMH(z) < \YMH(x)$. Suppose that there exists a sequence of gauge transformations $g_t \in \mathcal{G}^\C$ such that $g_t \cdot z \rightarrow x$ in $L_k^2$. Then the change of metric $h_t = g_t^* g_t$ satisfies $\sup_X \sigma(h_t) \rightarrow \infty$.
\end{lemma}

\begin{proof}
Let $U$ be the neighbourhood of $x$ from Lemma \ref{lem:slice-theorem}. Since $g_t \cdot z \rightarrow x$, then there exists $T$ such that $g_t \cdot z \in U$ for all $t \geq T$. Therefore there exists $f_t$ in a neighbourhood of the identity in $\mathcal{G}^\C$ such that $f_t \cdot g_t \cdot z \in S_x$. The uniqueness of the decomposition from the slice theorem shows that if $t > T$, then $f_{t} \cdot g_{t} \cdot z = f_{t, T} \cdot f_{T} \cdot g_{T} \cdot z$ with $f_{t,T} \in \mathcal{G}_x^\C$. Therefore $t \rightarrow \infty$ implies that $f_{t,T}$ diverges in $\mathcal{G}_x^\C$. Fix a point $p$ on the surface $X$, and let $\mathcal{G}_0^\C$ be the normal subgroup of complex gauge transformations that are the identity at $p$, as in \cite[Sec. 13]{AtiyahBott83}. We have the following short exact sequence of groups
\begin{equation*}
1 \rightarrow \mathcal{G}_0^\C \rightarrow \mathcal{G}^\C \rightarrow \GL(n, \C) \rightarrow 1 .
\end{equation*}
Since $\mathcal{G}_0^\C$ acts freely on the space of connections (and hence on $\mathcal{B}$), then restriction to the fibre over $p$ defines a bijective correspondence between $\mathcal{G}_x^\C \subset \mathcal{G}^\C$ and a subgroup of $\GL(n, \C)$ via the exact sequence above. Therefore $f_{t,T}$ diverges in $\mathcal{G}_x^\C$ implies that the restriction of $f_{t,T}$ to the fibre over $p$ diverges in $\GL(n, \C)$, and so the $C^0$ norm of $f_{t,T}$ diverges to $\infty$, and hence the same is true for $g_t = f_t^{-1} \cdot f_{t, T} \cdot f_T \cdot g_T \cdot z$ since $g_T$ is fixed and both $f_t$ and $f_T$ are contained in a fixed neighbourhood of the identity in $\mathcal{G}^\C$. Therefore $\sup_X \sigma(h_t) \rightarrow \infty$.
\end{proof}

\begin{corollary}\label{cor:bounded-metric-away-from-critical}
Let $x$ be a critical point of $\YMH$. Then for each neighbourhood $V$ of $x$ in the $L_k^2$ topology on $\mathcal{B}$ and each constant $C > 0$, there exists a neighbourhood $U$ of $x$ such that if $z \notin V$ and $\YMH(z) < \YMH(x)$, then $y = g \cdot z$ with $h = g^* g$ satisfying $\sup_X \sigma(h) \leq C$ implies that $y \notin U$.
\end{corollary}

\begin{proof}
If no such neighbourhood $U$ exists, then we can construct a sequence $y_t = g_t \cdot z$ converging to $x$ in $L_k^2$ such that $h_t = g_t^* g_t$ satisfies $\sup_X \sigma(h_t) \leq C$ for all $t$, however this contradicts the previous lemma.
\end{proof}

\subsection{Modifying the $\YMH$ flow in a neighbourhood of a critical point}

Let $x$ be a critical point, let $\beta = \mu(x) = *(F_A + [\phi, \phi^*])$, and let $\mathcal{G}_*^\C$ be the subgroup defined in the previous section. In this section we explain how to modify the $\YMH$ flow near $x$ so that the gauge transformation generating the flow is contained in $\mathcal{G}_*^\C$. The reason for modifying the flow is so that we can apply the distance-decreasing formula of Lemma \ref{lem:modified-distance-decreasing}, which is used for the convergence result of Section \ref{subsec:inverse-construction}.

Let $U$ be a $\mathcal{G}$-invariant neighbourhod of $x$ such that $U$ is homeomorphic to a neighbourhood of $[\id, 0, 0]$ in $\mathcal{G} \times_{\mathcal{G}_\beta} \left(  (\ker \rho_x)_*^\perp \times S_x \right)$ by Proposition \ref{prop:filtered-slice-theorem}. Let $V \subset U$ be the image of $(\ker \rho_x)_*^\perp \times S_x$ under the homeomorphism from Proposition \ref{prop:filtered-slice-theorem}. For each $y \in V$, let $\gamma_-(y)$ be the component of $\mu(y)$ in $\Omega^0(\End(E)_-)$. Since $\mu$ is $\mathcal{G}$-equivariant then we can extend $\gamma_-$ equivariantly from $V$ to all of $U$ using the action of $\mathcal{G}$. Define the map $\gamma : U \rightarrow \Lie(\mathcal{G})$ by
\begin{equation}\label{eqn:def-gamma}
\gamma(y) = \gamma_-(y) - \gamma_-(y)^*
\end{equation}

\begin{definition}\label{def:modified-flow}
The \emph{modified flow} with initial condition $y_0 \in U$ is the solution to 
\begin{equation}\label{eqn:modified-flow}
\frac{dy}{dt} = - I \rho_x(\mu(y)) + \rho_x(\gamma(y)) .
\end{equation}
More explicitly, on the space of Higgs bundles $y = (\bar{\partial}_A, \phi)$ satisfies
\begin{align*}
\frac{\partial A}{\partial t} & = i \bar{\partial}_A * (F_A + [\phi, \phi^*]) - \bar{\partial}_A \gamma(\bar{\partial}_A, \phi) \\
\frac{\partial \phi}{\partial t} & = i [\phi, *(F_A + [\phi, \phi^*])] - [\phi, \gamma(\bar{\partial}_A, \phi)] 
\end{align*}
\end{definition}

In analogy with \eqref{eqn:gauge-flow}, the modified flow is generated by the action of the gauge group $y_t = g_t \cdot y_0$, where $g_t$ satisfies the equation
\begin{equation}\label{eqn:modified-gauge-flow}
\frac{\partial g_t}{\partial t} g_t^{-1} = -i \mu(g_t \cdot y_0) + \gamma(g_t \cdot y_0), \quad g_0 = \id .
\end{equation}
As before, let $V \subset U$ be the image of $(\ker \rho_x)_*^\perp \times S_x$ under the homeomorphism from the slice theorem (Proposition \ref{prop:filtered-slice-theorem}). Note that if $y_0 \in V$ then $\frac{\partial g_t}{\partial t} g_t^{-1} \in \Lie(\mathcal{G}_*^\C)$, so $g_t \in \mathcal{G}_*^\C$ and the solution to the modified flow remains in $V$ for as long as it remains in the neighbourhood $U$.

\begin{lemma}\label{lem:relate-flows}
Let $y_t = g_t \cdot y_0$ be the solution to the $\YMH$ flow \eqref{eqn:gauge-flow} with initial condition $y_0$. Then there exists $s_t \in \mathcal{G}$ solving the equation 
\begin{equation}\label{eqn:unitary-modification}
\frac{ds}{dt} s_t^{-1} = \gamma(s_t \cdot y_t), \quad s_0 = \id
\end{equation}
such that $\tilde{y}_t = s_t \cdot y_t$ is a solution to the modified flow equation \eqref{eqn:modified-flow} with initial condition $y_0$.
\end{lemma}

\begin{proof}
Since $\gamma$ is $\mathcal{G}$-equivariant then \eqref{eqn:unitary-modification} reduces to 
\begin{equation*}
\frac{ds}{dt} s_t^{-1} = \Ad_{s_t} \gamma(y_t) .
\end{equation*}
Since $\gamma(y_t) \in \Lie(\mathcal{G})$ is already defined by the gradient flow $y_t$, then this equation reduces to solving an ODE on the fibres of the bundle, and therefore existence of solutions follows from ODE existence theory. Let $\tilde{g}_t = s_t \cdot g_t$. A calculation shows that
\begin{align*}
\frac{d \tilde{g}_t}{dt} \tilde{g}_t^{-1} & = \frac{ds}{dt} s_t^{-1} + \Ad_{s_t} \left( \frac{dg}{dt} g_t^{-1} \right) \\
 & = \gamma(s_t \cdot y_t) - i \Ad_{s_t} \mu(y_t) \\
 & = \gamma(\tilde{y}_t) - i \mu(\tilde{y}_t) \\
 & = \gamma(\tilde{g}_t \cdot y_0) - i \mu(\tilde{g}_t \cdot y_0) ,
\end{align*}
and so $\tilde{y}_t = \tilde{g}_t \cdot y_0 = s_t \cdot y_t$ is a solution to the modified flow \eqref{eqn:modified-flow} with initial condition $y_0$.
\end{proof}

As a corollary, we see that the change of metric is the same for the YMH flow \eqref{eqn:gauge-flow} and the modified flow \eqref{eqn:modified-gauge-flow}.

\begin{corollary}\label{cor:metrics-same}
Let $y_t = g_t \cdot y_0$ be a solution to the Yang-Mills-Higgs flow equation \eqref{eqn:gauge-flow} and $\tilde{y}_t = \tilde{g}_t \cdot y_0$ be a solution to the modified flow equation \eqref{eqn:modified-gauge-flow}. Then $h_t = g_t^* g_t = \tilde{g}_t^* \tilde{g}_t$.
\end{corollary}

Finally, we prove that convergence for the upwards $\YMH$ flow implies convergence for the modified flow.
\begin{lemma}\label{lem:unstable-sets-same}
Let $x$ be a critical point and let $y_0 \in W_x^-$. Then the modified flow with initial condition $y_0$ exists for all $t \in (-\infty, 0]$ and converges in the $C^\infty$ topology to a point in $\mathcal{G} \cdot x$. 
\end{lemma}

\begin{proof}
Let $y_t$ be the $\YMH$ flow with initial condition $y_0$ and $\tilde{y_t} = s_t \cdot y_t$ the modified flow. By the definition of $W_x^-$ the $\YMH$ flow exists for all $t \in (-\infty, 0]$ and $y_t \rightarrow x$ in the $C^\infty$ topology. Existence of the modified flow then follows from Lemma \ref{lem:relate-flows}. Proposition \ref{prop:exponential-convergence} shows that $y_t \rightarrow x$ exponentially in $L_k^2$ for all $k$, and so the same is true for $\gamma(y_t)$. Therefore the length of the modified flow line satisfies 
\begin{align*}
\int_{-\infty}^0 \| I \rho_{\tilde{y}_t}(\mu(\tilde{y}_t)) - \rho_{\tilde{y}_t}(\gamma(\tilde{y}_t)) \|_{L_k^2} \, dt & = \int_{-\infty}^0 \| I \rho_{y_t}(\mu(y_t)) - \rho_{y_t}(\gamma(y_t)) \|_{L_k^2} \, dt \\
 & \leq \int_{-\infty}^0 \| \rho_{y_t}(\mu(y_t)) \|_{L_k^2} \, dt + \int_{-\infty}^0 \| \rho_{y_t}(\gamma(y_t)) \|_{L_k^2} \, dt
\end{align*}
which is finite since the length $\int_{-\infty}^0 \| \rho_{y_t}(\mu(y_t)) \|_{L_k^2} \, dt$ of the $\YMH$ flow line is finite, $y_t$ is bounded and $\gamma(y_t) \rightarrow 0$ exponentially. This is true for all $k$, and so the modified flow converges in the $C^\infty$ topology.
\end{proof}

\subsection{Preliminary estimates for the $\YMH$  flow in a neighbourhood of a critical point}

% Quadratic estimate for action of $e^{i \beta t}$ on the negative slice.

Given eigenvalues for $i \beta$ labelled by $\lambda_1 \leq \cdots \leq \lambda_k < 0 \leq \lambda_{k+1} \leq \cdots$, for any $y \in S_x^-$ and any norm, we have the Lipschitz bounds
\begin{equation}\label{eqn:lipschitz-slice}
e^{\lambda_1 t} \| y - x \| \leq \| e^{i \beta t} \cdot y - x \| \leq e^{\lambda_k t} \| y - x \|  .
\end{equation}

\begin{lemma}\label{lem:moment-map-quadratic}
For any critical point $x$ there exists $C>0$ such that for any $y \in S_{x}^-$, we have $\| \mu(y) - \beta \|_{C^0} \leq C \| y - x \|_{C^0}^2$. 
\end{lemma}

\begin{proof}
Let $y \in S_x^-$ and define $\delta y := y-x \in V \cong T_x V$. Then the defining equation for the moment map shows that for all $v \in \mathfrak{k}$, we have 
\begin{equation*}
d \mu_x (\delta y) \cdot v = \omega(\rho_x(v), \delta y) = \left< I \rho_x(v), \delta y \right>
\end{equation*}
By the definition of the slice, each $\delta y \in S_x^-$ is orthogonal to the infinitesimal action of $\mathcal{G}^\C$ at $x$, and so $\left< I \rho_x(v), \delta y \right>=0$ for all $v \in \mathfrak{k}$. Therefore $d \mu_x(\delta y) = 0$. Since the moment map $\mu(\bar{\partial}_A, \phi) = F_A + [\phi, \phi^*]$ is quadratic, then we have
\begin{equation*}
\| \mu(y) - \mu(x) \|_{C^0} \leq \| d\mu_x(\delta y) \|_{C^0} + C \| \delta y \|_{C^0}^2 = C \| \delta y \|_{C^0}^2 .
\end{equation*}
Since the moment map is $\mathcal{G}$-equivariant and the norms above are all $\mathcal{G}$-invariant, then the constant $C$ is independent of the choice of critical point in the orbit $\mathcal{G} \cdot x$.
\end{proof}
%More explicitly, for any $(\bar{\partial}_A, \phi) \in \mathcal{B}$ and $(a, \varphi) \in \Omega^{0,1}(\End(E)) \oplus \Omega^{1,0}(\End(E))$ such that $(\bar{\partial}_A + a, \phi + \varphi) \in \mathcal{B}$ the moment map at $(\bar{\partial}_A + a, \phi + \varphi)$ has the expression
%\begin{multline*}
%F_{A+a} + [\phi + \varphi, (\phi + \varphi)^*] = F_A + [\phi, \phi^*] \\ 
%+ d_A (a+a^*) + [\phi, \varphi^*] + [\varphi, \phi^*] + [a, a^*] + [\varphi, \varphi^*]
%\end{multline*}
%Writing $x = (\bar{\partial}_A, \phi)$ and $y = (\bar{\partial}_A + a, \phi + \varphi) = x + \delta y$, this gives us $d\mu_x(\delta y) = d_A (a+a^*) + [\phi, \varphi^*] + [\varphi, \phi^*]$. 

Given $g \in \mathcal{G}^\C$, let $g^*$ denote the adjoint with respect to the Hermitian metric on $E$ and let $\mathcal{G}$ act on $\mathcal{G}^\C$ by left multiplication. In every equivalence class of the space of metrics $\mathcal{G}^\C/ \mathcal{G}$ there is a unique positive definite self-adjoint section $h$, which we use from now on to represent elements of $\mathcal{G}^\C/ \mathcal{G}$. Given $h = g^* g \in \mathcal{G}^\C/ \mathcal{G}$, define $\mu_h : \mathcal{B} \rightarrow \Omega^0(\End(E)) \cong \Lie (\mathcal{G}^\C)$ by
\begin{equation}\label{eqn:def-muh}
\mu_h(y) = \Ad_{g^{-1}} \left( \mu(g\cdot y) \right) .
\end{equation}
Since the moment map is $\mathcal{G}$-equivariant, then for any $k \in \mathcal{G}$ we have 
\begin{equation*}
\Ad_{g^{-1}} \Ad_{k^{-1}} \left( \mu(k \cdot g \cdot y) \right) = \Ad_{g^{-1}} \left( \mu(g\cdot y) \right)
\end{equation*}
and so $\mu_h$ is well-defined on $\mathcal{G}^\C/ \mathcal{G}$. The length of a geodesic in the space of positive definite Hermitian matrices is computed in \cite[Ch. VI.1]{Kobayashi87}. Following \cite[Prop. 13]{Donaldson85} (see also \cite[Prop. 6.3]{Simpson88}), it is more convenient to define the distance function $\sigma : \mathcal{G}^\C/ \mathcal{G}\rightarrow \R$ 
\begin{equation}\label{eqn:def-sigma}
\sigma(h) = \tr h + \tr h^{-1} - 2 \rank (E) .
\end{equation}
As explained in \cite{Donaldson85}, the function $\sup_X \sigma$ is not a norm in the complete metric space $\mathcal{G}^\C/ \mathcal{G}$, however we do have $h_t \stackrel{C^0}{\longrightarrow} h_\infty$ in $\mathcal{G}^\C/ \mathcal{G}$ if and only if $\sup_X \sigma(h_t h_\infty^{-1}) \rightarrow 0$. Note that if $h_1 = g_1^* g_1$ and $h_2 = g_2^* g_2$, then
\begin{equation}\label{eqn:metric-difference}
\sigma(h_1 h_2^{-1}) = \sigma \left( g_1^* g_1 g_2^{-1} (g_2^*)^{-1} \right) = \sigma \left( (g_1 g_2^{-1})^* g_1 g_2^{-1} \right) .
\end{equation}

Recall from \cite{Donaldson85}, \cite{Simpson88} that we have the following distance-decreasing formula for a solution to the downwards $\YMH$ flow. Since the change of metric is the same for the modified flow by Corollary \ref{cor:metrics-same}, then \eqref{eqn:distance-decreasing} is also valid for the modified flow.

\begin{lemma}\label{lem:distance-decreasing}
Let $y_1, y_2 \in \mathcal{B}$ and suppose that $y_1 = g_0 \cdot y_2$ for $g \in \mathcal{G}^\C$. For $j = 1,2$, define $y_j(t)$ to be the solution of the $\YMH$ flow \eqref{eqn:YMH-flow} with initial condition $y_j$. Define $g_t$ by $y_1(t) = g_t \cdot y_2(t)$ and let $h_t = g_t^* g_t$ be the associated change of metric. Then 
\begin{equation}\label{eqn:distance-decreasing}
\left( \frac{\partial}{\partial t} + \Delta \right) \sigma(h_t) \leq 0 .
\end{equation}
\end{lemma}

Since $\Lie(\mathcal{G}_*^\C) = \Omega^0(\End(E))_0 \oplus \Omega^0(\End(E))_+$ and the adjoint action of $e^{-i \beta t}$ is the identity on $\Omega^0(\End(E))_0$ and strictly contracting on $\Omega^0(\End(E))_+$, then we have the following lemma which is used in Section \ref{subsec:inverse-construction}.

\begin{lemma}\label{lem:modified-distance-decreasing}
Given any $g_0 \in \mathcal{G}_*^\C$, let $g_t = e^{-i \beta t} g_0 e^{i \beta t}$ and $h_t = g_t^* g_t$. Then $\frac{\partial}{\partial t} \sigma(h_t) \leq 0$.
\end{lemma}

As part of the proof of the distance-decreasing formula in \cite{Donaldson85} we also have the following inequalities. This result is used in the proof of Lemma \ref{lem:uniform-bound-sigma}.
\begin{lemma}\label{lem:metric-inequalities}
For any metric $h \in \mathcal{G}^\C / \mathcal{G}$ and any $y \in \mathcal{B}$, we have
\begin{align*}
-2i \tr \left( (\mu_h(y) - \mu(y)) h \right) + \Delta \tr(h) & \leq 0 \\
2i \tr \left( (\mu_h(y) - \mu(y)) h^{-1} \right) + \Delta \tr(h) & \leq 0 .
\end{align*}
\end{lemma}

\subsection{Exponential convergence of the backwards flow}

In this section we prove that if a solution to the backwards $\YMH$ flow converges to a critical point, then it must do so exponentially in each Sobolev norm.

\begin{proposition}\label{prop:exponential-convergence}
Let $y_t$ be a solution to the $\YMH$ flow \eqref{eqn:YMH-flow} such that $\lim_{t \rightarrow -\infty} y_t = x$. Then for each positive integer $k$ there exist positive constants $C_1$ and $\eta$ such that $\| y_t - x \|_{L_k^2} \leq C_1 e^{\eta t}$ for all $t \leq 0$.
\end{proposition}

The proof of the proposition reduces to the following lemmas. First recall from the slice theorem that there is a unique decomposition
\begin{equation*}
y = e^u \cdot (x + z)
\end{equation*}
for $u \in (\ker \rho_x)^\perp$ and $z \in S_x$. We can further decompose $z = z_{\geq 0} + z_-$, where $z_- \in S_x^-$ is the component of $z$ in the negative slice and $z_{\geq 0} = z - z_-$. At the critical point $x$ we have the decomposition $\End(E) \cong \End(E)_+ \oplus \End(E)_0 \oplus \End(E)_-$ according to the eigenspaces of $i \beta$ (cf. Sec. \ref{subsec:local-slice}). Then with respect to this decomposition $z_{\geq 0}$ is the component of $z$ in $\Omega^{0,1}(\End(E))_+ \oplus \End(E)_0) \oplus \Omega^{1,0}(\End(E)_+ \oplus \End(E)_0)$ and $z_-$ is the component in $\Omega^{0,1}(\End(E)_-) \oplus \Omega^{1,0}(\End(E)_-)$. In terms of the action of $\beta = \mu(x)$ we have $\lim_{t \rightarrow \infty} e^{i \beta t} \cdot z_- = 0$ and $\lim_{t \rightarrow \infty} e^{- i \beta t} \cdot z_{\geq 0} = z_0$, where $z_0$ is the  component of $z$ in $\Omega^{0,1}(\End(E)_0) \oplus \Omega^{1,0}(\End(E)_0)$. Note that if $y = e^u \cdot (x + z)$ is a Higgs bundle, then $x+z$ is a Higgs bundle since $e^u \in \mathcal{G}^\C$ preserves the space of Higgs bundles, however $x+z_{\geq 0}$ may not be a Higgs bundle as the pair $(\bar{\partial}_{A_{\geq 0}}, \phi_{\geq 0})$ representing $x + z_{\geq 0}$ may not satisfy $\bar{\partial}_{A_{\geq 0}} \phi_{\geq 0} = 0$. Even though $\phi_{\geq 0}$ may not be holomorphic, we can still apply the principle that curvature decreases in subbundles and increases in quotient bundles and follow the same idea as \cite[Sec. 8 \& 10]{AtiyahBott83} to prove the following lemma.

\begin{lemma}\label{lem:non-holomorphic-extensions}
\begin{enumerate}

\item $\YMH(e^u \cdot (x + z_{\geq 0})) \geq \YMH(x)$.

\item $\grad \YMH(e^u \cdot (x + z_{\geq 0}))$ is tangent to the set $\{z_- = 0\}$.
\end{enumerate}
\end{lemma}

The next lemma shows that the component in the negative slice is decreasing exponentially.

\begin{lemma}
Let $y_t = e^u \cdot (x + z_{\geq 0} + z_-)$ be a solution to the $\YMH$ flow such that $\lim_{t \rightarrow -\infty} y_t = x$. Then there exist positive constants $K_1$ and $K_2$ such that $\| z_- \|_{L_1^2}^2 \leq K_1 e^{K_2 t}$ for all $t \leq 0$.
\end{lemma}

\begin{proof}
The proof follows the idea of \cite[Sec. 10]{Kirwan84}. The downwards gradient flow equation for $z_-$ is
\begin{equation*}
\frac{\partial z_-}{\partial t} = L z_- + N_-(u, z_{\geq 0}, z_-) 
\end{equation*}
where $L$ is a linear operator and the derivative of $N_-$ vanishes at the origin. Since $z_-$ is orthogonal to the $\mathcal{G}^\C$ orbit through $x$, then the Laplacian term in $\grad \YMH$ vanishes on $z_-$ and so the linear part satisfies $e^{Lt} z_- = e^{-i \beta t} \cdot z_-$. Since $z_-$ is in the negative slice then there exists $\lambda_{min} > 0$ such that $\left< L z_- , z_- \right>_{L_1^2} \geq \lambda_{min} \| z_- \|_{L_1^2}$. Now Lemma \ref{lem:non-holomorphic-extensions} shows that the $\YMH$ flow preserves the set $\{ z_- = 0 \}$, and so $N_-(u, z_{\geq 0}, 0) = 0$. Since $N_-$ is $C^1$ with vanishing derivative at the origin then for all $\varepsilon > 0$ there exists $\delta > 0$ such that if $\| y_t - x \|_{L_1^2} < \delta$ then 
\begin{equation*}
\| N_-(u, z_{\geq 0}, z_-) \|_{L_1^2} \leq \varepsilon \| z_- \|_{L_1^2}
\end{equation*}
Therefore
\begin{equation*}
\frac{1}{2} \frac{\partial}{\partial t} \| z_- \|_{L_1^2}^2 = \left< L z_-, z_- \right>_{L_1^2} + \left< N_-(u, z_{\geq 0}, z_- ), z_- \right>_{L_1^2} \geq (\lambda_{min} - \varepsilon) \| z_- \|_{L_1^2}^2 ,
\end{equation*}
and so if $\varepsilon > 0$ is small enough (e.g. $\varepsilon < \frac{1}{2} \lambda_{min}$) then there exist positive constants $K_1$ and $K_2$ such that $\| z_- \|_{L_1^2}^2 \leq K_1 e^{K_2 t}$ for all $t \leq 0$.
\end{proof}

The next lemma shows that the difference $\YMH(x) - \YMH(y_t)$ is decreasing exponentially.

\begin{lemma}\label{lem:f-exponential}
Let $y_t = e^u \cdot (x+z_{\geq 0} + z_-)$ be a solution to the $\YMH$ flow such that $\lim_{t \rightarrow -\infty} y_t = x$. Then there exist positive constants $K_1'$ and $K_2'$ such that
\begin{equation*}
 \YMH(x) - \YMH(e^u \cdot (x + z_{\geq 0} + z_-)) \leq K_1' e^{K_2' t} 
\end{equation*}
for all $t \leq 0$.
\end{lemma}

\begin{proof}
Recall that the Morse-Kirwan condition from Lemma \ref{lem:non-holomorphic-extensions} implies
\begin{equation*}
\YMH(e^u \cdot (x + z_{\geq 0})) - \YMH(x) \geq 0
\end{equation*}
Since $x$ is a critical point of $\YMH$, then for all $\varepsilon > 0$ there exists $\delta > 0$ such that if $\| y_t - x \|_{L_1^2} < \delta$ we have 
\begin{equation*}
\YMH(e^u \cdot (x + z_{\geq 0} + z_-)) - \YMH(e^u \cdot (x + z_{\geq 0})) \geq -\varepsilon \| z_- \|_{L_1^2} .
\end{equation*}
Therefore
\begin{align*}
\YMH(e^u \cdot (x + z_{\geq 0} + z_-)) - \YMH(x) & = \YMH(e^u \cdot (x + z_{\geq 0} + z_-)) - \YMH(e^u \cdot (x + z_{\geq 0})) \\
 & \quad \quad + \YMH(e^u \cdot (x + z_{\geq 0})) - \YMH(x) \\
 &  \geq - \varepsilon \|z_- \|_{L_1^2} \geq - \varepsilon \sqrt{K_1} e^{\frac{1}{2} K_2 t}
\end{align*}
Since $\YMH(e^u \cdot (x + z_{\geq 0} + z_-))$ is monotone decreasing with $t$ and $\lim_{t \rightarrow -\infty} \YMH(e^u \cdot (x+z_{\geq 0} + z_-)) = \YMH(x)$, then $\YMH(e^u \cdot (x + z_{\geq 0} + z_-)) \leq \YMH(x)$, and so the above equation implies that
\begin{equation*}
\left|  \YMH(y_t) - \YMH(x) \right| \leq K_1' e^{K_2' t} 
\end{equation*}
for positive constants $K_1' = \varepsilon \sqrt{K_1}$ and $K_2' = \frac{1}{2} K_2$.
\end{proof}

\begin{lemma}\label{lem:interior-bound}
Let $y_t$ be a solution to the $\YMH$ flow such that $y_t \rightarrow x$ as $t \rightarrow -\infty$. Then for each positive integer $k$ there exists a constant $C$ and a constant $\tau_0 \in \R$ such that 
\begin{equation*}
\| y_\tau - x \|_{L_k^2} \leq C \int_{-\infty}^\tau \| \grad \YMH(y_s) \|_{L^2} \, ds
\end{equation*}
for all $\tau \leq \tau_0$.
\end{lemma}

\begin{proof}
Recall the interior estimate from \cite[Lem. 7.3]{Rade92}, \cite[Prop. 3.6]{Wilkin08} which says that for all positive integers $k$ there exists a neighbourhood $U$ of $x$ in the $L_k^2$ topology and a constant $C$ such that if $y_t \in U$ for all $t \in [0, T]$ then
\begin{equation*}
\int_1^T \| \grad \YMH(y_t) \|_{L_k^2} \, dt \leq C \int_0^T \| \grad \YMH(y_t) \|_{L^2} \, dt .
\end{equation*}
The constant $C$ is uniform as long as the initial condition satisfies a uniform bound on the derivatives of the curvature of the underlying holomorphic bundle and the flow line $y_t$ remains in the fixed neighbourhood $U$ of the critical point $x$ (cf. \cite[Prop. A]{Rade92}). In particular, the estimates of \cite[Lem. 3.14, Cor 3.16]{Wilkin08} show that this bound on the curvature is satisfied for any initial condition along a given flow line $y_t$. \emph{A priori} the constant depends on $T$, however it can be made uniform in $T$ using the following argument. Let $C$ be the constant for $T = 2$. For any $T \geq 2$, let $N$ be an integer greater than $T$ such that $y_t \in U$ for all $t \in [0, N]$. We then have
\begin{align*}
\int_1^T \| \grad \YMH(y_t) \|_{L_k^2} \, dt & \leq \sum_{n=1}^{N-1} \int_n^{n+1} \| \grad \YMH(y_t) \|_{L_k^2} \, dt \\
 & \leq C \sum_{n=1}^{N-1} \int_{n-1}^{n+1} \| \grad \YMH(y_t) \|_{L^2} \, dt \\
 & \leq 2 C \int_0^N \| \grad \YMH(y_t) \|_{L^2} \, dt
\end{align*}
Since $\lim_{t \rightarrow - \infty} y_t = x$ in the $C^\infty$ topology, then for any $\varepsilon > 0$ there exists $\tau_0$ such that $\tau \leq \tau_0$ implies that $\| y_t - x \|_{L_k^2} < \varepsilon$ for all $t \leq \tau$ and therefore by choosing $\varepsilon$ small we can apply the above interior estimate on any interval $[t, \tau]$ for $\tau \leq \tau_0$. Therefore we have the bound
\begin{equation*}
\int_t^\tau \| \grad \YMH(y_s) \|_{L_k^2} \, ds \leq 2C \int_{-\infty}^\tau \| \grad \YMH(y_s) \|_{L^2} \, ds  
\end{equation*}
For fixed $\tau$ the right-hand side of the above inequality is constant, and so
\begin{equation*}
\| y_\tau - x \|_{L_k^2} \leq \int_{-\infty}^\tau \| \grad \YMH(y_s) \|_{L_k^2} \, ds \leq 2C \int_{-\infty}^\tau \| \grad \YMH(y_s) \|_{L^2} \, ds 
\end{equation*}
\end{proof}

\begin{proof}[Proof of Proposition \ref{prop:exponential-convergence}]
After possibly shrinking the neighbourhood $U$ from the previous lemma, we can apply the Lojasiewicz inequality (cf. \cite[Prop. 3.5]{Wilkin08}) which implies that
\begin{equation*}
\int_{-\infty}^\tau \| \grad \YMH(y_s) \|_{L^2} \, ds \leq \frac{1}{C\theta} \left( f(x) - f(y_\tau) \right)^\theta
\end{equation*}
for constants $C > 0$ and $\theta \in (0, \frac{1}{2}]$. Lemma \ref{lem:f-exponential} shows that
\begin{equation*}
\left( f(x) - f(y_\tau) \right)^\theta \leq (K_1')^\theta e^{\theta K_2' t} 
\end{equation*}
for all $t \leq 0$. These two estimates together with the result of Lemma \ref{lem:interior-bound} show that
\begin{equation*}
\| y_t - x \|_{L_k^2} \leq C_1 e^{\eta t}
\end{equation*}
for some positive constants $C_1, \eta$ and all $t \leq 0$.
\end{proof}

\section{The isomorphism classes in the unstable set}\label{sec:local-analysis}

Given a critical point $x \in \mathcal{B}$, in this section we show that for each $y \in S_x^-$ there exists a smooth gauge transformation $g \in \mathcal{G}^\C$ such that $g \cdot y \in W_x^-$ (Proposition \ref{prop:convergence-group-action}), and conversely for each $y \in W_x^-$ there exists $g \in \mathcal{G}^\C$ such that $g \cdot y \in S_x^-$ (Proposition \ref{prop:unstable-maps-to-slice}). As a consequence, the isomorphism classes in the unstable set are in bijective correspondence with the isomorphism classes in the negative slice, and so we have a complete description of these isomorphism classes by Lemma \ref{lem:classify-neg-slice}. This leads to Theorem \ref{thm:algebraic-flow-line} which gives an algebraic criterion for two points to be connected by a flow line. 

\subsection{Convergence of the scattering construction}\label{sec:scattering-convergence}

The goal of this section is to prove Proposition \ref{prop:convergence-group-action}, which shows that every point in the negative slice $S_x^-$ is complex gauge equivalent to a point in the unstable set $W_x^-$. 

The construction involves flowing up towards the critical point on the slice using the linearisation of the $\YMH$ flow and then flowing down using the $\YMH$ flow. A similar idea is used by Hubbard in \cite{Hubbard05} for analytic flows around a critical point in $\C^n$, where the flow on the slice is defined by projecting the flow from the ambient space.  Hubbard's construction uses the fact that the ambient space is a manifold to (a) define this projection to the negative slice, and (b) define local coordinates in which the nonlinear part of the gradient flow satisfies certain estimates in terms of the eigenvalues for the linearised flow \cite[Prop. 4]{Hubbard05}, which is necessary to prove convergence. This idea originated in the study of the existence of scattering states in classical and quantum mechanics. In the context of this paper, one can think of the linearised flow and the YMH flow as two dynamical systems and the goal is to compare their behaviour as $t \rightarrow - \infty$ (see \cite[Ch. XI.1]{ReedSimonVol3} for an overview). As noted in \cite{Hubbard05}, \cite{Nelson69} and \cite{ReedSimonVol3}, the eigenvalues of the linearised flow play an important role in comparing the two flows. 

The method of this section circumvents the need for a local manifold structure by defining the flow on the slice using the linearised flow and then using the distance-decreasing property of the flow on the space of metrics from \cite{Donaldson85}, \cite{Simpson88} (cf. Lemma \ref{lem:distance-decreasing}) in place of the estimate of \cite[Prop. 4]{Hubbard05} on the nonlinear part of the flow. The entire construction is done in terms of the complex gauge group, and so it is valid on any subset preserved by $\mathcal{G}^\C$, thus avoiding any problems associated with the singularities in the space of Higgs bundles. Moreover, using this method it follows naturally from the Lojasiewicz inequality and the smoothing properties of the heat equation that the backwards $\YMH$ flow with initial condition in the unstable set converges in the $C^\infty$ topology.

\subsubsection{A $C^0$ bound on the metric}

First we derive an \emph{a priori} estimate on the change of metric along the flow. Fix an initial condition $y_0 \in S_x^-$ and let $\beta = \mu(x) = \Lambda(F_A + [\phi, \phi^*]) \in \Omega^0(\ad(E)) \cong \Lie(\mathcal{G})$. In this section we also use the function $\mu_h(y) = \Ad_{g^{-1}} \left( \mu(g\cdot y) \right)$ from \eqref{eqn:def-muh}. The linearised flow with initial condition $y_0$ has the form $e^{- i \beta t} \cdot y_0$, and the $\YMH$ flow \eqref{eqn:gauge-flow} has the form $g_t \cdot y_0$. Let $f_t = g_t \cdot e^{i \beta t}$ and define $h_t = f_t^* f_t \in \mathcal{G}^\C / \mathcal{G}$. This is summarised in the diagram below. 

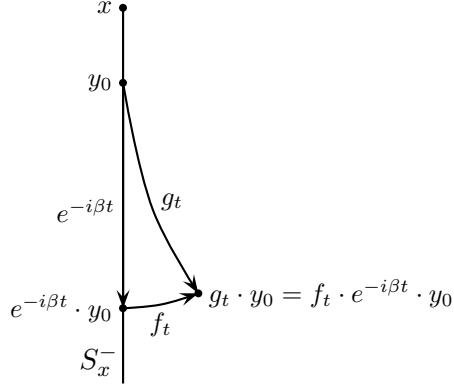
\begin{figure}[ht]
\begin{center}
\begin{pspicture}(2,0)(10,5)
\psline(4,5)(4,0)
\psline[arrowsize=5pt]{->}(4,4)(4,1)
\pscurve[arrowsize=5pt]{->}(4,4)(4.4,2.3)(5,1.2)
\pscurve[arrowsize=5pt]{->}(4,1)(4.5,1.05)(5,1.2)
\psdots[dotsize=3pt](4,5)(4,4)(4,1)(5,1.2)
\uput{4pt}[180](4,5){\small{$x$}}
\uput{4pt}[180](4,4){\small{$y_0$}}
\uput{4pt}[180](4,1){\small{$e^{-i \beta t} \cdot y_0$}}
\uput{4pt}[0](5,1.2){\small{$g_t \cdot y_0 = f_t \cdot e^{-i \beta t} \cdot y_0$}}
\uput{2pt}[180](4,0.3){$S_x^-$}
\uput{3pt}[270](4.5,1.05){\small{$f_t$}}
\uput{3pt}[30](4.4,2.3){\small{$g_t$}}
\uput{3pt}[180](4,2.3){\small{$e^{-i \beta t}$}}
\end{pspicture}
\caption{Comparison of the gradient flow and the linearised flow.}
\end{center}
\end{figure}

\begin{lemma}\label{lem:derivative-difference}
For any initial condition $y_0 \in S_x^-$, the induced flow on $\mathcal{G}^\C / \mathcal{G}$ satisfies
\begin{equation*}
\frac{dh_t}{dt} = -2 i h_t \, \mu_h(e^{-i \beta t} \cdot y_0) + i \beta h_t + h_t (i\beta)
\end{equation*}
\end{lemma}

\begin{proof}
First compute 
\begin{equation}\label{eqn:group-scattering}
\frac{df}{dt} f_t^{-1} = \frac{dg}{dt} g_t^{-1} + g_t (i\beta) e^{i\beta t} f_t^{-1} = -i \mu(g_t \cdot y_0) + f_t (i\beta) f_t^{-1}
\end{equation}
Then 
\begin{align*}
\frac{dh}{dt} & = \frac{df^*}{dt} f_t + f_t^* \frac{df}{dt} \\
 & = f_t^* \left( \frac{df}{dt} f_t^{-1} \right)^* f_t + f_t^* \left( \frac{df}{dt} f_t^{-1} \right) f_t \\
 & = - f_t^* i \mu(g_t \cdot y_0) f_t + i \beta h_t - f_t^* i \mu(g_t \cdot y_0) f_t + h_t (i\beta) \\
 & = -2 f_t^* i \mu(g_t \cdot y_0) f_t + i \beta h_t + h_t (i\beta) \\
 & = -2 i h_t \Ad_{f_t^{-1}} \left( \mu(g_t \cdot y_0) \right) + i \beta h_t + h_t (i\beta) \\
 & = -2 i h_t \, \mu_h(e^{-i \beta t} \cdot y_0) + i \beta h_t + h_t (i\beta)
\end{align*}
where the last step follows from the definition of $\mu_h$ in \eqref{eqn:def-muh} and the fact that $e^{-i \beta t} = f_t^{-1} \cdot g_t$.
\end{proof}

The next estimate gives a bound for $\sup_X \sigma(h_t)$ in terms of $\| y_0 - x \|_{C^0}$.

\begin{lemma}\label{lem:uniform-bound-sigma}
For every $\varepsilon > 0$ there exists a constant $C > 0$ such that for any initial condition $y_0 \in S_{x}^-$ with $\| e^{-i \beta T} \cdot y_0 - x \|_{C^0} < \varepsilon$ we have the estimate $\sup_X \sigma(h_t) \leq C \| e^{-i \beta T} \cdot y_0 - x \|_{C^0}^2$ for all $0 \leq t \leq T$. 
\end{lemma}

\begin{proof}
Taking the trace of the result of Lemma \ref{lem:derivative-difference} gives us
\begin{align*}
\frac{d}{dt} \tr h_t = \tr \left( \frac{dh}{dt} \right) & =  - 2i \tr \left( (\mu_h(e^{-i \beta t} \cdot y_0)  - \beta) h_t \right) \\
\frac{d}{dt} \tr h_t^{-1} = - \tr \left( h_t^{-1} \frac{dh}{dt} h_t^{-1} \right) & = 2i \tr \left( h_t^{-1} (\mu_h(e^{-i \beta t} \cdot y_0) - \beta) \right)
\end{align*}

Therefore
\begin{equation*}
\frac{d}{dt} \tr ( h_t ) = -2i \tr \left( (\mu_h(e^{-i \beta t} \cdot y_0) - \mu(e^{-i\beta t} \cdot y_0) ) h_t \right) - 2i \tr \left( (\mu(e^{-i \beta t} \cdot y_0) - \beta) h_t \right)
\end{equation*}
Lemma \ref{lem:metric-inequalities} together with the fact that $h_t$ is positive definite then shows that
\begin{align*}
\left( \frac{\partial}{\partial t} + \Delta \right) \tr(h_t) & \leq -2i \tr \left( (\mu(e^{-i \beta t} \cdot y_0) - \beta) h_t \right) \\
 & \leq C_1 \| \mu(e^{-i \beta t} \cdot y_0) - \beta \|_{C^0} \tr (h_t) \\
 & \leq C_1 \| e^{-i \beta t} \cdot y_0 - x \|_{C^0}^2 \tr (h_t) \quad \text{(by Lemma \ref{lem:moment-map-quadratic})}
\end{align*}
A similar calculation shows that
\begin{equation*}
\left( \frac{\partial}{\partial t} + \Delta \right) \tr(h_t^{-1}) \leq C_1 \| e^{-i \beta t} \cdot y_0 - x \|_{C^0}^2 \tr (h_t^{-1})
\end{equation*}

If we label the eigenvalues of $i \beta$ as $\lambda_1 \leq \cdots \leq \lambda_k < 0 \leq \lambda_{k+1} \leq \cdots \leq \lambda_n$, then the estimate $\| e^{i \beta s} \cdot (y_0 - x) \|_{C^0}^2 \leq e^{2\lambda_k s} \| y_0 - x \|_{C^0}^2$ from \eqref{eqn:lipschitz-slice} gives us
\begin{align}\label{eqn:sigma-sub-estimate}
\begin{split}
\left( \frac{\partial}{\partial t} + \Delta \right) \sigma(h_t) & = \left( \frac{\partial}{\partial t} + \Delta \right) \left( \tr (h_t) + \tr (h_t^{-1}) \right) \\
 & \leq C_1 \| e^{-i \beta t} \cdot (y_0 - x) \|_{C^0}^2 \left( \tr (h_t) + \tr (h_t^{-1}) \right) \\
 & = C_1 \| e^{i \beta (T-t)} \cdot e^{-i \beta T} \cdot (y_0 - x) \|_{C^0}^2 \left( \tr (h_t) + \tr (h_t^{-1}) \right) \\
 & \leq C_1 e^{2\lambda_k (T-t)} \| e^{-i \beta T} \cdot (y_0 - x) \|_{C^0}^2 \sigma(h_t) + C_1 e^{2 \lambda_k (T-t)} \|  e^{-i \beta T} \cdot (y_0 - x) \|_{C^0}^2 \rank(E)
\end{split}
\end{align}
Let $K_1 = C_1 \| e^{-i \beta T} \cdot y_0 - x  \|_{C^0}^2$ and $K_2 = C_1 \| e^{-i \beta T} \cdot y_0 - x \|_{C^0}^2 \rank(E)$. Define 
\begin{equation*}
\nu_t = \sigma(h_t) \exp\left( \frac{K_1}{2\lambda_k} e^{2 \lambda_k (T-t)} \right) - \int_0^t K_2 e^{2 \lambda_k (T-s)} \exp \left( \frac{K_1}{2\lambda_k} e^{2 \lambda_k (T-s)} \right) \, ds
\end{equation*}
Note that $\nu_0 = 0$ since $h_0 = \id$. A calculation using \eqref{eqn:sigma-sub-estimate} then shows that
\begin{equation*}
\left( \frac{\partial}{\partial t} + \Delta \right) \nu_t \leq 0
\end{equation*}
and so $\sup_X \nu_t \leq \sup_X \nu_0 = 0$ by the maximum principle. Therefore
\begin{align*}
\sup_X \sigma(h_t) & \leq \exp \left(-\frac{K_1}{2\lambda_k} e^{2 \lambda_k (T-t)} \right) \int_0^t K_2 e^{2 \lambda_k (T-s)} \exp \left(\frac{K_1}{2\lambda_k} e^{2 \lambda_k (T-s)} \right) \, ds \\
 & \leq \exp \left( - \frac{K_1}{2\lambda_k} \right) \int_0^t K_2 e^{2\lambda_k (T-s)} \, ds \leq C \| e^{-i \beta T} \cdot y_0 - x  \|_{C^0}^2
\end{align*}
for some constant $C$, since $\lambda_k < 0$, $0 \leq s \leq t < T$, $K_1$ is bounded since $\| e^{-i \beta T} \cdot y_0 - x \|_{C^0} < \varepsilon$ by assumption and $K_2$ is proportional to $\| e^{-i \beta T} \cdot y_0 - x \|_{C^0}^2$.
\end{proof}

\subsubsection{$C^\infty$ convergence in the space of metrics}\label{subsec:metric-convergence}

Now consider the case of a fixed $y_0 \in S_x^-$ and define $y_t = e^{i \beta t} \cdot y_0$. Define $g_s(y_t) \in \mathcal{G}^\C$ to be the unique solution of \eqref{eqn:gauge-flow} such that $g_s(y_t) \cdot y_t$ is the solution to the $\YMH$ flow at at time $s$ with initial condition $y_t$. Let $f_s(y_t) = g_s(y_t) \cdot e^{i \beta s} \in \mathcal{G}^\C$, and define $h_s(y_t) = f_s(y_t)^* f_s(y_t)$ to be the associated change of metric. The estimate from the previous lemma now becomes
\begin{equation}\label{eqn:sigma-C0-estimate}
\sup_X \sigma(h_s(y_t)) \leq C \| y_t - x \|_{C^0}^2 = C \| e^{i \beta t} \cdot (y_0 - x) \|_{C^0}^2 \leq C e^{2 \lambda_k t} \| y_0 - x \|_{C^0}^2 
\end{equation}
This is summarised in the diagram below.

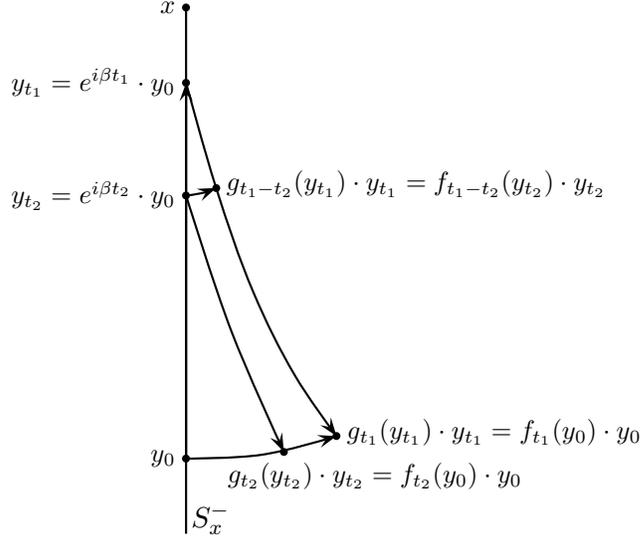
\begin{figure}[ht]\label{fig:flow-up-flow-down}
\begin{center}
\begin{pspicture}(2,-1)(10,6)
\psline(4,6)(4,-1)
\psline[arrowsize=5pt]{->}(4,0)(4,5)
\pscurve[arrowsize=5pt]{->}(4,5)(4.4,3.6)(4.9,2.2)(5.45,1.1)(6,0.3)
\pscurve[arrowsize=5pt]{->}(4,3.5)(4.2,3.53)(4.4,3.6)
\pscurve[arrowsize=5pt]{->}(4,0)(5,0.05)(6,0.3)
\pscurve[arrowsize=5pt]{->}(4,3.5)(4.6,1.7)(5.3,0.09)
\psdots[dotsize=3pt](4,6)(4,5)(4,0)(6,0.3)(4,3.5)(4.4,3.6)(5.3,0.09)
\uput{4pt}[180](4,6){\small $x$}
\uput{4pt}[180](4,5){\small $y_{t_1} = e^{i \beta t_1} \cdot y_0$}
\uput{4pt}[180](4,0){\small $y_0$}
\uput{4pt}[180](4,3.5){\small $y_{t_2}=e^{i \beta t_2} \cdot y_0$}
\uput{4pt}[10](6,0.3){\small $g_{t_1}(y_{t_1}) \cdot y_{t_1} = f_{t_1}(y_0) \cdot y_0$}
\uput{4pt}[300](5.3,0.09){\small{$g_{t_2}(y_{t_2}) \cdot y_{t_2} = f_{t_2}(y_0) \cdot y_0$}}
\uput{4pt}[10](4.4,3.6){\small{$g_{t_1-t_2}(y_{t_1}) \cdot y_{t_1} = f_{t_1-t_2}(y_{t_2}) \cdot y_{t_2}$}}
\uput{2pt}[0](4,-0.8){$S_x^-$}
\end{pspicture}
\end{center}
\caption{Comparison of $f_{t_1}(y_0) \cdot y_0$ and $f_{t_2}(y_0) \cdot y_0$.}
\end{figure}

\begin{proposition}\label{prop:metrics-converge}
$h_t(y_0) \stackrel{C^0}{\longrightarrow} h_\infty(y_0) \in \mathcal{G}^\C / \mathcal{G}$ as $t \rightarrow \infty$. The limit depends continuously on the initial condition $y_0$. The rate of convergence is given by
\begin{equation}\label{eqn:metric-convergence-rate}
\sup_X \sigma(h_t(y_0) (h_\infty(y_0))^{-1}) \leq C_2 e^{2 \lambda_k t} \| y_0 - x \|_{C^0}^2 
\end{equation}
where $C_2 > 0$ is a constant depending only on the orbit $\mathcal{G} \cdot x$.
\end{proposition}

\begin{proof}
Let $t_1 > t_2 \geq T$. The estimate \eqref{eqn:sigma-C0-estimate} shows that
\begin{equation*}
\sup_X \sigma(h_{t_1-t_2}(y_{t_2})) \leq C_2 \| y_{t_2} - x \|_{C^0}^2 \leq C e^{2 \lambda_k t_2} \| y_0 - x \|_{C^0}^2 \leq C e^{2 \lambda_k T} \| y_0 - x \|_{C^0}^2 .
\end{equation*}
Recall from \eqref{eqn:metric-difference} that 
\begin{equation*}
\sigma(h_{t_1}(y_0) h_{t_2}(y_0)^{-1}) = \sigma\left( ( f_{t_1}(y_0) f_{t_2}(y_0)^{-1} )^* f_{t_1}(y_0) f_{t_2}(y_0)^{-1} \right) .
\end{equation*}
The distance-decreasing formula of Lemma \ref{lem:distance-decreasing} shows that 
\begin{equation*}
\sup_X \sigma\left( ( f_{t_1}(y_0) f_{t_2}(y_0)^{-1} )^* f_{t_1}(y_0) f_{t_2}(y_0)^{-1} \right) \leq \sup_X \sigma( h_{t_1-t_2}(y_{t_2}) ) .
\end{equation*}
Therefore the distance (measured by $\sigma$) between the two metrics $h_{t_1}(y_0)$ and $h_{t_2}(y_0)$ satisfies the following bound
\begin{align*}
\sup_X \sigma(h_{t_1}(y_0) h_{t_2}(y_0)^{-1}) & = \sup_X \sigma\left( ( f_{t_1}(y_0) f_{t_2}(y_0)^{-1} )^* f_{t_1}(y_0) f_{t_2}(y_0)^{-1} \right) \\
 & \leq \sup_X \sigma( h_{t_1-t_2}(y_{t_2}) ) \leq C_2 e^{2 \lambda_k T} \| y_0 - x \|_{C^0}^2  
\end{align*}
and so $h_t(y_0)$ is a Cauchy sequence in $C^0$ with a unique limit $h_\infty \in \mathcal{G}^\C / \mathcal{G}$, The above equation shows that the rate of convergence is given by \eqref{eqn:metric-convergence-rate}.

Since the finite-time Yang-Mills-Higgs flow and linearised flow both depend continuously on the initial condition, then $h_t(y_0)$ depends continuously on $y_0$ for each $t > 0$. Continuous dependence of the limit then follows from the estimate \eqref{eqn:metric-convergence-rate}.
\end{proof}

Now we can improve on the previous estimates to show that $h_t(y_0)$ converges in the smooth topology along a subsequence, and therefore the limit $h_\infty$ is $C^\infty$. Define $z_t = f_t(y_0) \cdot y_0$, where $y_0 \in S_x^-$ and $f_t(y_0) \in \mathcal{G}^\C$ are as defined in the previous proposition. Given a Higgs bundle $z_t = (\bar{\partial}_A, \phi)$, let $\nabla_A$ denote the covariant derivative with respect to the metric connection associated to $\bar{\partial}_A$. 

\begin{lemma}
For each initial condition $y_0 \in S_x^-$, there is a uniform bound on $\sup_X | \nabla_A^\ell \mu(z_t) |$ and $\sup_X |\nabla_A^\ell \phi|$ for each $\ell \geq 0$.
\end{lemma}

\begin{proof}
Since $\{ e^{i \beta t} \cdot y_0 \, : \, t \in [0, \infty] \}$ is a compact curve in the space of $C^\infty$ Higgs bundles connecting two $C^\infty$ Higgs bundles $y_0$ and $x$, then $\sup_X \left| \mu(e^{i \beta t} \cdot y_0) \right|$ and $\sup_X \left| \nabla_A \phi \right|$ are both uniformly bounded along the sequence $e^{i \beta t} \cdot y_0$. By construction, $z_t$ is the time $t$ $\YMH$ flow with initial condition $e^{i \beta t} \cdot y_0$. Along the $\YMH$ flow, for each $\ell$ the quantities $\sup_X \left| \nabla_A^\ell \mu \right|$ and $\sup_X \left| \nabla_A^\ell \phi \right|$ are both uniformly bounded by a constant depending on the value of $\sup_X \left| \mu \right|$ and $\sup_X \left| \nabla_A \phi \right|$ at the initial condition (cf. \cite[Sec. 3.2]{Wilkin08}). Since these quantities are uniformly bounded for the initial conditions, then the result follows.
\end{proof}

\begin{corollary}
There is a subsequence $t_n$ such that $h_{t_n} \rightarrow h_\infty$ in the $C^\infty$ topology. Therefore $h_\infty$ is $C^\infty$.
\end{corollary}

\begin{proof}
Since $z_t$ is contained in the complex gauge orbit of $y_0$ for all $t$, then \cite[Lem. 3.14]{Wilkin08} shows that the uniform bound on $\left| \nabla_A^\ell \mu(z_t) \right|$ from the previous lemma implies a uniform bound on $\left| \nabla_A^\ell F_A \right|$ for all $\ell$. Therefore, since Proposition \ref{prop:metrics-converge} shows that $h_t$ converges in $C^0$, then the estimates of \cite[Lem. 19 \& 20]{Donaldson85} show that $h_t$ is bounded in $C^\ell$ for all $\ell$, and so there is a subsequence $h_{t_n}$ converging in the $C^\infty$ topology.
\end{proof}

\subsubsection{$C^\infty$ convergence in the space of Higgs bundles}\label{sec:convergence-in-B}

In this section we show that the scattering construction converges in the $C^\infty$ topology on the space of Higgs bundles. As a consequence of the methods, we obtain an estimate that shows the solution to the reverse heat flow constructed in Section \ref{subsec:construct-reverse-solution} converges to the critical point $x$ in the smooth topology.

This section uses a slightly modified version of the flow from the previous section, defined as follows. Given $y_0 \in S_x^-$ and $t > 0$, let $x_s = g_s \cdot e^{i \beta t} \cdot y_0$ be the time $s$ solution to the $\YMH$ flow \eqref{eqn:gauge-flow} with initial condition $e^{i \beta t} \cdot y_0$, let $s(t)$ be the unique point in time such that $\YMH(x_{s(t)}) = \YMH(y_0)$ and define $t' = \min \{ t, s(t) \}$. Since the critical values of $\YMH$ are discrete, then $t'$ is well-defined for small values of $\YMH(x) - \YMH(y_0)$. 

\begin{center}
\begin{pspicture}(0,-0.5)(8,5)
\psline(4,5)(4,0)
\psline(4,4)(4,1)
\pscurve[arrowsize=5pt]{->}(4,4)(4.4,2.3)(5,1.2)(5.2,1)
%\pscurve[arrowsize=5pt]{->}(4,1)(4.5,1.05)(5,1.2)
\psline[linestyle=dashed](0,1)(8,1)
\psdots[dotsize=3pt](4,5)(4,4)(4,1)(5.2,1)
\uput{4pt}[180](4,5){\small{$x$}}
\uput{4pt}[180](4,4){\small{$e^{i \beta t} \cdot y_0$}}
\uput{4pt}[180](4,0.7){\small{$y_0$}}
\uput{4pt}[0](4.5,0.7){\small{$z_t = g_{t'} \cdot e^{i \beta t} \cdot y_0$}}
\uput{2pt}[180](4,0){$S_x^-$}
%\uput{3pt}[270](4.5,1.05){\small{$f_t$}}
\uput{3pt}[30](4.4,2.3){\small{$g_{t'}$}}
\uput{3pt}[180](4,2.3){\small{$e^{i \beta t}$}}
\uput{3pt}[90](1,1){\small{$\YMH^{-1}(\YMH(y_0))$}}
\end{pspicture}
\end{center}

Now define $z_t = g_{t'} \cdot e^{i \beta t} \cdot y_0$ and $y_t = e^{i \beta (t-t')} \cdot y_0$. Note that $z_t = g_{t'} \cdot e^{i \beta t'} \cdot y_t$ and so the results of the previous section show that the $C^0$ norm of the change of metric connecting $y_t$ and $z_t$ is bounded. Therefore Corollary \ref{cor:bounded-metric-away-from-critical} shows that $y_t$ and $z_t$ are both uniformly bounded away from $x$.

\begin{lemma}\label{lem:bounded-away-from-critical}
There exists $T > 0$ such that $t-t' \leq T$ for all $t$.
\end{lemma}

\begin{proof}
If $s(t) \geq t$ then $t' = t$ and the desired inequality holds. Therefore the only non-trivial case is $s(t) < t$. Since $\YMH(z_t) = \YMH(y_0)$ and $\YMH$ is continuous in the $L_1^2$ norm on $\mathcal{B}$, then there exists a neighbourhood $V$ of $x$ such that $z_t \notin V$ for all $t$. We also have $z_t = f_{t'} \cdot y_t$ with $f_{t'} = g_{t'} e^{i \beta t'}$ such that $h_t = f_{t'}^* f_{t'}$ satisfies $\sup_X \sigma(h_t) \leq C \| y_t - x \|_{C^0}^2 \leq C \| y_0 - x \|_{C^0}^2$ by Lemma \ref{lem:uniform-bound-sigma}, and so Corollary \ref{cor:bounded-metric-away-from-critical} shows that there exists a neighbourhood $U$ of $x$ in the $L_1^2$ topology on $\mathcal{B}$ such that $y_t \notin U$. Therefore there exists $\eta > 0$ such that $\| y_t - x \|_{L_1^2} \geq \eta$ and $\| z_t - x \|_{L_1^2} \geq \eta$.

Since $y_t = e^{i \beta (t-t')} \cdot y_0$ and $e^{i \beta s} \cdot y_0$ converges to $x$ as $s \rightarrow \infty$, then there exists $T$ such that $t-t' \leq T$ for all $t$, since otherwise $\| y_t - x \|_{L_1^2} < \eta$ for some $t$ which contradicts the inequality from the previous paragraph.
\end{proof}

Next we use the Lojasiewicz inequality to derive a uniform bound on $\| z_t - x \|_{L_1^2}$.

\begin{lemma}\label{lem:L12-bound}
Given $\varepsilon > 0$ there exists $\delta > 0$ such that for each $y_0 \in S_x^-$ with $\| y_0 - x \|_{L_1^2} < \delta$ there exists a neighbourhood $U$ of $x$ in the $L_1^2$ topology such that $\| z_t - x \|_{L_1^2} < \varepsilon$ for all $t$ such that $e^{i \beta t} \cdot y_0 \in U$.
\end{lemma}

\begin{proof}
Recall from \cite[Prop. 3.5]{Wilkin08} that there exists $\varepsilon_1 > 0$ and constants $C > 0$ and $\theta \in \left( 0, \frac{1}{2} \right)$ such that the Lojasiewicz inequality 
\begin{equation}\label{eqn:lojasiewicz}
\| \grad \YMH(z) \|_{L^2} \geq C \left| \YMH(x) - \YMH(z) \right|^{1-\theta}
\end{equation}
holds for all $z$ such that $\| z - x \|_{L_1^2} < \varepsilon_1$. Recall the interior estimate \cite[Prop. 3.6]{Wilkin08} which says that for any positive integer $k$ there exists $\varepsilon_2 > 0$ and a constant $C_k'$ such that for any solution $x_s = g_s \cdot e^{i \beta t} \cdot y_0$ to the $\YMH$ flow with initial condition $e^{i \beta t} \cdot y_0$ which satisfies $\| x_s - x \|_{L_k^2} < \varepsilon_2$ for all $0 \leq s \leq S$, then we have
\begin{equation}\label{eqn:interior-estimate}
\int_1^S \| \grad \YMH(x_s) \|_{L_k^2} \, dt \leq C_k' \int_0^S \| \grad \YMH(x_s) \|_{L^2} \, dt .
\end{equation}
where the constant $C_k'$ is uniform over all initial conditions in a given $\mathcal{G}^\C$ orbit and for all $S$ such that $\| x_s - x \|_{L_k^2} < \varepsilon_2$ for all $s \in [0, S]$ (cf. Lemma \ref{lem:interior-bound}). Define $\varepsilon' = \min \{ \varepsilon, \varepsilon_1, \varepsilon_2 \}$. A calculation using \eqref{eqn:lojasiewicz} (cf. \cite{Simon83}) shows that any flow line $x_s$ which satisfies $\| x_s - x \|_{L_1^2} < \varepsilon'$ for all $s \in [0,t']$ also satisfies the gradient estimate
\begin{equation*}
C \theta \| \grad \YMH(x_s) \|_{L^2} \leq \frac{\partial}{\partial s} \left| \YMH(x) - \YMH(x_s) \right|^\theta
\end{equation*}
and so if $\| x_s - x \|_{L_1^2} < \varepsilon'$ for all $s < t'$ then
\begin{align}\label{eqn:flow-length-estimate}
\begin{split}
\int_0^{t'} \| \grad \YMH(x_s) \|_{L^2} \, ds & \leq \frac{1}{C\theta} \left( |\YMH(x) - \YMH(x_{t'}) |^\theta - |\YMH(x) - \YMH(x_0)|^\theta \right) \\
 & \leq \frac{1}{C\theta} \left| \YMH(x) - \YMH(x_{t'}) \right|^\theta
\end{split}
\end{align} 
Let $k=1$ in \eqref{eqn:interior-estimate} and choose $\delta > 0$ so that $\| y_0 - x \|_{L_1^2} < \delta$ implies that that $\frac{1}{C \theta} |\YMH(x) - \YMH(y_0)|^\theta \leq  \frac{\varepsilon'}{3C_1'}$, where $C$ and $\theta$ are the constants from the Lojasiewicz inequality \eqref{eqn:lojasiewicz} and $C_1'$ is the constant from \eqref{eqn:interior-estimate} for $k=1$. Therefore, since $\YMH(y_0) = \YMH(x_{t'}) < \YMH(x_\tau) \leq \YMH(x)$ for all $\tau < t'$, then 
\begin{equation}\label{eqn:energy-bound}
\frac{1}{C\theta} \left| \YMH(x) - \YMH(x_\tau) \right|^\theta \leq \frac{\varepsilon'}{3C_1'} \quad \text{for all $\tau < t'$}.
\end{equation}
Since the finite-time $\YMH$ flow depends continuously on the initial condition in the $L_1^2$ norm by \cite[Prop. 3.4]{Wilkin08}, then there exists a neighbourhood $U$ of $x$ such that $x_0 \in U$ implies that $\| x_1 - x \|_{L_1^2} < \frac{1}{3} \varepsilon'$. Choose $t$ large so that $e^{i \beta t} \cdot y_0 = e^{i \beta t'} \cdot y_t \in U$ and let $x_s = g_s \cdot e^{i \beta t} \cdot y_0$ be the solution to the $\YMH$ flow at time $s$ with initial condition $x_0 = e^{i \beta t} \cdot y_0$. Note that $x_{t'} = z_t$. Define
\begin{equation*}
\tau = \sup \{ s \mid \| x_r - x \|_{L_1^2} < \varepsilon' \, \, \text{for all $r \leq s$} \}
\end{equation*}
and note that $\tau > 0$. By definition of $\tau$, the Lojasiewicz inequality \eqref{eqn:lojasiewicz} and the interior estimate \eqref{eqn:flow-length-estimate} are valid for the flow line $x_s$ on the interval $[0,\tau]$. If $\tau < t'$, then \eqref{eqn:flow-length-estimate} and \eqref{eqn:energy-bound} imply that
\begin{align*}
\| x_\tau - x \|_{L_1^2} & \leq \| x_1 - x \|_{L_1^2} + \| x_\tau - x_1 \|_{L_1^2} \\
 & < \frac{1}{3} \varepsilon' + \int_1^\tau \| \grad \YMH(x_s) \|_{L_1^2} \, ds \\
 & \leq \frac{1}{3} \varepsilon' + C_1' \int_0^\tau \| \grad \YMH(x_s) \|_{L^2} \, ds \\
 & \leq \frac{1}{3} \varepsilon' + \frac{C_1'}{C \theta} \left| \YMH(x) - \YMH(x_\tau) \right|^\theta \leq \frac{1}{3} \varepsilon' + \frac{1}{3} \varepsilon' 
\end{align*}
contradicting the definition of $\tau$ as the supremum. Therefore $t' \leq \tau$ and the same argument as above shows that $\| x_{t'} - x_0 \|_{L_1^2} < \frac{2}{3} \varepsilon'$, so we conclude that $z_t = x_{t'}$ satisfies $\| z_t - x \|_{L_1^2} < \frac{2}{3} \varepsilon' < \varepsilon$ for all $t$ such that $e^{i \beta t} \cdot y_0 \in U$. 
\end{proof}

Now that we have a uniform $L_1^2$ bound on $z_t - x$, then we can apply the same idea using the interior estimate \eqref{eqn:interior-estimate} as well as continuous dependence on the initial condition in the $L_k^2$ norm from \cite[Prop. 3.4]{Wilkin08} to prove the following uniform $L_k^2$ bound on $z_t - x$.

\begin{lemma}\label{lem:Lk2-length}
Given $\varepsilon > 0$ and a positive integer $k$ there exists $\delta > 0$ such that if $\| y_0 - x \|_{L_1^2} < \delta$ then there exists a neighbourhood $U$ of $x$ in the $L_k^2$ topology such that $\| z_t - x \|_{L_k^2} < \varepsilon$ for all $t$ such that $e^{i \beta t} \cdot y_0 \in U$.
\end{lemma}

Now we can prove that there is a limit $z_\infty$ in the space of $C^\infty$ Higgs bundles. In Section \ref{subsec:construct-reverse-solution} we will show that $z_\infty \in W_x^-$.

\begin{proposition}\label{prop:strong-convergence}
For each $y_0 \in S_x^-$, let $z_t$ be the sequence defined above. Then there exists $z_\infty \in \mathcal{B}$ such that for each positive integer $k$ there exists a subsequence of $z_t$ converging to $z_\infty$ strongly in $L_k^2$.
\end{proposition}

\begin{proof}
The previous estimate with $k=2$ shows that $\| z_t - x \|_{L_2^2}$ is bounded. Compactness of the embedding $L_{k+1}^2 \hookrightarrow L_k^2$ shows that there is a subsequence $\{ z_{t_n} \}$ converging strongly to a limit $z_\infty$ in $L_1^2$.  %See Adams-Fournier Theorem 6.3, Part I (borderline case).

For any $k > 1$, the same argument applied to the subsequence $\{ z_{t_n} \}$ from the previous paragraph shows that there exists a further subsequence, which we denote by $\{ z_{t_{n_j}} \}$, which converges strongly in $L_k^2$. Since $z_{t_{n_j}} \stackrel{L_1^2}{\longrightarrow} z_\infty$ then the limit in $L_k^2$ of $z_{t_{n_j}}$ must be $z_\infty$ also. Therefore $z_\infty$ is a $C^\infty$ Higgs pair.
\end{proof}

Finally, we can prove that $z_\infty$ is gauge-equivalent to $y_0$. Recall the constant $T$ from Lemma \ref{lem:bounded-away-from-critical} and let $\varphi(z_t,s)$ denote the time $s$ downwards $\YMH$ flow \eqref{eqn:YMH-flow} with initial condition $z_t$. The gauge transformation $f_t(y_0) \in \mathcal{G}^\C$ from Proposition \ref{prop:metrics-converge} satisfies $f_t(y_0) \cdot y_0 = \phi(z_t, t-t')$. 

For any $k$, let $z_{t_n}$ be a subsequence converging strongly to $z_\infty$ in $L_k^2$. Such a subsequence exists by Proposition \ref{prop:strong-convergence}. Since $0 \leq t_n - t_n' \leq T$ for all $n$ then there exists $s \in [0, T]$ and a subsequence $\{t_{n_\ell} \}$ such that $t_{n_\ell} - t_{n_\ell}' \rightarrow s$. Since the finite-time $\YMH$ flow depends continuously on the initial condition in $L_k^2$, then $f_{t_{n_\ell}} (y_0) \cdot y_0 = \varphi(z_{t_{n_\ell}}, t_{n_\ell} - t_{n_\ell}')$ converges to $z_\infty^0 := \varphi(z_\infty, s)$ strongly in $L_k^2$. After taking a further subsequence if necessary, the method of Section \ref{subsec:metric-convergence} shows that the change of metric associated to $f_{t_{n_\ell}}(y_0)$ converges strongly in $L_{k+1}^2$. Therefore, since the action of the Sobolev completion $\mathcal{G}_{L_{k+1}^2}^\C$ on $\mathcal{B}_{L_k^2}$ is continuous, then $\varphi(z_\infty, s)$ (and hence $z_\infty$) is related to $y_0$ by a gauge transformation in $\mathcal{G}_{L_{k+1}^2}^\C$. Since $y_0$ and $z_\infty$ are both smooth Higgs pairs then an elliptic regularity argument shows that this gauge transformation is smooth. Therefore we have proved the following result.

\begin{proposition}\label{prop:limit-in-group-orbit}
Given any $y_0 \in S_x^-$, let $z_\infty$ be the limit from Proposition \ref{prop:strong-convergence}. Then there exists a smooth gauge transformation $g \in \mathcal{G}^\C$ such that $z_\infty = g \cdot y_0$.
\end{proposition}

Since $z_\infty^0 = \varphi(z_\infty, s)$ is related to $z_\infty$ by the finite-time flow and $s$ is bounded, then we have the following estimate for $\| z_\infty^0 - x \|_{L_k^2}$. Note that this requires a bound on $\| y_0 - x \|_{L_1^2}$ for the estimates of this section to work, and a bound on $\| y_0 - x \|_{C^0}$ for the estimates of Lemma \ref{lem:uniform-bound-sigma} to work.

\begin{corollary}\label{cor:flow-bound}
For all $\varepsilon > 0$ there exists $\delta > 0$ such that $\| y_0 - x \|_{L_1^2} + \| y_0 - x \|_{C^0} < \delta$ implies $\| z_\infty^0 - x \|_{L_k^2} < \varepsilon$. 
\end{corollary}

\begin{remark}
The previous proof uses the fact that the \emph{finite-time} flow depends continuously on the initial condition. The limit of the downwards $\YMH$ flow as $t \rightarrow \infty$ depends continuously on initial conditions within the same Morse stratum (cf. \cite[Thm. 3.1]{Wilkin08}). It is essential that the constant $T$ from Lemma \ref{lem:bounded-away-from-critical} is finite (which follows from Corollary \ref{cor:bounded-metric-away-from-critical}) in order to guarantee that $z_\infty$ and $\varphi(z_\infty, s)$ are gauge equivalent. Without a bound on $T$, it is possible that $z_\infty$ may be in a different Morse stratum to $\lim_{t \rightarrow \infty} \varphi(z_t, t-t')$.
\end{remark}

\subsubsection{Constructing a convergent solution to the backwards $\YMH$ flow}\label{subsec:construct-reverse-solution}

In this section we show that the limit $z_\infty$ is in the unstable set $W_x^-$.

\begin{proposition}\label{prop:convergence-group-action}
For each $y_0 \in S_x^-$ there exists $g \in \mathcal{G}^\C$ such that $g \cdot y_0 \in W_x^-$.
\end{proposition}

\begin{proof}
In what follows, fix any positive integer $k$. Given $y_0 \in S_x^-$, let $z_t^0 = f_t(y_0) \cdot y_0$, where $f_t$ is the complex gauge transformation from Proposition \ref{prop:metrics-converge}. Then Proposition \ref{prop:limit-in-group-orbit} shows that there exists $z_\infty^0 := \phi(z_\infty, s)$ and a subsequence $\{ z_{t_n}^0 \}$ such that $z_{t_n}^0 \rightarrow z_\infty^0$ strongly in $L_k^2$.

For any $s > 0$, let $y_s = e^{i \beta s} \cdot y_0$ and define $z_t^{-s} = f_t(y_s) \cdot y_s$. By definition, $z_t^0$ is the downwards $\YMH$ flow for time $s$ with initial condition $z_t^{-s}$. Applying Proposition \ref{prop:strong-convergence} to the subsequence $z_{t_n}^{-s}$ shows that there is a subsequence $z_{t_{n_j}}^{-s}$ converging in $L_k^2$ to some $z_\infty^{-s}$. Since the $\YMH$ flow for finite time $s$ depends continuously on the initial condition (cf. \cite[Prop. 3.4]{Wilkin08}) then $z_{t_{n_j}}^{-s} \rightarrow z_\infty^{-s}$ and $z_{t_{n_j}}^0 \rightarrow z_\infty^0$ implies that $z_\infty^0$ is the time $s$ flow with initial condition $z_\infty^{-s}$. Therefore, for any $s > 0$ we have constructed a solution to the $\YMH$ flow on $[-s, 0]$ connecting $z_\infty^0$ and $z_\infty^{-s}$. Proposition \ref{prop:backwards-uniqueness} shows that this solution must be unique for each $s$, and therefore there is a well-defined solution on the time interval $(-\infty, 0]$.

Moreover, we also have the uniform bound from Corollary \ref{cor:flow-bound} which shows that for all $\varepsilon > 0$ there exists $\delta > 0$ such that $\| z_\infty^{-s} - x \|_{L_k^2} \leq \varepsilon$ for all $y_0$ such that $\| y_0 - x \|_{L_1^2} < \delta$. Therefore as $s \rightarrow \infty$, the sequence $z_\infty^{-s}$ converges strongly to $x$ in the $L_k^2$ norm for any $k$, and so $z_\infty^0 = g \cdot y_0 \in W_x^-$. Proposition \ref{prop:exponential-convergence} then shows that the convergence is exponential in each Sobolev norm.
\end{proof} 

\subsection{Convergence of the inverse process}\label{subsec:inverse-construction}

In this section we consider the inverse procedure to that of the previous section and prove that each point in the unstable set $W_x^-$ is gauge equivalent to a point in the negative slice $S_x^-$. The idea is similar to that of the previous section, except here we use the modified flow.

\subsubsection{A $C^0$ bound in the space of metrics}

Given $y_0 \in W_x^-$, let $y_t = g_t \cdot y_0$ be the solution to the modified flow \eqref{eqn:modified-flow} with initial condition $y_0$. Define $f_t = g_t \cdot e^{i\beta t}$ and let $h_t = f_t^* f_t$. This is summarised in the diagram below.

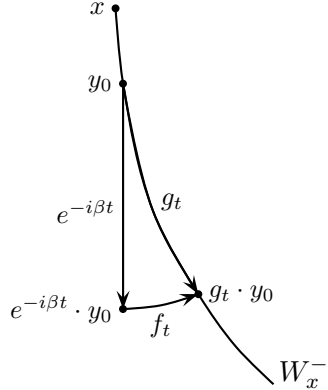
\begin{figure}[ht]
\begin{center}
\begin{pspicture}(2,0)(10,5)
\pscurve(3.9,5)(4,4)(4.4,2.3)(5,1.2)(5.5,0.5)(6,0)
\psline[arrowsize=5pt]{->}(4,4)(4,1)
\pscurve[arrowsize=5pt]{->}(4,4)(4.4,2.3)(5,1.2)
\pscurve[arrowsize=5pt]{->}(4,1)(4.5,1.05)(5,1.2)
\psdots[dotsize=3pt](3.9,5)(4,4)(4,1)(5,1.2)
\uput{4pt}[180](3.9,5){\small{$x$}}
\uput{4pt}[180](4,4){\small{$y_0$}}
\uput{4pt}[180](4,1){\small{$e^{- i \beta t} \cdot y_0$}}
\uput{4pt}[0](5,1.2){\small{$g_t \cdot y_0$}}
\uput{2pt}[30](6,0){$W_x^-$}
\uput{3pt}[270](4.5,1.05){\small{$f_t$}}
\uput{3pt}[30](4.4,2.3){\small{$g_t$}}
\uput{3pt}[180](4,2.3){\small{$e^{- i \beta t}$}}
\end{pspicture}
%\caption{Comparison of the gradient flow and the linearised flow.}
\end{center}
\end{figure}

Using a similar calculation as the previous section, we have the same expression for the change of metric as in Lemma \ref{lem:derivative-difference}.
\begin{lemma}\label{lem:reverse-derivative-difference}
For any initial condition $y_0 \in W_x^-$, the induced flow on $\mathcal{G}^\C / \mathcal{G}$ satisfies
\begin{equation}\label{eqn:metric-derivative}
\frac{dh_t}{dt} = -2i h_t \mu_h(e^{-i \beta t} \cdot y_0) + i \beta h_t + i h_t \beta .
\end{equation}
\end{lemma}

\begin{proof}
A similar calculation as in the proof of Lemma \ref{lem:derivative-difference} (this time using the modified flow \eqref{eqn:modified-flow}) shows that
\begin{equation*}
\frac{df_t}{dt} f_t^{-1} = -i\mu(g_t \cdot y_0) + \gamma(g_t \cdot y_0) + f_t (i \beta) f_t^{-1}  .
\end{equation*}
Then
\begin{align*}
\frac{dh_t}{dt} & = f_t^* \left( \frac{df}{dt} f_t^{-1} \right)^* f_t + f_t^* \left( \frac{df}{dt} f_t^{-1} \right) f_t \\
 & = f_t^* \left( -i \mu(g_t \cdot y_0) - \gamma(g_t \cdot y_0) + (f_t^*)^{-1} (i \beta) f_t^* - i \mu(g_t \cdot y_0) + \gamma(g_t \cdot y_0) + f_t (i \beta) f_t^{-1} \right) f_t \\
 & = -2i h_t f_t^{-1} \mu(g_t \cdot y_0) f_t + i \beta h_t + i h_t \beta \\
 & = -2i h_t \mu_h(e^{-i \beta t} \cdot y_0) + i \beta h_t + i h_t \beta . \qedhere
\end{align*}
\end{proof}

\begin{lemma}\label{lem:reverse-uniform-bound-sigma}
For every $\varepsilon > 0$ there exists a constant $C > 0$ such that for any initial condition $y_0 \in W_{x}^-$ with $\| e^{-i \beta T} \cdot y_0 - x \|_{C^1} + \| g_T \cdot y_0 - x \|_{C^1} < \varepsilon$ we have the estimate 
\begin{equation*}
\sup_X \sigma(h_t) \leq C \left( \| e^{-i \beta T} \cdot y_0 - x \|_{C^1} + \| g_T \cdot y_0 - x \|_{C^1} \right)
\end{equation*}
for all $0 \leq t \leq T$. 
\end{lemma}

\begin{proof}
In contrast to the proof of Lemma \ref{lem:uniform-bound-sigma}, $e^{-i \beta t} \cdot y_0$ is not in the slice $S_x$ and so it satisfies the inequality $\| \mu(e^{-i \beta t} \cdot y_0) - \beta \|_{C^0} \leq C' \| e^{-i \beta t} \cdot y_0 - x \|_{C^1}$ instead of the quadratic bound of Lemma \ref{lem:moment-map-quadratic}. Using this inequality, the same idea as in the proof of Lemma \ref{lem:uniform-bound-sigma} leads to the bound
\begin{align}\label{eqn:heat-operator-bounded}
\begin{split}
\left( \frac{\partial}{\partial t} + \Delta \right) \sigma(h_t) & \leq C_1 \| e^{-i \beta t} \cdot (y_0-x) \|_{C^1} \left( \tr(h_t) + \tr(h_t^{-1}) \right) \\
 & = C_1 \| e^{-i \beta t} \cdot (y_0 - x) \|_{C^1} \sigma(h_t) + 2 C_1 \| e^{-i \beta t} \cdot (y_0 - x) \|_{C^1} \rank(E)
\end{split}
\end{align}
In general, if the heat operator is bounded for all $t \geq 0$
\begin{equation*}
\left( \frac{\partial}{\partial t} + \Delta \right) f(p,t) \leq C(t) f(p,t) + D(t), \quad p \in X, t \in [0, \infty)
\end{equation*}
for some nonnegative functions $C(t)$ and $D(t)$ independent of $p \in X$, then $f(p,t)$ satisfies the bound
\begin{equation}\label{eqn:general-heat-bound}
f(p,t) \leq \exp \left( \int_0^t C(s) \, ds \right) \int_0^t D(s) \, ds + f(p,0)
\end{equation}
Therefore \eqref{eqn:heat-operator-bounded} implies that the problem reduces to finding a bound for $\int_0^t \| e^{-i \beta s} \cdot (y_0 - x) \|_{C^1} \, ds$. Proposition \ref{prop:exponential-convergence} shows that the backwards flow with initial condition in $W_x^-$ converges exponentially to $x$ in every Sobolev norm. Therefore there exists a neighbourhood $U$ of $x$ such that if $g_T \cdot y_0 \in U$ then there exist positive constants $C_1$ and $\eta$ such that the following estimate holds
\begin{equation*}
\| y_0 - x \|_{C^1} \leq C_1 e^{-\eta T} \| g_T \cdot y_0 - x \|_{C^1} .
\end{equation*}
Recall the eigenbundles $\End(E)_-$, $\End(E)_0$ and $\End(E)_+$ from Section \ref{sec:preliminaries}. The above estimate shows that each component of $y_0 - x$ in $\End(E)_-$, $\End(E)_0$ and $\End(E)_+$ is bounded by $C_1 e^{-\eta T} \| g_T \cdot y_0 - x \|_{C^1}$. Since the component of $e^{-i \beta t} \cdot (y_0 - x)$ in $\End(E)_+$ is exponentially decreasing with $t$ then 
\begin{equation*}
\int_0^T \| (e^{-i \beta t} \cdot y_0 - x)_{\End(E)_+} \|_{C^1} dt \leq C_1' \| y_0 - x \|_{C^1} \leq C_1 e^{-\eta T} \| g_T \cdot y_0 - x \|_{C^1} .
\end{equation*}
The component of $e^{-i \beta t} \cdot (y_0 - x)$ in $\End(E)_0$ is constant with respect to $t$, and so 
\begin{equation*}
\int_0^T \| (e^{-i \beta t} \cdot y_0 - x)_{\End(E)_0} \|_{C^1} dt \leq C_2' T \| y_0 - x \|_{C^1} \leq C_2 T e^{-\eta T} \| g_T \cdot y_0 - x \|_{C^1} .
\end{equation*}
Finally, the component of $e^{-i \beta t} \cdot (y_0 - x)$ in $\End(E)_-$ is exponentially increasing, and so we have the bound
\begin{equation*}
\int_0^T \| (e^{-i \beta t} \cdot y_0 - x)_{\End(E)_-} \|_{C^1} \leq C_3 \| e^{-i \beta T} \cdot (y_0 - x) \|_{C^1} .
\end{equation*}

Combining the estimates for the three components shows that the integral 
\begin{equation*}
I(t) = \int_0^t \| e^{-i \beta s} \cdot (y_0-x) \|_{C^1} \, ds
\end{equation*}
is bounded by
\begin{equation*}
I(t) \leq C_1 e^{-\eta T} \| g_T \cdot y_0 - x \|_{C^1} + C_2 T e^{-\eta T} \| g_T \cdot y_0 - x \|_{C^1} + C_3 \| e^{-i \beta T} \cdot y_0 - x \|_{C^1}
\end{equation*}
The inequality \eqref{eqn:general-heat-bound} together with the assumption $\| g_T \cdot y_0 - x \|_{C^1} + \| e^{-i \beta T} \cdot y_0 - x \|_{C^1} < \varepsilon$ shows that there exists a constant $C$ such that
\begin{equation*}
\sup_X \sigma(h_t) \leq C \left( \| e^{-i \beta T} \cdot y_0 - x \|_{C^1} + \| g_T \cdot y_0 - x \|_{C^1} \right) \qedhere
\end{equation*}
\end{proof}

\subsubsection{Convergence in the space of Higgs bundles}

In this section we use a method analogous to that of Section \ref{sec:convergence-in-B} to show that the sequence converges in the space of Higgs bundles and that the limit is gauge equivalent to $y_0$. Given $y_0 \in W_x^-$ and $t \in (-\infty, 0]$, define $s < 0$ by $\| e^{i \beta s} \cdot g_t(y_0) \cdot y_0 - x \|_{L_k^2} = \| y_0 - x \|_{L_k^2}$. Note that this is well-defined for small values of $\| y_0 - x \|_{L_k^2}$ since Lemma \ref{lem:unstable-sets-same} shows that $g_t(y_0) \cdot y_0 \rightarrow x$ in the $C^\infty$ topology as $t \rightarrow - \infty$ and for $s < 0$ the action of $e^{i \beta s}$ exponentially increases the $C^0$ norm of the component of $g_t(y_0) \cdot y_0$ in $\End(E)_-$. Now define $t' := \max \{ t, s \} < 0$, let $f_t(y_0) = e^{i \beta t'} \cdot  g_{t'}(g_{t-t'}(y_0) \cdot y_0) $ and $z_t := e^{i \beta t'} \cdot g_t(y_0) \cdot y_0 = f_t(y_0) \cdot g_{t-t'}(y_0) \cdot y_0$. Let $h_t = f_t^* f_t$ be the associated change of metric.

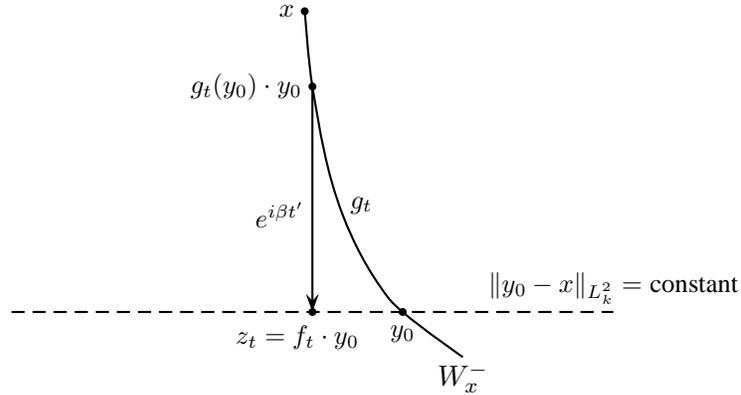
\begin{center}
\begin{pspicture}(0,-0.5)(8,5.5)
\psline[arrowsize=5pt]{->}(4,4)(4,1)
\pscurve(3.9,5)(4,4)(4.4,2.3)(5,1.2)(5.2,1)(6,0.4)
%\pscurve[arrowsize=5pt]{->}(4,1)(4.5,1.05)(5,1.2)
\psline[linestyle=dashed](0,1)(8,1)
\psdots[dotsize=3pt](3.9,5)(4,4)(4,1)(5.2,1)
\uput{4pt}[180](3.9,5){\small{$x$}}
\uput{4pt}[180](4,4){\small{$g_t(y_0) \cdot y_0$}}
\uput{5pt}[270](3.8,1){\small{$z_t = f_t \cdot y_0$}}
\uput{5pt}[270](5.2,1){\small{$y_0$}}
\uput{2pt}[270](6,0.4){$W_x^-$}
%\uput{3pt}[270](4.5,1.05){\small{$f_t$}}
\uput{3pt}[30](4.4,2.3){\small{$g_{t}$}}
\uput{3pt}[180](4,2.3){\small{$e^{i \beta t'}$}}
\uput{3pt}[90](8,1){\small{$\| y_0 - x \|_{L_k^2} = \text{constant}$}}
\end{pspicture}
\end{center}

Lemma \ref{lem:reverse-uniform-bound-sigma} then shows that $\sup_X \sigma(h_t) \leq C \left( \| z_t - x \|_{C^1} + \| g_{t-t'}(y_0) \cdot y_0 - x \|_{C^1} \right)$. Since either $\| z_t - x \|_{L_k^2} = \| y_0 - x \|_{L_k^2}$ (when $t < t'$) or $g_{t-t'}(y_0) \cdot y_0 = y_0$ (when $t'=t$), then  Corollary \ref{cor:bounded-metric-away-from-critical} shows that $g_{t-t'}(y_0) \cdot y_0$ and $z_t$ are both bounded away from $x$ in the $L_k^2$ norm. As a consequence, $|t-t'|$ is uniformly bounded in the same way as Lemma \ref{lem:bounded-away-from-critical}. Therefore
\begin{equation}\label{eqn:bounded-linear-flow}
\| e^{i \beta t} \cdot g_t(y_0) \cdot y_0 - x \|_{L_k^2} = \| e^{i \beta (t-t')} \cdot z_t - x \|_{L_k^2} \leq C' \| z_t - x \|_{L_k^2} = C' \| y_0 - x \|_{L_k^2}
\end{equation}
for some constant $C'$, which implies that there is a subsequence of $e^{i \beta t} \cdot g_t(y_0) \cdot y_0$ converging strongly to a limit $z_\infty^0$ in $L_{k-1}^2$. Since this is true for all $k$, then $z_\infty^0$ is a $C^\infty$ Higgs pair. 

A special case of \eqref{eqn:bounded-linear-flow} is
\begin{equation}\label{eqn:bounded-linear-flow-C1}
\| e^{i \beta t} \cdot g_t(y_0) \cdot y_0 - x \|_{C^1} \leq C \| e^{i \beta t} \cdot g_t(y_0) \cdot y_0 - x \|_{L_k^2} \leq C' \| y_0 - x \|_{L_k^2}
\end{equation}
for any $k$ such that $L_k^2 \hookrightarrow C^1$ is an embedding. 

By modifying the method of Proposition \ref{prop:metrics-converge} we can now show that the change of metric converges in $C^0$. For $t \in (-\infty, 0]$, define $f_t(y_0) = e^{i \beta t} \cdot g_{t}(y_0)$ and let $t_1 \leq t_2 \leq T < 0$. This is summarised in the diagram below.

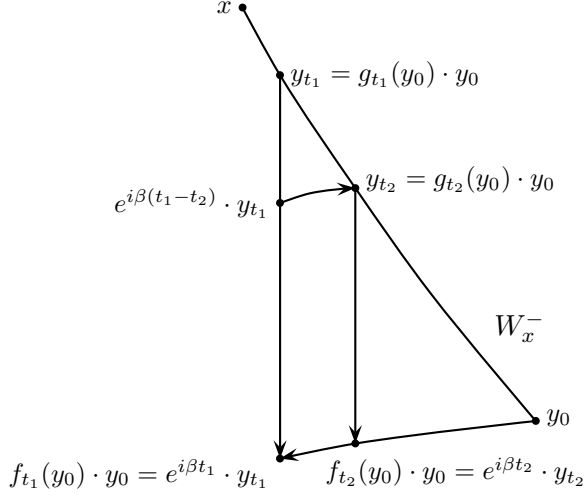
\begin{figure}[ht]
\begin{center}
\begin{pspicture}(2,-1)(10,6)
%\psline[linestyle=dashed](3.5,6)(3.5,-1)
\psline[arrowsize=5pt]{->}(5,3.6)(5,0.2)
\psline[arrowsize=5pt]{->}(4,5.1)(4,0)
\pscurve(3.5,6)(4,5.1)(5,3.6)(6,2.2)(7.4,0.5)
\pscurve[arrowsize=5pt]{<-}(5,3.6)(4.4,3.53)(4,3.4)
\pscurve[arrowsize=5pt]{->}(7.4,0.5)(5,0.2)(4,0)
%\pscurve[arrowsize=5pt]{->}(4,3.5)(4.6,1.7)(5.3,0.09)
\psdots[dotsize=3pt](3.5,6)(4,5.1)(4,3.4)(4,0)(7.4,0.5)(5,3.6)(5,0.2)
\uput{4pt}[180](3.5,6){\small $x$}
\uput{4pt}[20](4,5){\small $y_{t_1} = g_{t_1}(y_0) \cdot y_0$}
\uput{4pt}[180](4,-0.2){\small $f_{t_1}(y_0) \cdot y_0 = e^{i \beta t_1} \cdot y_{t_1}$}
\uput{4pt}[180](4,3.4){\small $e^{i \beta (t_1-t_2)} \cdot y_{t_1}$}
\uput{4pt}[10](7.4,0.5){\small $y_0$}
\uput{5pt}[305](5,0.2){\small{$f_{t_2}(y_0) \cdot y_0 = e^{i\beta t_2} \cdot y_{t_2}$}}
\uput{4pt}[20](5,3.6){\small{$y_{t_2} = g_{t_2}(y_0) \cdot y_0$}}
%\uput{2pt}[180](3.5,-0.8){$S_x^-$}
\uput{2pt}[45](6.8,1.5){$W_x^-$}
\end{pspicture}
\end{center}
\end{figure}

\begin{proposition}\label{prop:reverse-metrics-converge}
$h_t(y_0)$ converges in the $C^0$ norm to a unique limit $h_\infty(y_0) \in \mathcal{G}^\C / \mathcal{G}$ as $t \rightarrow -\infty$. The limit depends continuously on the initial condition $y_0 \in S_x^-$. The rate of convergence is given by
\begin{equation}\label{eqn:reverse-metric-convergence-rate}
\sup_X \sigma(h_t(y_0) (h_\infty(y_0))^{-1}) \leq C_2 e^{2 \eta t} \| y_0 - x \|_{L_k^2} 
\end{equation}
where $C_2 > 0$ is a constant depending only on the orbit $\mathcal{G} \cdot x$, the constant $\eta$ is from Proposition \ref{prop:exponential-convergence} and $k$ is a positive integer chosen so that $L_k^2 \hookrightarrow C^1$ is a continuous embedding.
\end{proposition}

\begin{proof}
The result follows from the same procedure as the proof of Proposition \ref{prop:metrics-converge}, except now we use the estimate from Lemma \ref{lem:reverse-uniform-bound-sigma} instead of the estimate from Lemma \ref{lem:uniform-bound-sigma} and the distance-decreasing formula for the modified flow from Lemma \ref{lem:modified-distance-decreasing}.

Let $h_{t_1-t_2}(y_{t_2})$ be the change of metric connecting $y_{t_2} = g_{t_2}(y_0) \cdot y_0$ and $e^{i \beta (t_1-t_2)} \cdot y_{t_1}$. Lemma \ref{lem:reverse-uniform-bound-sigma} and the estimate \eqref{eqn:bounded-linear-flow-C1} above show that $h_{t_1-t_2}(y_{t_2})$ satisfies
\begin{align*}
\sup_X \sigma(h_{t_1-t_2}(y_{t_2})) & \leq C \left( \| e^{i \beta (t_1-t_2)} \cdot y_{t_1} - x \|_{C^1} + \| y_{t_2} - x \|_{C^1} \right)  \\
 & \leq C C' \| y_{t_2} - x \|_{L_k^2} + C \| y_{t_2} - x \|_{C^1} \\
 & \leq C'' \| y_T - x \|_{L_k^2} \leq C_2 e^{2 \eta T} \| y_0 - x \|_{L_k^2}
\end{align*}
By the construction of the modified flow, the gauge transformation connecting $y_{t_2}$ and $e^{i \beta (t_1-t_2)} \cdot y_{t_1}$ is in $\mathcal{G}_*^\C$, The distance-decreasing formula for the action of $e^{i \beta (t_1 - t_2)}$ from Lemma \ref{lem:modified-distance-decreasing} then implies that
\begin{equation*}
\sigma(h_{t_1}(y_0) h_{t_2}(y_0)^{-1}) \leq \sigma(h_{t_1-t_2}(y_{t_2}))
\end{equation*}
and so the sequence $h_t(y_0)$ is Cauchy in the $C^0$ norm, by the same proof as Proposition \ref{prop:metrics-converge}.
\end{proof}

Therefore $y_0$ is connected to $z_\infty^0$ by a $C^0$ gauge transformation. Elliptic regularity together with the fact that $z_\infty^0$ is a $C^\infty$ Higgs pair then shows that $y_0$ is gauge equivalent to $z_\infty^0$ by a $C^\infty$ gauge transformation.

The same method as the proof of Proposition \ref{prop:convergence-group-action} then allows us to explicitly construct a solution of the linearised flow $z_\infty^{-s} = e^{i\beta s} \cdot z_\infty^0$ converging to $x$ as $s \rightarrow +\infty$. Lemma \ref{lem:classify-neg-slice} then shows that $z_\infty^0$ is $\mathcal{G}^\C$ equivalent to a point in $S_x^-$, which is smooth by Lemma \ref{lem:slice-smooth}. 

Therefore $y_0$ is $\mathcal{G}^\C$ equivalent to a point in $S_x^-$, and so we have proved the following converse to Proposition \ref{prop:convergence-group-action}.

\begin{proposition}\label{prop:unstable-maps-to-slice}
For each $y_0 \in W_x^-$ there exists a $C^\infty$ gauge transformation $g \in \mathcal{G}^\C$ such that $g \cdot y_0 \in S_x^-$.
\end{proposition}

\subsection{An algebraic criterion for the existence of flow lines}\label{sec:filtration-criterion}

The results of the previous two sections combine to give the following theorem.
\begin{theorem}\label{thm:algebraic-flow-line}
Let $E$ be a complex vector bundle over a compact Riemann surface $X$, and let $(\bar{\partial}_A, \phi)$ be a Higgs bundle on $E$. Suppose that $E$ admits a filtration $(E^{(1)}, \phi^{(1)}) \subset \cdots \subset (E^{(n)}, \phi^{(n)}) = (E, \phi)$ by Higgs subbundles such that the quotients $(Q_k, \phi_k) := (E^{(k)}, \phi^{(k)}) / (E^{(k-1)}, \phi^{(k-1)})$ are Higgs polystable and $\slope(Q_k) < \slope(Q_j)$ for all $k < j$. Then there exists $g \in \mathcal{G}^\C$ and a solution to the reverse Yang-Mills-Higgs heat flow equation with initial condition $g \cdot (\bar{\partial}_A, \phi)$ which converges to a critical point isomorphic to $(Q_1, \phi_1) \oplus \cdots \oplus (Q_n, \phi_n)$.

Conversely, if there exists a solution of the reverse heat flow from the initial condition $(\bar{\partial}_A, \phi)$ converging to a critical point $(Q_1, \phi_1) \oplus \cdots \oplus (Q_n, \phi_n)$ then $(\bar{\partial}_A, \phi)$ admits a filtration $(E^{(1)}, \phi^{(1)}) \subset \cdots \subset (E^{(n)}, \phi^{(n)}) = (E, \phi)$ whose graded object is isomorphic to $(Q_1, \phi_1) \oplus \cdots \oplus (Q_n, \phi_n)$.
\end{theorem}

\begin{proof}
Suppose first that $(\bar{\partial}_A, \phi)$ admits a filtration $(E^{(1)}, \phi^{(1)}) \subset \cdots \subset (E^{(n)}, \phi^{(n)}) = (E, \phi)$ by Higgs subbundles such that the quotients $(Q_k, \phi_k) := (E^{(k)}, \phi^{(k)}) / (E^{(k-1)}, \phi^{(k-1)})$ are Higgs polystable and $\slope(Q_k) < \slope(Q_j)$ for all $k < j$. Let $x$ be a critical point isomorphic to $(Q_1, \phi_1) \oplus \cdots \oplus (Q_n, \phi_n)$, and let $U$ be the neighbourhood of $x$ from Lemma \ref{lem:classify-neg-slice}. Then by applying the isomorphism $x \cong (Q_1, \phi_1) \oplus \cdots \oplus (Q_n, \phi_n)$ and scaling the extension classes there exists a complex gauge transformation such that $g \cdot (\bar{\partial}_A, \phi)$ is in $U$. Applying Lemma \ref{lem:classify-neg-slice} shows that $(\bar{\partial}_A, \phi)$ is isomorphic to a point in $S_x^-$, and therefore Proposition \ref{prop:convergence-group-action} shows that $(\bar{\partial}_A, \phi)$ is isomorphic to a point in $W_x^-$.

Conversely, if $x = (Q_1, \phi_1) \oplus \cdots \oplus (Q_n, \phi_n)$ is a critical point and $(\bar{\partial}_A, \phi) \in W_x^-$, then Proposition \ref{prop:unstable-maps-to-slice} shows that there exists $g \in \mathcal{G}^\C$ such that $g \cdot (\bar{\partial}_A, \phi) \in S_x^-$. Therefore Lemma \ref{lem:classify-neg-slice} shows that $(\bar{\partial}_A, \phi)$ admits a filtration whose graded object is isomorphic to $(Q_1, \phi_1) \oplus \cdots \oplus (Q_n, \phi_n)$.
\end{proof}

\section{The Hecke correspondence via Yang-Mills-Higgs flow lines}\label{sec:hecke}

Let $(E, \phi)$ be a polystable Higgs bundle of rank $r$ and degree $d$, and let $(L_u, \phi_u)$ be a Higgs line bundle with $\deg L_u < \slope E$. Let $F$ be a smooth complex vector bundle $C^\infty$ isomorphic to $E \oplus L_u$ and choose a metric on $F$ such that the Higgs structure on $(E, \phi) \oplus (L_u,\phi_u)$ is a Yang-Mills-Higgs critical point in the space $\mathcal{B}(F)$ of Higgs bundles on $F$. The goal of this section is to show that Hecke modifications of the Higgs bundle $(E, \phi)$ correspond to Yang-Mills-Higgs flow lines in $\mathcal{B}(F)$ connecting the critical point $(E, \phi) \oplus (L_u, \phi_u)$ to lower critical points. 

In Section \ref{sec:Higgs-hecke-review} we review Hecke modifications of Higgs bundles. Section \ref{sec:canonical-map} describes how the space of Hecke modifications relates to the geometry of the negative slice and Section \ref{sec:YMH-flow-hecke} contains the proof of Theorem \ref{thm:flow-hecke} which shows that Hecke modifications correspond to $\YMH$ flow lines. In Section \ref{sec:secant-criterion} we give a geometric criterion for points to be connected by unbroken flow lines in terms of the secant varieties of the space of Hecke modifications inside the negative slice. In particular, this gives a complete classification of the $\YMH$ flow lines for rank $2$ (cf. Corollary \ref{cor:rank-2-classification}). Throughout this section the notation $\mathcal{E}$ is used to denote the sheaf of holomorphic sections of the bundle $E$.

\subsection{Hecke modifications of Higgs bundles}\label{sec:Higgs-hecke-review}

The purpose of this section is to derive some basic results for Hecke modifications of Higgs bundles which will be used in Section \ref{sec:YMH-flow-hecke} to prove Theorem \ref{thm:flow-hecke}. In Section \ref{sec:secant-criterion} we extend these results to study unbroken YMH flow lines.

First recall that a Hecke modification of a holomorphic bundle $E$ over a Riemann surface $X$ is determined by points $p_1, \ldots, p_n \in X$ (not necessarily distinct) and nonzero elements $v_j \in E_{p_j}^*$ for $j = 1, \ldots, n$. This data determines a sheaf homomorphism $\mathcal{E} \rightarrow \oplus_{j=1}^n \C_{p_j}$ to the skyscraper sheaf supported at $p_1, \ldots, p_n$ with kernel a locally free sheaf $\mathcal{E}'$. This determines a holomorphic bundle $E' \rightarrow X$ which we call the \emph{Hecke modification of $E$ determined by $v = (v_1, \ldots, v_n)$}.
\begin{equation*}
0 \rightarrow \mathcal{E}' \rightarrow \mathcal{E} \stackrel{v}{\longrightarrow} \bigoplus_{j=1}^n \C_{p_j} \rightarrow 0  
\end{equation*}
Since the kernel sheaf $\mathcal{E}'$ only depends on the equivalence class of each $v_j$ in $\mathbb{P} E_{p_j}^*$ then from now on we abuse the notation slightly and also use $v_j \in \mathbb{P} E_{p_j}^*$ to denote the equivalence class of $v_j \in E_{p_j}^*$. 

As explained in \cite[Sec. 4.5]{witten-hecke}, if $(E, \phi)$ is a Higgs bundle, then a Hecke modification of $(E, \phi)$ may introduce poles into the Higgs field and so there are restrictions on the allowable modifications which preserve holomorphicity of the Higgs field. 

\begin{definition}
Let $(E, \phi)$ be a Higgs bundle. A Hecke modification $E'$ of $E$ is \emph{compatible} with $\phi$ if the induced Higgs field on $E'$ is holomorphic.
\end{definition}

The next result describes a basic condition for the modification to be compatible with the Higgs field.

\begin{lemma}
Let $(E, \phi)$ be a Higgs bundle, and $0 \rightarrow \mathcal{E}' \rightarrow \mathcal{E} \stackrel{v}{\longrightarrow} \mathbb{C}_p \rightarrow 0$ a Hecke modification of $E$ induced by $v \in E_p^*$. Then the induced Higgs field $\phi'$ on $E'$ is holomorphic if and only if there exists an eigenvalue $\mu$ of $\phi(p)$ such that the composition $\mathcal{E} \otimes K^{-1} \stackrel{\phi - \mu \cdot \id}{\longlonglongrightarrow} \mathcal{E} \stackrel{v}{\rightarrow} \C_p$ is zero.
\end{lemma}

\begin{proof}
Let $\phi \in H^0(\End(E) \otimes K)$. Then $\phi$ pulls back to a holomorphic Higgs field $\phi' \in H^0(\End(E') \otimes K)$ if and only if for any open set $U \subset X$ and any section $s \in \mathcal{E}(U)$, the condition $s \in \ker (\mathcal{E}(U) \stackrel{v}{\rightarrow} \C_p(U))$ implies that $\phi(s) \in \ker ((\mathcal{E} \otimes K)(U) \stackrel{v}{\rightarrow} \C_p(U))$. After choosing a trivialisation of $K$ in a neighbourhood of $p$, we can decompose the Higgs field $\phi(p)$ on the fibre $E_p$ as follows
\begin{equation}\label{eqn:fibre-extension}
\xymatrix{
0 \ar[r] & \ker v  \ar[r] \ar[d]^{\left. \phi(p) \right|_{\ker v}} & E_p  \ar[r] \ar[d]^{\phi(p)} & \mathbb{C}_p  \ar[r] \ar[d]^\mu & 0 \\
0 \ar[r] & \ker v  \ar[r] & E_p  \ar[r] & \mathbb{C}_p  \ar[r] & 0
}
\end{equation}
where scalar multiplication by $\mu$ is induced from the action of $\phi(p)$ on the quotient $\mathbb{C}_p = E_p / \ker v$. Therefore the endomorphism $\left( \phi(p) - \mu \cdot \id \right)$ maps $E_p$ into the subspace $\ker v$ and so $v \in E_p^*$ descends to a well-defined homomorphism $v' : \coker \left( \phi(p) - \mu \cdot \id \right) \rightarrow \C$.

Conversely, given an eigenvalue $\mu$ of $\phi(p)$ and an element $v' \in \coker(\phi(p) - \mu \cdot \id)^*$, one can choose a basis of $E_p$ and extend $v'$ to an element $v \in E_p^*$ such that $\im (\phi(p) - \mu \cdot \id) \subset \ker v$. Equivalently, $\phi(p)$ preserves $\ker v$ and so $v \in E_p^*$ defines a Hecke modification $E'$ of $E$ such that the induced Higgs field on $E'$ is holomorphic. 
\end{proof}

\begin{corollary}\label{cor:Higgs-compatible}
Let $(E, \phi)$ be a Higgs bundle and let $0 \rightarrow \mathcal{E}' \rightarrow \mathcal{E} \stackrel{v}{\rightarrow} \C_p \rightarrow 0$ be a Hecke modification of $E$ induced by $v \in \mathbb{P} E_p^*$. The following conditions are equivalent
\begin{enumerate}
\item The induced Higgs field $\phi'$ on $E'$ is holomorphic.

\item There exists an eigenvalue $\mu$ of $\phi(p)$ such that $v(\phi(s)) = \mu v(s)$ for all sections $s$ of $E$.

\item There exists an eigenvalue $\mu$ of $\phi(p)$ such that $v$ descends to a well-defined $v' \in (\coker (\phi(p) - \mu \cdot \id))^*$.

\end{enumerate}
\end{corollary}

\begin{lemma}\label{lem:resolve-higgs-subsheaf}
Let $(E, \phi)$ be a Higgs bundle and $(G, \varphi)$ a Higgs subsheaf. Then there exists a Higgs subbundle $(G', \varphi') \subset (E, \phi)$ such that $\rank(G) = \rank (G')$ and $(G, \varphi)$ is a Higgs subsheaf of $(G', \varphi')$.
\end{lemma}

\begin{proof}
Since $\dim_\C X = 1$ then a standard procedure shows that there is a holomorphic subbundle $G' \subset E$ with $\rank(G) = \rank (G')$ and $G$ is a subsheaf of $G'$, and so it only remains to show that this is a Higgs subbundle. The reverse of the construction above shows that the Higgs field $\varphi$ preserving $G$ extends to a meromorphic Higgs field $\varphi'$ preserving $G'$, and since this is the restriction of a holomorphic Higgs field $\phi$ on $E$ to the holomorphic subbundle $G'$, then $\varphi'$ must be holomorphic on $G'$. Therefore $G'$ is $\phi$-invariant.
\end{proof}

\begin{definition}\label{def:m-n-stable}
A Higgs bundle $(E, \phi)$ is \emph{$(m,n)$-stable (resp. $(m,n)$-semistable)} if for every proper $\phi$-invariant holomorphic subbundle $F \subset E$ we have
\begin{equation*}
\frac{\deg F + m}{\rank F} < \frac{\deg E - n}{\rank E} \quad \text{(resp. $\leq$)} .
\end{equation*}
\end{definition}

If $(E, \phi)$ is $(0,n)$-semistable then any Hecke modification $0 \rightarrow (\mathcal{E}', \phi') \rightarrow (\mathcal{E}, \phi) \rightarrow \oplus_{j=1}^n \C_{p_j} \rightarrow 0$ is semistable.

\begin{definition}
Then \emph{space of admissible Hecke modifications} is the subset $\mathcal{N}_{\phi} \subset \mathbb{P} E^*$ corresponding to the Hecke modifications which are compatible with the Higgs field.
\end{definition}

\begin{remark}\label{rem:hecke-explicit}
\begin{enumerate}

\item If $\phi = 0$ then $\mathcal{N}_{0} = \mathbb{P} E^*$. If $E$ is $(0,1)$-stable then there is a well-defined map $\mathbb{P} E^* \rightarrow \mathbb{P} H^1(E^*)$. The construction of the next section generalises this to a map $\mathcal{N}_{\phi} \rightarrow \mathbb{P} \mathcal{H}^1(E^*)$ (cf. Remark \ref{rem:canonical-map}).

\item Note that the construction above is the reverse of that described in \cite{witten-hecke}, which begins with $E'$ and modifies the bundle to produce a bundle $E$ with $\deg E = \deg E' + 1$. Here we begin with $E$ and construct $E'$ via a modification $0 \rightarrow \mathcal{E}' \rightarrow \mathcal{E} \rightarrow \C_p \rightarrow 0$ since we want to interpret the compatible modifications in terms of the geometry of the negative slice (see Section \ref{sec:canonical-map}) in order to draw a connection with the results on gradient flow lines for the Yang-Mills-Higgs flow functional from Section \ref{sec:filtration-criterion}.

\item One can also see the above construction more explicitly in local coordinates as in \cite{witten-hecke} by choosing a local frame $\{ s_1, \ldots, s_n \}$ for $E$ in a neighbourhood $U$ of $p$ with local coordinate $z$ centred at $p$ and for which the evaluation map $\mathcal{E} \stackrel{v}{\rightarrow} \C_p$ satisfies $v(s_1) = s_1(0)$ and $v(s_j) = 0$ for all $j=2, \ldots, n$. Then over $U \setminus \{ p \}$, the functions $\{ \frac{1}{z} s_1(z), s_2(z) \ldots, s_n(z) \}$ form a local frame for $E'$. Equivalently, the transition function $g = \left( \begin{matrix} \frac{1}{z} & 0 \\ 0 & \id \end{matrix} \right)$ maps the trivialisation for $E$ to a trivialisation for $E'$ (note that this is the inverse of the transition function from \cite[Sec. 4.5.2]{witten-hecke} for the reason explained in the previous paragraph). In this local frame field on $E$ we write $\phi(z) = \left( \begin{matrix} A(z) & B(z) \\ C(z) & D(z) \end{matrix} \right)$. The action on the Higgs field is then
\begin{equation*}
g \left( \begin{matrix} A(z) & B(z) \\ C(z) & D(z) \end{matrix} \right) g^{-1} = \left( \begin{matrix} A(z) & \frac{1}{z} B(z) \\ z C(z) & D(z) \end{matrix} \right)
\end{equation*}
Therefore the induced Higgs field on $E'$ will have a pole at $p$ unless $B(0) = 0$. The scalar $A(0)$ in this local picture is the same as the scalar $\mu$ from \eqref{eqn:fibre-extension}, and we see that 
\begin{equation*}
\phi(p) - \mu \cdot \id = \left( \begin{matrix} 0 & 0 \\ C(0) & D(0) - \mu \cdot \id \end{matrix} \right)
\end{equation*}
With respect to the basis of $E_p$ given by the choice of local frame, $v(\phi(p) - \mu \cdot \id) = 0$. Moreover, via this local frame $\coker ( \phi(p) - \mu \cdot \id)$ is identified with a subspace of $E_p$ which contains the linear span of $s_1(0)$. Therefore we see in the local coordinate picture that $v \in E_p^*$ descends to an element of $(\coker (\phi(p) - \mu \cdot \id))^*$. 
\end{enumerate}
\end{remark}

The next result shows that the admissible Hecke modifications have an interpretation in terms of the spectral curve associated to the Higgs field. This extends the results of \cite{witten-hecke} to include the possibility that $p$ is a branch point of the spectral cover. 

First recall Hitchin's construction of the spectral curve from \cite{Hitchin87-2}. Let $(E, \phi)$ be a Higgs pair. Then there is a projection map $\pi : K \rightarrow X$ and a bundle $\pi^* E$ over the total space of the canonical bundle together with a tautological section $\lambda$ of $\pi^* E$. The zero set of the characteristic polynomial of $\pi^* \phi$ defines a subvariety $S$ inside the total space of $K$. The projection $\pi$ restricts to a map $\pi : S \rightarrow X$, where for each $p \in X$ the fibre $\pi^{-1}(p)$ consists of the eigenvalues of the Higgs field $\phi(p)$. As explained in \cite{Hitchin87-2}, generically the discriminant of the Higgs field has simple zeros and in this case $S$ is a smooth curve called the \emph{spectral curve}. The induced projection $\pi : S \rightarrow X$ is then a ramified covering map with ramification divisor denoted $\mathcal{R} \subset S$. 

The pullback of the Higgs field to the spectral curve is a bundle homomorphism $\pi^* E \rightarrow \pi^*(E \otimes K)$, and the eigenspaces correspond to $\ker (\pi^* \phi - \lambda \cdot \id)$, where $\lambda$ is the tautological section defined above.  When the discriminant of the Higgs field has simple zeros then Hitchin shows in \cite{Hitchin87-2} that the eigenspaces form a line bundle $\mathcal{N} \rightarrow S$ and that the original bundle $E$ can be reconstructed as $\pi_* \mathcal{L}$, where the line bundle $\mathcal{L} \rightarrow S$ is formed by modifying $\mathcal{N}$ at the ramification points $0 \rightarrow \mathcal{N} \rightarrow \mathcal{L} \rightarrow \bigoplus_{p \in \mathcal{R}} \C_p \rightarrow 0$. One can reconstruct the Higgs field $\phi$ by pushing forward the endomorphism defined by the tautological section $\lambda : \mathcal{L} \rightarrow \mathcal{L} \otimes \pi^* K$.

\begin{lemma}
If the discriminant of $\phi$ has simple zeros then an admissible Hecke modification of $(E, \phi)$ corresponds to a Hecke modification of the line bundle $\mathcal{L}$ over the spectral curve.
\end{lemma}

\begin{proof}
Consider the pullback bundle $\pi^* E \rightarrow S$. The pullback of the Higgs field induces a sheaf homomorphism $(\pi^* \phi - \lambda \cdot \id) : \pi^* \mathcal{E} \otimes (\pi^* K)^{-1} \rightarrow \pi^* \mathcal{E}$. As explained in \cite[Sec. 2.6]{witten-hecke}, when the discriminant of $\phi$ has simple zeros then the cokernel of this homomorphism is the line bundle $\mathcal{L} \rightarrow S$ such that $\mathcal{E} \cong \pi_* \mathcal{L}$. 

For $\mu \in S$ such that $p = \pi(\mu)$, there is an isomorphism of the stalks of the skyscraper sheaves $\C_p \cong \pi_* (\C_\mu)$. Then a Hecke modification $\mathcal{L} \stackrel{v'}{\rightarrow} \C_\mu$ given by nonzero $v' \in \mathcal{L}_\mu^*$ induces a Hecke modification $v = v' \circ q \circ \pi^* : \mathcal{E} \rightarrow \C_p$, defined by the commutative diagram below.
\begin{equation*}
\xymatrix{
\pi^* \mathcal{E} \otimes (\pi^* K)^{-1} \ar[rr]^(0.6){\pi^*\phi - \lambda \cdot \id} & & \pi^* \mathcal{E} \ar[r]^(0.3)q & \coker(\pi^* \phi - \lambda \cdot \id) \ar[dr]^(0.6){v'} \ar[r] &  0 \\
& & \mathcal{E} \ar[u]^{\pi^*} \ar[rr]^v & & \C_p \ar[r] & 0
}
\end{equation*}
The definition of $v$ implies that for any open set $U \subset X$ with a trivialisation of $K$ in a neighbourhood of $p$, and all $s \in \mathcal{E}(U)$ we have 
\begin{equation*}
v(\phi s) = v' \circ q(\pi^* (\phi s)) = v' \circ q(\mu \, \pi^* (s)) = \mu \, v' \circ q \circ \pi^* (s) = \mu \, v(s)
\end{equation*}
and so $v$ is compatible with the Higgs field by Corollary \ref{cor:Higgs-compatible}.

Conversely, let $v \in E_p^*$ be compatible with the Higgs field $\phi$. Corollary \ref{cor:Higgs-compatible} shows that this induces a well-defined element of $\coker(\phi - \mu \cdot \id)^*$. Consider the endomorphisms $\phi(p) - \mu \cdot \id$ on the fibre of $E$ over $p \in X$ and $\pi^* \phi(\mu) - \mu \cdot \id$ on the fibre of $\pi^* E$ over $\mu \in S$.

\begin{equation*}
\xymatrix{
(\pi^* E \otimes \pi^* K^{-1})_\mu \ar[rr]^(0.6){\pi^* \phi - \mu \cdot \id} & & (\pi^* E)_\mu \ar[r] & \coker (\pi^* \phi - \mu \cdot \id)_\mu \ar[r] & 0\\
(E \otimes K^{-1})_p \ar[rr]^(0.6){\phi - \mu \cdot \id} \ar[u] & & E_p \ar[r] \ar[u] & \coker(\phi - \mu \cdot \id)_p \ar[r] \ar@{-->}[u] & 0 
}
\end{equation*}

 The universal property of cokernel defines a map $\coker(\phi - \mu \cdot \id)_p \rightarrow \coker(\pi^* \phi - \mu \cdot \id)_\mu$. Since the discriminant of the Higgs field has simple zeros then both fibres are one-dimensional and so this map becomes an isomorphism. Therefore $v$ induces a well-defined homomorphism on the fibre $\coker(\pi^* \phi - \mu \cdot \id)_\mu \rightarrow \C$, and hence a Hecke modification of $\mathcal{L}$ at $\mu \in S$. 
\end{proof}

\begin{remark}
When $p \in X$ is not a branch point of $\pi : S \rightarrow X$ then this result is contained in \cite{witten-hecke}.
\end{remark}

\begin{corollary}
If the discriminant of $\phi$ has simple zeros then the space of Hecke modifications is $\mathcal{N}_{\phi} = S$. 
\end{corollary}

\subsection{Secant varieties associated to the space of Hecke modifications}\label{sec:canonical-map}

The purpose of this section is to connect the geometry of the space of Hecke modifications with the geometry of the negative slice at a critical point in order to prepare for the proof of Theorem \ref{thm:flow-hecke} in the next section. 

Let $(E_1, \phi_1)$ and $(E_2, \phi_2)$ be Higgs bundles and let $\bar{\partial}_A$ denote the induced holomorphic structure on $E_1^* E_2$. Then there is an elliptic complex
\begin{equation*}
\Omega^0(E_1^* E_2) \stackrel{L_1}{\longrightarrow} \Omega^{0,1}(E_1^* E_2) \oplus \Omega^{1,0}(E_1^* E_2) \stackrel{L_2}{\longrightarrow} \Omega^{1,1}(E_1^* E_2) ,
\end{equation*}
where $L_1(u) = (\bar{\partial}_A u, \phi_2 u - u \phi_1)$ and $L_2(a, \varphi) = (\bar{\partial}_A \varphi + [a, \phi])$. Let $\mathcal{H}^0 = \ker L_1$, $\mathcal{H}^1 = \ker L_1^* \cap \ker L_2$ and $\mathcal{H}^2 = \ker L_2^*$ denote the spaces of harmonic forms. Recall that if $(E_1, \phi_1)$ and $(E_2, \phi_2)$ are both Higgs stable and $\slope(E_2) < \slope(E_1)$ then $\mathcal{H}^0(E_1^* E_2) = 0$.

Now consider the special case where $(E_1, \phi_1)$ is $(0, n)$-stable and $(E_2, \phi_2)$ is a Higgs line bundle. Let $\mathcal{B}$ denote the space of Higgs bundles on the smooth bundle $E_1 \oplus E_2$ and choose a metric such that $(E_1, \phi_1) \oplus (E_2, \phi_2)$ is a critical point of $\YMH : \mathcal{B} \rightarrow \R$. Definition \ref{def:slice} shows that $\mathcal{H}^1(E_1^* E_2) \cong S_x^-$ is the negative slice at this critical point.

Let $0 \rightarrow (\mathcal{E}', \phi') \rightarrow (\mathcal{E}_1, \phi_1) \rightarrow \oplus_{j=1}^n \C_{p_j} \rightarrow 0$ be a Hecke modification defined by $v_1, \ldots, v_n \in \mathbb{P} E_1^*$. Applying the functor $\Hom(\cdot, \mathcal{E}_2)$ to the short exact sequence $0 \rightarrow \mathcal{E}' \rightarrow \mathcal{E}_1 \rightarrow \oplus_j \C_{p_j} \rightarrow 0$ gives us an exact sequence of sheaves $0 \rightarrow \Hom(\mathcal{E}_1, \mathcal{E}_2) \rightarrow \Hom(\mathcal{E}', \mathcal{E}_2) \rightarrow \oplus_{j=1}^n \C_{p_j}^* \rightarrow 0$, where the final term comes from the isomorphism $\Ext^1(\oplus_j \C_{p_j}, \mathcal{E}_2) \cong \Hom(\mathcal{E}_2, \oplus_j \C_{p_j} \otimes K)^* \cong \oplus_j \C_{p_j}^*$. Note that this depends on a choice of trivialisations of $E_2$ and $K$, however the kernel of the map $\Hom(\mathcal{E}', \mathcal{E}_2) \rightarrow \oplus_j \C_{p_j}$ is independent of these choices. This gives us the following short exact sequence of Higgs sheaves
\begin{equation}\label{eqn:dual-short-exact}
0 \rightarrow \mathcal{E}_1^* \mathcal{E}_2 \rightarrow (\mathcal{E}')^* \mathcal{E}_2 \rightarrow \bigoplus_{j=1}^n \C_{p_j}^* \rightarrow 0
\end{equation}

There is an induced map $\Omega^0((E')^* E_2) \rightarrow \Omega^{1,0}((E')^* E_2)$ given by $s \mapsto \phi_2 s - s \phi'$. Recall from Corollary \ref{cor:Higgs-compatible} that there exists an eigenvalue $\mu_j$ for $\phi_1(p_j)$ such that $v(\phi_1(p_j) - \mu_j \cdot \id) = 0$ for each $j=1, \ldots, n$. From the above exact sequence there is an induced homomorphism $\Omega^{1,0}((E')^* E_2) \stackrel{ev^1}{\longrightarrow} \oplus_{j=1}^n \C_{p_j} \rightarrow 0$. The component of $ev^1(\phi_2 s - s \phi')$ in $\C_{p_j}$ is $(\phi_2(p_j) - \mu_j ) s$. In particular, $\phi_2 s - s \phi' \in \ker(ev^1)$ iff $\phi_2(p_j) = \mu_j$ for all $j=1, \ldots, n$.

\begin{definition}\label{def:Hecke-compatible}
Let $(E_1, \phi_1)$ be a Higgs bundle, and $(E_2, \phi_2)$ a Higgs line bundle. The \emph{space of Hecke modifications compatible with $\phi_1$ and $\phi_2$}, denoted $\mathcal{N}_{\phi_1, \phi_2} \subset \mathcal{N}_{\phi_1}$, is the set of Hecke modifications compatible with $\phi_1$ such that $ev^1(\phi_2 s - s \phi') = 0$ for all $s \in \Omega^0((E')^* E_2)$. 
\end{definition}

\begin{remark}\label{rem:miniscule-compatible}
Note that if $n = 1$ and $v \in \mathbb{P} E_1^*$ is a Hecke modification compatible with $\phi_1$, then the requirement that $v \in \mathcal{N}_{\phi_1, \phi_2}$ reduces to $\phi_2(p) = \mu$, where $\mu$ is the eigenvalue of $\phi_1(p)$ from Corollary \ref{cor:Higgs-compatible}. Such a $\phi_2 \in H^0(\End(E_2) \otimes K) = H^0(K)$ always exists since the canonical linear system is basepoint free and therefore $\bigcup_{\phi_2 \in H^0(K)} \mathcal{N}_{\phi_1, \phi_2} = \mathcal{N}_{\phi_1}$.  If $n > 1$ then $\phi_2$ with these properties may not exist for some choices of $\phi_1 \in H^0(\End(E_1) \otimes K)$ and $v_1, \ldots, v_n \in \mathbb{P} E_1^*$ (the existence of $\phi_2$ depends on the complex structure of the surface $X$). If $\phi_1 = 0$, then we can choose $\phi_2 = 0$ and in this case $\mathcal{N}_{\phi_1, \phi_2} = \mathcal{N}_{\phi_1} = \mathbb{P} E_1^*$ (this corresponds to the case of the Yang-Mills flow in Theorem \ref{thm:flow-hecke}).
\end{remark}

\begin{lemma}
Let $(E_1, \phi_1)$ be Higgs polystable and $(E_2, \phi_2)$ be a Higgs line bundle. Let $0 \rightarrow (\mathcal{E}', \phi') \rightarrow (\mathcal{E}, \phi) \rightarrow \oplus_{j=1}^n \C_{p_j} \rightarrow 0$ be a Hecke modification defined by distinct $v_1, \ldots, v_n \in \mathcal{N}_{\phi_1, \phi_2}$. 

Then there is an exact sequence
\begin{equation}\label{eqn:hyper-exact-sequence}
0 \rightarrow \mathcal{H}^0(E_1^* E_2) \rightarrow \mathcal{H}^0((E')^* E_2) \rightarrow \C^n \rightarrow \mathcal{H}^1(E_1^* E_2) \rightarrow \mathcal{H}^1((E')^* E_2)
\end{equation}
\end{lemma}

\begin{proof}

The short exact sequence \eqref{eqn:dual-short-exact} leads to the following commutative diagram of spaces of smooth sections
\begin{equation*}
\xymatrix@C=1.6em{
0 \ar[r] & \Omega^0(E_1^* E_2) \ar[r]^{i^*} \ar[d]^{L_1} & \Omega^0((E')^* E_2) \ar[r]^{ev^0} \ar[d]^{L_1} & \bigoplus_{j=1}^n \C_{p_j} \ar[r] & 0 \\
0 \ar[r] & \Omega^{0,1}(E_1^* E_2) \oplus \Omega^{1,0}(E_1^* E_2) \ar[r]^(0.47){i^*} & \Omega^{0,1}((E')^* E_2) \oplus \Omega^{1,0}((E')^* E_2) \ar[r]^(0.62){ev^1} & \bigoplus_{j=1}^n \C_{p_j} \oplus \C_{p_j} \ar[r] & 0 
}
\end{equation*}
Since $\bar{\partial}_A s$ depends on the germ of a section around a point, then there is no well-defined map $\bigoplus_{j=1}^n \C_{p_j} \rightarrow \bigoplus_{j=1}^n \C_{p_j} \oplus \C_{p_j}$ making the diagram commute, so the exact sequence \eqref{eqn:hyper-exact-sequence} does not follow immediately from the standard construction, and therefore we give an explicit construction below.

First construct a map $\C^n \rightarrow \mathcal{H}^1(E_1^* E_2)$ as follows. Given $z \in \C^n$, choose a smooth section $s' \in \Omega^0((E')^* E_2)$ such that $ev^0(s') = z$ and $ev^1(\bar{\partial}_A s') = 0$. Since $\phi_2(p_j) = \mu_j$, then $ev^1(\phi_2 s' - s' \phi') = 0$ and so $ev^1(L_1 s') = 0$. Therefore $(\bar{\partial}_A s', \phi_2 s' - s' \phi') = i^*(a, \varphi)$ for some $(a, \varphi) \in \Omega^{0,1}(E_1^* E_2) \oplus \Omega^{1,0}(E_1^* E_2)$. Let $[(a, \varphi)] \in \mathcal{H}^1(E_1^* E_2)$ denote the harmonic representative of $(a, \varphi)$. Define the map $\C^n \rightarrow \mathcal{H}^1(E_1^* E_2)$ by $z \mapsto [(a, \varphi)]$. 

To see that this is well-defined independent of the choice of $s' \in \Omega^0((E')^* E_2)$, note that if $s'' \in \Omega^0((E')^* E_2)$ is another section such that $ev^0(s'') = z$ and $ev^1(\bar{\partial}_A s'') = 0$, then $ev^0(s'' - s') = 0$, and so  $s'' - s' = i^*(s)$ for some $s \in \Omega^0(E_1^* E_2)$. Therefore $L_1(s'' - s') = i^* L_1(s)$ with $[L_1(s)] = 0 \in \mathcal{H}^1(E_1^* E_2)$, and so $s'$ and $s''$ determine the same harmonic representative in $\mathcal{H}^1(E_1^* E_2)$. 

To check exactness of \eqref{eqn:hyper-exact-sequence} at the term $\C^n$, note that if $z = ev^0(s')$ for some harmonic $s' \in \mathcal{H}^0((E')^* E_2)$, then $L_1(s') = 0 = i^* (0,0)$, and so $z \in \C^n$ maps to $0 \in \mathcal{H}^1(E_1^* E_2)$. Moreover, if $z$ maps to $0 \in \mathcal{H}^1(E_1^*E_2)$, then there exists $s' \in \Omega^0((E')^* E_2)$ such that $L_1(s') = i^*(a, \varphi)$ where $(a, \varphi) \in \Omega^{0,1}(E_1^* E_2) \oplus \Omega^{1,0}((E')^* E_2)$ and $(a, \varphi) = L_1(s)$ for some $s \in \Omega^0(E_1^* E_2)$. Therefore $s'$ and $i^* s$ differ by a harmonic section of $\mathcal{H}^0((E')^* E_2)$. Since $ev^0(i^* s) = 0$ then $z$ is the image of this harmonic section under the map $\mathcal{H}^0((E')^* E_2) \rightarrow \C^n$.

To check exactness at $\mathcal{H}^1(E_1^* E_2)$, given $z \in \C^n$ construct $(a, \varphi)$ as above and note that $i^*(a, \varphi) = L_1 s'$ for some $s' \in \Omega^0((E')^* E_2)$. Therefore $i^*[(a, \varphi)] = 0 \in \mathcal{H}^1((E')^* E_2)$ and so the image of $\C^n \rightarrow \mathcal{H}^1(E_1^* E_2)$ is contained in the kernel of $\mathcal{H}^1(E_1^* E_2) \rightarrow \mathcal{H}^1((E')^* E_2)$. Now suppose that the image of $[(a, \varphi)]$ is zero in $\mathcal{H}^1((E')^* E_2)$, i.e. $i^*(a, \varphi) = L_1 s'$ for some $s' \in \Omega^0((E')^* E_2)$. Let $z = ev^0(s')$. Note that $z = 0$ implies that $s' = i^* s$ for some $s \in \Omega^0(E_1^* E_2)$, and so $[(a, \varphi)] = 0$. If $z \neq 0$ then there exists $s'' \in \Omega^0((E')^* E_2)$ such that $ev^1(L_1(s'')) = 0$ and $ev^0(s'') = z$. Then $L_1(s'') = i^*(a'', \varphi'')$ for some $(a'', \varphi'') \in \Omega^{0,1}(E_1^* E_2) \oplus \Omega^{1,0}(E_1^* E_2)$. Moreover, $ev^0(s'' - s') = 0$, so $s'' - s' = i^* s$ for some $s \in \Omega^0(E_1^* E_2)$. Commutativity implies that $L_1 s = (a'', \varphi'') - (a, \varphi)$, and so the harmonic representatives $[(a, \varphi)]$ and $[(a'', \varphi'')]$ are equal. Therefore $[(a, \varphi)]$ is the image of $z$ by the map $\C^n \rightarrow \mathcal{H}^1(E_1^* E_2)$, which completes the proof of exactness at $\mathcal{H}^1(E_1^* E_2)$.

Exactness at the rest of the terms in the sequence \eqref{eqn:hyper-exact-sequence} then follows from standard methods.
\end{proof}

For any stable Higgs bundle $(E, \phi)$ with $d = \deg E$ and $r = \rank E$, define the \emph{generalised Segre invariant} by 
\begin{equation*}
s_k(E, \phi) := k d - r \left( \max_{F \subset E, \rank F = k} \deg F \right) .
\end{equation*}
where the maximum is taken over all $\phi$-invariant holomorphic subbundles of rank $k$. Note that $s_k(E, \phi) \geq s_k(E, 0) =: s_k(E)$ and
\begin{equation*}
\frac{1}{rk} s_k(E, \phi) = \min_{F \subset E, \rank F = k} \left( \slope(E) - \slope(F) \right)
\end{equation*}
Note that any Hecke modification $(E', \phi') \hookrightarrow (E, \phi)$ with $\deg E - \deg E' = n$ has Segre invariant $s_k(E', \phi') \geq s_k(E, \phi) - nk$. As a special case, $(E', \phi')$ is stable if $n < \frac{1}{k} s_k(E, \phi)$ for all $k = 1, \ldots, r-1$.

A theorem of Lange \cite[Satz 2.2]{Lange83} shows that a general stable holomorphic bundle $E$ satisfies $s_k(E) \geq k(r - k)(g-1)$ for all $k = 1, \ldots, r - 1$.  Since there is an dense open subset of stable Higgs bundles whose underlying holomorphic bundle is stable, then Lange's theorem also gives the same lower bound on the Segre invariant for a general stable Higgs bundle.

\begin{lemma}\label{lem:segre-bound}
Let $0 \rightarrow (E', \phi') \rightarrow (E, \phi) \rightarrow \oplus_{j=1}^n \C_{p_j} \rightarrow 0$ be a Hecke modification defined by distinct points $v_1, \ldots, v_n \in \mathbb{P} E^*$ such that $n < \frac{1}{k} s_k(E, \phi)$ for all $k = 1, \ldots, r-1$. Then $\slope(G) < \slope (E')$ for any proper non-zero Higgs subbundle $(G, \phi_G) \subset (E, \phi)$. In particular, this condition is satisfied if $(E, \phi)$ is a general stable Higgs bundle and $n < g-1$.
\end{lemma}

\begin{proof}
Let $k = \rank G$ and $h = \deg G$. Then the lower bound on the Segre invariant implies that
\begin{align*}
\slope(E') - \slope(G) = \frac{d - n}{r} - \frac{h}{k} & = \frac{1}{rk} \left(kd - kn - rh \right) \\
 & \geq \frac{1}{rk} \left( s_k(E, \phi) - kn \right) %\\
\end{align*}
Therefore if $n < \frac{1}{k} s_k(E, \phi)$ then $\slope(E') - \slope(G) > 0$ for any Higgs subbundle of rank $k$. If $n < g-1$ then \cite[Satz 2.2]{Lange83} shows that this condition is satisfied for general stable Higgs bundles.
\end{proof}

\begin{corollary}\label{cor:n-dim-kernel}
Let $(E_1, \phi_1)$ be a stable Higgs bundle, let $n < \frac{1}{k} s_k(E_1, \phi_1)$ for all $k=1, \ldots, \rank(E_1)-1$ and  let $(E_2, \phi_2)$ be a Higgs line bundle such that $\deg E_2 < \frac{\deg E_1 - n}{\rank E_1}$. Then given any set of $n$ distinct points $\{ v_1, \ldots, v_n \} \subset \mathcal{N}_{\phi_1, \phi_2}$ there is a well-defined $n$-dimensional subspace $\ker (\mathcal{H}^1(E_1^* E_2) \rightarrow \mathcal{H}^1((E')^* E_2))$.
\end{corollary}

\begin{proof}
Let $(E', \phi')$ be the Hecke modification of $(E_1, \phi_1)$ determined by $\{ v_1, \ldots v_n \} \subset \mathbb{P} E_1^*$. The lower bound on the Segre invariant implies that $(E', \phi')$ is Higgs stable, and therefore $\mathcal{H}^0((E')^* E_2) = 0$ since $\slope(E_2) < \slope(E') = \frac{\deg E_1 - n}{\rank E_1}$. The exact sequence \eqref{eqn:hyper-exact-sequence} then reduces to
\begin{equation*}
0 \rightarrow \C^n \rightarrow \mathcal{H}^1(E_1^* E_2) \rightarrow \mathcal{H}^1((E')^* E_2)
\end{equation*}
and so $\ker (\mathcal{H}^1(E_1^* E_2) \rightarrow \mathcal{H}^1((E')^* E_2))$ is a well-defined $n$-dimensional subspace of $\mathcal{H}^1(E_1^* E_2)$ associated to $\{ v_1, \ldots, v_n \}$.
\end{proof}

\begin{remark}\label{rem:canonical-map}
As noted above, the maps $\C^n \rightarrow \mathcal{H}^1(E_1^* E_2)$ depend on choosing trivialisations, but different choices lead to the same map up to a change of basis of $\C^n$, and so the subspace $\ker (\mathcal{H}^1((E')^* E_2) \rightarrow \mathcal{H}^1(E_1^* E_2))$ is independent of these choices.

In the special case where $n=1$, then this construction gives a well-defined map $\mathcal{N}_{\phi_1, \phi_2} \rightarrow \mathbb{P} \mathcal{H}^1(E_1^* E_2)$. When $n < \frac{1}{k} s_k(E_1, \phi_1)$ for all $k$, then Corollary \ref{cor:n-dim-kernel} shows that any $n$ distinct points $v_1, \ldots, v_n$ span a nondegenerate copy of $\mathbb{P}^{n-1}$ in $\mathbb{P} \mathcal{H}^1(E_1^* E_2)$.

In the special case where $\phi_1 = \phi_2 = 0$ and $E_2$ is trivial, then $\mathcal{N}_{\phi_1, \phi_2} = \mathbb{P} E^*$ and $\mathcal{H}^1(E_1^*) \cong H^{0,1}(E_1^*) \oplus H^{1,0}(E_1^*)$. Then the map $\mathbb{P} E^* \rightarrow \mathcal{H}^1(E_1^*) \rightarrow H^{0,1}(E_1^*) \cong H^0(E_1 \otimes K)^*$ is the usual map defined for holomorphic bundles (cf. \cite[p804]{HwangRamanan04}).
\end{remark}

\begin{definition}\label{def:secant-variety}
The \emph{$n^{th}$ secant variety}, denoted $\Sec^n(\mathcal{N}_{\phi_1, \phi_2}) \subset \mathbb{P} \mathcal{H}^1(E_1^* E_2)$, is the union of the subspaces $\vecspan \{ v_1, \ldots, v_n \} \subset \mathbb{P} \mathcal{H}^1(E_1^* E_2)$ taken over all $n$-tuples of distinct points $v_1, \ldots, v_n \in \mathcal{N}_{\phi_1, \phi_2}$. 
\end{definition}

The next lemma is a Higgs bundle version of \cite[Lemma 3.1]{NarasimhanRamanan69}. Since the proof is similar to that in \cite{NarasimhanRamanan69} then it is omitted.

\begin{lemma}\label{lem:nr-higgs}
Let $0 \rightarrow (E_2, \phi_2) \rightarrow (F, \tilde{\phi}) \rightarrow (E_1, \phi_1) \rightarrow 0$ be an extension of Higgs bundles defined by the extension class $[(a, \varphi)] \in \mathcal{H}^1(E_1^* E_2)$. Let $(E', \phi') \stackrel{i}{\longrightarrow} (E_1, \phi_1)$ be a Higgs subsheaf such that $i^*[(a, \varphi)] = 0 \in \mathcal{H}^1((E')^* E_2)$. Then $(E', \phi')$ is a Higgs subsheaf of $(F, \tilde{\phi})$. 
\end{lemma}

\begin{equation*}
\xymatrix{
 & & & (E', \phi') \ar[d]^i \ar@{-->}[dl] \\
0 \ar[r] & (E_2, \phi_2) \ar[r] & (F, \tilde{\phi}) \ar[r] & (E_1, \phi_1) \ar[r] & 0 \\
}\qedhere
\end{equation*}

\begin{corollary}\label{cor:linear-span}
Let $(E_1, \phi_1)$ be stable, let $n < \frac{1}{k} s_k(E_1, \phi_1)$ for all $k=1, \ldots, \rank(E_1)-1$, let $(E_2, \phi_2)$ be a Higgs line bundle and suppose that $\deg E_2 < \frac{\deg E_1 - n}{\rank E_1}$. Let $0 \rightarrow (E_2, \phi_2) \rightarrow (F, \tilde{\phi}) \rightarrow (E_1, \phi_1) \rightarrow 0$ be an extension of Higgs bundles with extension class $[(a, \varphi)] \in \mathcal{H}^1(E_1^* E_2)$. Let $0 \rightarrow (E', \phi') \stackrel{i}{\hookrightarrow} (E_1, \phi_1) \rightarrow \oplus_{j=1}^n \C_{p_j} \rightarrow 0$ be a Hecke modification determined by distinct points $\{v_1, \ldots, v_n\} \in \mathcal{N}_{\phi_1, \phi_2}$. 

Then $(E', \phi')$ is a subsheaf of $(F, \tilde{\phi})$ if $[(a, \varphi)] \in \vecspan\{ v_1, \ldots, v_n \} \subset \mathcal{H}^1(E_1^* E_2)$. 
\end{corollary}

\begin{proof}
If $[(a, \varphi)] \in \vecspan\{ v_1, \ldots, v_n \}$ then $[(a, \varphi)] \in \ker (\mathcal{H}^1(E_1^* E_2) \rightarrow \mathcal{H}^1((E')^* E_2))$ by Corollary \ref{cor:n-dim-kernel}, and therefore $(E', \phi')$ is a subsheaf of $(F, \tilde{\phi})$ by Lemma \ref{lem:nr-higgs}.  
\end{proof}

The next lemma gives a condition on the extension class $[(a, \varphi)] \in \mathcal{H}^1(E_1^* E_2)$ for $(E', \phi')$ to be the subsheaf of largest degree which lifts to a subsheaf of $(F, \tilde{\phi})$. This is used to study unbroken flow lines in Section \ref{sec:secant-criterion}.

\begin{lemma}\label{lem:nondegenerate-maximal}
Let $(E_1, \phi_1)$ be a stable Higgs bundle, choose $n$ such that $2n-1 < \frac{1}{k} s_k(E_1, \phi_1)$ for all $k=1, \ldots, \rank(E_1)$, let $(E_2, \phi_2)$ be a Higgs line bundle and suppose that $\deg E_2 < \frac{\deg E_1 - (2n-1)}{\rank E_1}$. Let $0 \rightarrow (E_2, \phi_2) \rightarrow (F, \tilde{\phi}) \rightarrow (E_1, \phi_1) \rightarrow 0$ be an extension of Higgs bundles with extension class $[(a, \varphi)] \in  \Sec^n(\mathcal{N}_{\phi_1, \phi_2}) \setminus \Sec^{n-1}(\mathcal{N}_{\phi_1, \phi_2}) \subset \mathbb{P} \mathcal{H}^1(E_1^* E_2)$ and let $0 \rightarrow (E', \phi') \stackrel{i}{\hookrightarrow} (E_1, \phi_1) \rightarrow \oplus_{j=1}^n \C_{p_j} \rightarrow 0$ be a Hecke modification determined by distinct points $v_1, \ldots, v_n \in \mathcal{N}_{\phi_1, \phi_2}$ such that $i^* [(a, \varphi)] = 0$.

Let $(\mathcal{E}'', \phi'') \stackrel{i''}{\hookrightarrow} (\mathcal{E}, \phi)$ be a subsheaf such that $(i'')^* [(a, \varphi)] = 0 \in \mathcal{H}^1((E'')^* E_2)$ and $\rank E'' = \rank E$. Then $\deg(E'') \leq \deg(E')$.
\end{lemma}

\begin{proof}
Let $\{ v_1'', \ldots, v_m'' \} \subset \mathcal{N}_{\phi_1, \phi_2}$ be the set of distinct points defining the Hecke modification $(\mathcal{E}'', \phi'') \stackrel{i''}{\hookrightarrow} (\mathcal{E}_1, \phi_1)$. Then $i^* [(a, \varphi)] = 0$ and $(i'')^*[(a, \varphi)] = 0$ together imply that $[(a, \varphi)] \in \vecspan\{ v_1, \ldots, v_n \} \cap \vecspan \{ v_1'', \ldots, v_m''\}$. Either $m + n > 2n-1$ (and so $\deg E'' \leq \deg E'$) or $m + n \leq 2n-1$ in which case Corollary \ref{cor:n-dim-kernel} together with the lower bound $2n-1 < \frac{1}{k} s_k(E_1, \phi_1)$ implies that $\vecspan \{ v_1, \ldots, v_n \} \cap \vecspan \{v_1'', \ldots, v_m'' \}$ is the linear span of $\{ v_1, \ldots, v_n \} \cap \{ v_1'', \ldots, v_m'' \}$. Since $m+n \leq 2n-1$ then $\{ v_1, \ldots, v_n \} \cap \{ v_1'', \ldots, v_m'' \}$ is a strict subset of $\{ v_1, \ldots, v_n\}$, which is not possible since $[(a, \varphi)] \notin \Sec^{n-1}(\mathcal{N}_{\phi_1, \phi_2})$. Therefore $\deg E'' \leq \deg E'$.
\end{proof}

\subsection{Constructing Hecke modifications of Higgs bundles via the Yang-Mills-Higgs flow.}\label{sec:YMH-flow-hecke}

Let $(E, \phi)$ be a stable Higgs bundle and $L_u$ a line bundle with $\deg L_u < \frac{\deg E - 1}{\rank E}$, and let $E'$ be a Hecke modification of $E$ which is compatible with the Higgs field
\begin{equation*}
0 \rightarrow (\mathcal{E}', \phi') \stackrel{i}{\hookrightarrow} (\mathcal{E}, \phi) \stackrel{v}{\rightarrow} \C_p \rightarrow 0 .
\end{equation*}

The goal of this section is to construct critical points $x_u = (L_u, \phi_u) \oplus (E, \phi)$ and $x_\ell = (L_\ell, \phi_\ell) \oplus (E', \phi')$ together with a broken flow line connecting $x_u$ and $x_\ell$. The result of Theorem \ref{thm:algebraic-flow-line} shows that this amounts to constructing a Higgs field $\phi_u \in H^0(K)$, a Higgs pair $(F, \tilde{\phi})$ in the unstable set of $x_u$ and a complex gauge transformation $g \in \mathcal{G}^\C$ such that $(E', \phi')$ is a Higgs subbundle of $g \cdot (F, \tilde{\phi})$. 

\begin{lemma}\label{lem:construct-Higgs-extension}
Let $0 \rightarrow (\mathcal{E}', \phi') \rightarrow (\mathcal{E}, \phi) \stackrel{v}{\rightarrow} \C_p \rightarrow 0$ be a Hecke modification such that $(E, \phi)$ and $(E', \phi')$ are both Higgs semistable, and let $L_u$ be a line bundle with $\deg L_u < \slope(E') < \slope(E)$. Then there exists a Higgs field $\phi_u \in H^0(K)$ and a non-trivial Higgs extension $(F, \tilde{\phi})$ of $(L_u, \phi_u)$ by $(E, \phi)$ such that $(E', \phi')$ is a Higgs subsheaf of $(F, \tilde{\phi})$.
\end{lemma}

\begin{proof}
By Remark \ref{rem:miniscule-compatible}, there exists $\phi_u \in H^0(K)$ such that $v \in \mathcal{N}_{\phi, \phi_u}$. Since $(E', \phi')$ is semistable with $\slope(E') > \slope(L_u)$ then $\mathcal{H}^0((E')^* L_u) = 0$ and so the exact sequence \eqref{eqn:hyper-exact-sequence} shows that the Hecke modification $v \in \mathbb{P} E^*$ determines a one-dimensional subspace of $\mathcal{H}^1(E^* L_u)$, and that any non-trivial extension class in this subspace is in the kernel of the map $\mathcal{H}^1(E^* L_u) \rightarrow \mathcal{H}^1((E')^* L_u)$. Let $0 \rightarrow (L_u, \phi_u) \rightarrow (F, \tilde{\phi}) \rightarrow (E, \phi) \rightarrow 0$ be such an extension. Then Lemma \ref{lem:nr-higgs} shows that $(E', \phi')$ is a Higgs subsheaf of $(F, \tilde{\phi})$.
\end{proof}

We can now use this result to relate Hecke modifications at a single point with $\YMH$ flow lines.

\begin{theorem}\label{thm:flow-hecke}

\begin{enumerate}

\item Let $0 \rightarrow (E', \phi') \rightarrow (E, \phi) \stackrel{v}{\rightarrow} \C_p \rightarrow 0$ be a Hecke modification such that $(E, \phi)$ is stable and $(E', \phi')$ is semistable, and let $L_u$ be a line bundle with $\deg L_u + 1 < \slope(E') < \slope(E)$. Then there exist sections $\phi_u, \phi_\ell \in H^0(K)$, a line bundle $L_\ell$ with $\deg L_\ell = \deg L_u + 1$ and a metric on $E \oplus L_u$ such that $x_u = (E, \phi) \oplus (L_u, \phi_u)$ and $x_\ell = (E_{gr}', \phi_{gr}') \oplus (L_\ell, \phi_\ell)$ are critical points connected by a $\YMH$ flow line, where $(E_{gr}', \phi_{gr}')$ is isomorphic to the graded object of the Seshadri filtration of $(E', \phi')$.

\item Let $x_u = (E, \phi) \oplus (L_u, \phi_u)$ and $x_\ell = (E', \phi') \oplus (L_\ell, \phi_\ell)$ be critical points connected by a $\YMH$ flow line such that $L_u, L_\ell$ are line bundles with $\deg L_u = \deg L_\ell + 1$, $(E, \phi)$ is stable and $(E', \phi')$ is polystable with $\deg L_u + 1 < \slope(E') < \slope(E)$. If $(E', \phi')$ is Higgs stable then it is a Hecke modification of $(E, \phi)$. If $(E', \phi')$ is Higgs polystable then it is the graded object of the Seshadri filtration of a Hecke modification of $(E, \phi)$.

\end{enumerate}

\end{theorem}

\begin{proof}[Proof of Theorem \ref{thm:flow-hecke}]
Given a Hecke modification $0 \rightarrow (\mathcal{E}', \phi') \rightarrow (\mathcal{E}, \phi) \rightarrow \C_p \rightarrow 0$ as in Lemma \ref{lem:construct-Higgs-extension}, choose $\phi_u \in H^0(K)$ such that $v \in \mathcal{N}_{\phi, \phi_u}$ and apply a gauge transformation to $E \oplus L_u$ such that $x_u = (E, \phi) \oplus (L_u, \phi_u)$ is a critical point of $\YMH$. The harmonic representative of the extension class $[(a, \varphi)] \in \mathcal{H}^1(E^* L_u)$ from Lemma \ref{lem:construct-Higgs-extension} defines an extension $0 \rightarrow (L_u, \phi_u) \rightarrow (F, \tilde{\phi}) \rightarrow (E, \phi) \rightarrow 0$ such that $y = (F, \tilde{\phi})$ is in the negative slice of $x_u$, and therefore flows down to a limit isomorphic to the graded object of the Harder-Narasimhan-Seshadri filtration of $(F, \tilde{\phi})$. 

Lemma \ref{lem:construct-Higgs-extension} also shows that $(E', \phi')$ is a Higgs subsheaf of $(F, \tilde{\phi})$. Lemma \ref{lem:resolve-higgs-subsheaf} shows that this has a resolution as a Higgs subbundle of $(F, \tilde{\phi})$, however since the Harder-Narasimhan type of $(F, \tilde{\phi})$ is strictly less than that of $(E, \phi) \oplus (L_u, \phi_u)$, $\rank(E') = \rank(F) - 1$ and $\deg E' = \deg E - 1$, then $(E', \phi')$ already has the maximal possible slope for a semistable Higgs subbundle of $(F, \tilde{\phi})$, and therefore $(E', \phi')$ must be the maximal semistable Higgs subbundle. Since $\rank(E') = \rank(F) - 1$, then the graded object of the Harder-Narasimhan-Seshadri filtration of $(F, \tilde{\phi})$ is $(E_{gr}', \phi_{gr}') \oplus (L_\ell, \phi_\ell)$, where $(L_\ell, \phi_\ell) = (F, \tilde{\phi}) / (E', \phi')$.   Theorem \ref{thm:algebraic-flow-line} then shows that $(E, \phi) \oplus (L_u, \phi_u)$ and $(E_{gr}', \phi_{gr}') \oplus (L_\ell, \phi_\ell)$ are connected by a flow line.

Conversely, if $x_u = (E, \phi) \oplus (L_u, \phi_u)$ and $x_\ell = (E', \phi') \oplus (L_\ell, \phi_\ell)$ are critical points connected by a flow line, then Theorem \ref{thm:algebraic-flow-line} shows that there exists a Higgs pair $(F, \tilde{\phi})$ in the negative slice of $x_u$ such that $(E', \phi')$ is the graded object of the Seshadri filtration of the maximal semistable Higgs subbundle of $(F, \tilde{\phi})$. If $(E', \phi')$ is Higgs stable, then since $\slope(E') > \slope (L_u)$ we see $(E', \phi')$ is a Higgs subsheaf of $(E, \phi)$ with $\rank(E) = \rank(E')$ and $\deg(E') = \deg(E) - 1$. Therefore $(E', \phi')$ is a Hecke modification of $(E, \phi)$. If $(E', \phi')$ is Higgs polystable then the same argument shows that $(E', \phi')$ is the graded object of the Seshadri filtration of a Hecke modification of $(E, \phi)$. 
\end{proof}

In general, for any flow one can define the space $\mathcal{F}_{\ell, u}$ of flow lines connecting upper and lower critical sets $C_u$ and $C_\ell$, and the space $\mathcal{P}_{\ell, u} \subset C_u \times C_\ell$ of pairs of critical points connected by a flow line. These spaces are equipped with projection maps to the critical sets defined by the canonical projection taking a flow line to its endpoints.
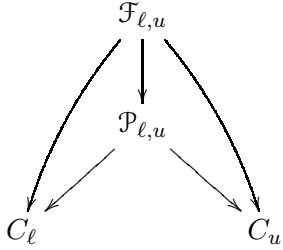
\begin{equation}
\xymatrix{
 & \mathcal{F}_{\ell, u} \ar[d] \ar@/_/[ddl] \ar@/^/[ddr] &  \\
 & \mathcal{P}_{\ell, u} \ar[dl] \ar[dr] &  \\
C_\ell & & C_u 
}
\end{equation}

For the Yang-Mills-Higgs flow, given critical sets $C_u$ and $C_\ell$ of respective Harder-Narasimhan types $(\frac{d}{r}, \deg L_u)$ and $(\frac{d-1}{r}, \deg L_u + 1)$ as in Theorem \ref{thm:flow-hecke} above, there are natural projection maps to the moduli space $C_u \rightarrow \mathcal{M}_{ss}^{Higgs}(r, d)$ and $C_\ell \rightarrow \mathcal{M}_{ss}^{Higgs}(r, d-1)$. Since the flow is $\mathcal{G}$-equivariant, then there is an induced correspondence variety $\mathcal{M}_{\ell, u} \subset \mathcal{M}_{ss}^{Higgs}(r, d-1) \times \mathcal{M}_{ss}^{Higgs}(r, d)$. 
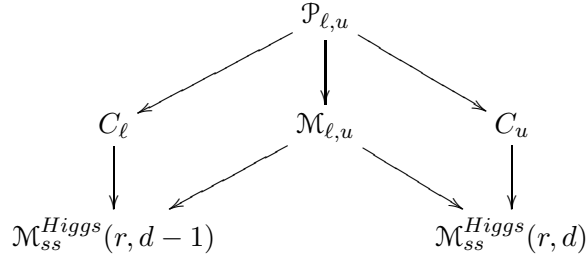
\begin{equation}\label{eqn:flow-hecke-diagram}
\xymatrix{
 & \mathcal{P}_{\ell, u} \ar[dl] \ar[dr] \ar[d] & \\
C_\ell \ar[d] & \mathcal{M}_{\ell,u} \ar[dl] \ar[dr] & C_u \ar[d] \\
\mathcal{M}_{ss}^{Higgs}(r, d-1) & & \mathcal{M}_{ss}^{Higgs}(r,d)
}
\end{equation}

Theorem \ref{thm:flow-hecke} shows that $\left( (E', \phi'), (E, \phi) \right) \in \mathcal{M}_{\ell, u}$ if and only if $(E', \phi')$ is a Hecke modification of $(E, \phi)$ and both Higgs pairs are semistable. If $r$ and $d$ are coprime then $\mathcal{M}_{ss}^{Higgs}(r,d)$ consists of stable Higgs pairs and so every Hecke modification of $(E, \phi)$ is semistable. Therefore we have proved
\begin{corollary}
$\mathcal{M}_{\ell,u}$ is the Hecke correspondence.
\end{corollary}

For Hecke modifications defined at multiple points (non-miniscule Hecke modifications in the terminology of \cite{witten-hecke}), we immediately have the following result.

\begin{corollary}\label{cor:broken-hecke}
Let $(E, \phi)$ be a $(0,n)$-stable Higgs bundle and consider a Hecke modification $0 \rightarrow (\mathcal{E}', \phi') \rightarrow (\mathcal{E}, \phi) \rightarrow \oplus_{j=1}^n \C_{p_n} \rightarrow 0$ defined by $n > 1$ distinct points $\{ v_1, \ldots, v_n \} \in \mathbb{P} E^*$. If there exists $\phi_u \in H^0(K)$ such that $v_1, \ldots, v_n \in \mathcal{N}_{\phi, \phi_u}$, then there is a broken flow line connecting $x_u = (E, \phi) \oplus (L_u, \phi_u)$ and $x_\ell = (E_{gr}', \phi_{gr}') \oplus (L_\ell, \phi_\ell)$, where $(E_{gr}', \phi_{gr}')$ is the graded object of the Seshadri filtration of the semistable Higgs bundle $(E', \phi')$.
\end{corollary}

\begin{proof}
Inductively apply Theorem \ref{thm:flow-hecke}.
\end{proof}

\subsection{A geometric criterion for unbroken $\YMH$ flow lines}\label{sec:secant-criterion}

Corollary \ref{cor:broken-hecke} gives a criterion for two $\YMH$ critical points $x_u = (E, \phi) \oplus (L_u, \phi_u)$ and $x_\ell = (E', \phi') \oplus (L_\ell, \phi_\ell)$ to be connected by a broken flow line. It is natural to ask whether they are also connected by an \emph{unbroken} flow line. The goal of this section is to answer this question by giving a geometric construction for points in the negative slice of $x_u$ which correspond to unbroken flow lines connecting $x_u$ and $x_\ell$ in terms of the secant varieties $\Sec^n(\mathcal{N}_{\phi, \phi_u})$. For holomorphic bundles, the connection between secant varieties and Hecke modifications has been studied in \cite{LangeNarasimhan83}, \cite{ChoeHitching10} and \cite{Hitching13}.

Given a $\YMH$ critical point $x_u = (E, \phi) \oplus (L_u, \phi_u)$ with $(E, \phi)$ stable and $\rank L_u =1$, consider an extension $0 \rightarrow (L_u, \phi_u) \rightarrow (F, \tilde{\phi}) \rightarrow (E, \phi) \rightarrow 0$ with extension class $[(a, \varphi)] \in \mathcal{H}^1(E^* L_u) = S_{x_u}^-$. Let $0 \rightarrow (E', \phi') \rightarrow (E, \phi) \rightarrow \oplus_{j=1}^n \C_{p_j} \rightarrow 0$ be a Hecke modification of $(E, \phi)$ as in the previous lemma, such that $\deg L_u < \slope(E')$. 

\begin{lemma}\label{lem:subbundle-slope-bound}
If $(G, \phi_G)$ is a semistable Higgs subbundle of $(F, \tilde{\phi})$ with $\slope(G) > \deg L_u$ and $\rank(G) < \rank (E)$, then there is a Higgs subbundle $(G', \phi_G') \subset (E, \phi)$ with $\slope(G') \geq \slope(G)$ and $\rank(G) = \rank(G')$.
\end{lemma}

\begin{proof}
If $(G, \phi_G)$ is a semistable Higgs subbundle of $(F, \tilde{\phi})$ with $\slope(G) > \deg L_u$, then $\mathcal{H}^0(G^* L_u) = 0$, and so $(G, \phi_G)$ is a Higgs subsheaf of $(E, \phi)$.
\begin{equation*}
\xymatrix{
 & & & (G, \phi_G) \ar[dl] \ar@{-->}[d] \\
0 \ar[r] & (L_u, \phi_u) \ar[r] & (F, \tilde{\phi}) \ar[r] & (E, \phi) \ar[r] & 0
}
\end{equation*}
Lemma \ref{lem:resolve-higgs-subsheaf} shows that the subsheaf $(G, \phi_G)$ can be resolved to form a Higgs subbundle $(G', \phi_G')$ of $(E, \phi)$ with $\slope(G') \geq \slope(G)$.
\end{proof}

\begin{theorem}\label{thm:unbroken-criterion}
Let $(E, \phi)$ be a stable Higgs bundle with Segre invariant $s_k(E, \phi)$ and choose $n$ such that $0 < 2n-1 < \min_{1 \leq k \leq r-1} \left( \frac{1}{k} s_k(E, \phi) \right)$. Let $0 \rightarrow (\mathcal{E}', \phi') \rightarrow (\mathcal{E}, \phi) \rightarrow \oplus_{j=1}^n \C_{p_j} \rightarrow 0$ be a Hecke modification of $(E, \phi)$ defined by distinct points $v_1, \ldots, v_n \in \mathbb{P} E^*$, and let $(L_u, \phi_u)$ be a Higgs line bundle such that $v_1, \ldots, v_n \in \mathcal{N}_{\phi, \phi_u}$. Choose a metric such that $x_u = (E, \phi) \oplus (L_u, \phi_u)$ is a $\YMH$ critical point.

Then any extension class $[(a, \varphi)] \in \vecspan \{ v_1, \ldots, v_n \} \cap \left( \Sec^n(\mathcal{N}_{\phi, \phi_u}) \setminus \Sec^{n-1}(\mathcal{N}_{\phi, \phi_u}) \right) \subset \mathbb{P} \mathcal{H}^1(E^* L_u)$ is isomorphic to an unbroken flow line connecting $x_u = (E, \phi) \oplus (L_u, \phi_u)$ and $x_\ell = (E', \phi') \oplus (L_\ell, \phi_\ell)$.
\end{theorem}

\begin{proof}
Let $(F, \tilde{\phi})$ be a Higgs bundle determined by the extension class $[(a, \varphi)] \in \mathbb{P} \mathcal{H}^1(E^* L_u)$. The choice of bundle is not unique, but the isomorphism class of $(F, \tilde{\phi})$ is unique. The proof reduces to showing that $(E', \phi')$ is the maximal semistable Higgs subbundle of $(F, \tilde{\phi})$. 

Since $[(a, \varphi)] \notin \Sec^{n-1}(\mathcal{N}_{\phi, \phi_u})$, then Lemma \ref{lem:nondegenerate-maximal} shows that $(E', \phi')$ is the subsheaf of $(E, \phi)$ with maximal degree among those that lift to a subsheaf of $(F, \tilde{\phi})$. Any semistable Higgs subbundle $(E'', \phi'')$ of $(F, \tilde{\phi})$ with $\rank(E'')  = \rank(E)$ either has $\slope(E'') \leq \deg L_u < \slope(E')$, or it is a subsheaf of $(E, \phi)$ and so must have $\slope(E'') \leq \slope(E')$.

The previous lemma shows that if $(G, \phi_G)$ is any semistable Higgs subbundle of $(F, \tilde{\phi})$ with $\slope(G) > \deg L_u$ and $\rank(G) < \rank(E)$,  then there is a Higgs subbundle $(G', \phi_G')$ of $(E, \phi)$ with $\slope(G') \geq \slope(G)$. The upper bound on $n = \deg E - \deg E'$ in terms of the Segre invariant then implies that $\slope(E') > \slope(G') \geq \slope(G)$ by Lemma \ref{lem:segre-bound}.

Therefore the subbundle $(\tilde{E}', \tilde{\phi}')$ resolving the subsheaf $(E', \phi') \subset (F, \tilde{\phi})$ is the maximal semistable Higgs subbundle of $(F, \tilde{\phi})$. Since $(\tilde{E}', \tilde{\phi}')$ is semistable and $\slope(\tilde{E}') \geq \slope(E') > \deg L_u$, then $\mathcal{H}^0((\tilde{E}')^* L_u) = 0$, and so $(\tilde{E}', \tilde{\phi}')$ is a Higgs subsheaf of $(E, \phi)$ that lifts to a subbundle of $(F, \tilde{\phi})$. Since $\deg E'$ is maximal among all such subsheaves, then we must have $(E', \phi') = (\tilde{E}', \tilde{\phi}')$ and so $(E', \phi')$ is the maximal semistable subbundle of $(F, \tilde{\phi})$. Therefore Theorem \ref{thm:algebraic-flow-line} shows that $x_u$ and $x_\ell$ are connected by an unbroken flow line.
\end{proof}

If $\rank(F) = 2$ (so that $E$ is a line bundle), then the condition on the Segre invariant $s_k(E, \phi)$ becomes vacuous. Moreover, $\mathbb{P} E^* \cong X$ and so Hecke modifications of $E$ are determined by a subset $\{ v_1, \ldots, v_n \} \subset X$. Therefore in the case $\rank(F) = 2$, we have a complete classification of the $\YMH$ flow lines on the space of Higgs bundles $\mathcal{B}(F)$.

\begin{corollary}\label{cor:rank-2-classification}
Let $F \rightarrow X$ be a $C^\infty$ Hermitian vector bundle with $\rank(F) = 2$. Let $x_u = (L_1^u, \phi_1^u) \oplus (L_2^u, \phi_2^u)$ and $x_\ell = (L_1^\ell, \phi_1^\ell) \oplus (L_2^\ell, \phi_2^\ell)$ be non-minimal critical points with $\YMH(x_u) > \YMH(x_\ell)$. Suppose without loss of generality that $\deg L_1^u > \deg L_1^\ell > \deg L_2^\ell > \deg L_2^u$. Let $n = \deg L_1^u - \deg L_1^\ell$. 

Then $x_u$ and $x_\ell$ are connected by a broken flow line if and only if there exists $\{ v_1, \ldots, v_n \} \in \mathcal{N}_{\phi_1^u, \phi_2^u}$ such that
\begin{align*}
0 \rightarrow (L_1^\ell, \phi_1^\ell) \rightarrow (L_1^u, \phi_1^u) \rightarrow \oplus_{j=1}^n \C_{p_j} \rightarrow 0 \\
0 \rightarrow (L_2^u, \phi_2^u) \rightarrow (L_2^\ell, \phi_2^\ell) \rightarrow \oplus_{j=1}^n \C_{p_j} \rightarrow 0
\end{align*}
are both Hecke modifications determined by $\{ v_1, \ldots, v_n \}$. They are connected by an unbroken flow line if the previous condition holds and $\{ v_1, \ldots, v_n \} \in \Sec^n(\mathcal{N}_{\phi_1^u, \phi_2^u}) \setminus \Sec^{n-1}(\mathcal{N}_{\phi_1^u, \phi_2^u})$. 
\end{corollary}

\appendix

\section{Uniqueness for the reverse Yang-Mills-Higgs flow}\label{sec:uniqueness}

The methods of Donaldson \cite{Donaldson85} and Simpson \cite{Simpson88} show that the Yang-Mills-Higgs flow resembles a nonlinear heat equation, and therefore the backwards flow is ill-posed. In Section \ref{sec:scattering-convergence} we prove existence of solutions to the backwards heat flow that converge to a critical point. To show that these solutions are well-defined we prove in this section that if a solution to the reverse $\YMH$ flow exists then it must be unique.

Using the Hermitian metric, let $d_A$ be the Chern connection associated to $\bar{\partial}_A$ and let $\psi = \phi + \phi^* \in \Omega^1(i \ad(E))$. The holomorphicity condition $\bar{\partial}_A \phi = 0$ becomes the pair of equations $d_A \psi = 0$, $d_A^* \psi = 0$ which also imply that $[F_A, \psi] = d_A^2 \psi = 0$, and the Yang-Mills-Higgs functional is $\| F_A + \psi \wedge \psi \|_{L^2}^2$. 

\begin{proposition}\label{prop:backwards-uniqueness}
Let $(d_{A_1}, \psi_1)(t)$, $(d_{A_2}, \psi_2)(t)$ be two solutions of the Yang-Mills-Higgs flow \eqref{eqn:YMH-flow-general} on a compact Riemann surface with respective initial conditions $(d_{A_1}, \psi_1)(0)$ and $(d_{A_2}, \psi_2)(0)$. If there exists a finite $T > 0$ such that $(d_{A_1}, \psi_1)(T) = (d_{A_2}, \psi_2)(T)$ then $(d_{A_1}, \psi_1)(t) = (d_{A_2}, \psi_2)(t)$ for all $t \in [0, T]$.
\end{proposition}

The result of Proposition \ref{prop:backwards-uniqueness} is valid when the base manifold is a compact Riemann surface, since we use the estimates of \cite[Sec. 3.2]{Wilkin08} to prove that the constant $C$ in Lemma \ref{lem:heat-inequalities} is uniform. In the case of the Yang-Mills flow on a compact K\"ahler manifold the estimates of Donaldson in \cite{Donaldson85} show that we can make this constant uniform on a finite time interval $[0,T]$ and so the result also applies in this setting. The setup described in the previous paragraph consisting of Higgs pairs $(d_A, \psi)$ satisfying $d_A \psi = 0$, $d_A^* \psi = 0$ is valid on any Riemannian manifold, and so the result of Proposition \ref{prop:backwards-uniqueness} will also apply to any class of solutions for which one can prove that the connection, Higgs field, the curvature and all of their derivatives are uniformly bounded on the given finite time interval $[0, T]$.

Let $\nabla_A$ denote the covariant derivative associated to the connection $d_A$. The complex connection associated to the pair $(d_A, \psi)$ is $D_{(A, \psi)} \eta = d_A \eta + [\psi, \eta]$ and the Laplacian is $\Delta_{(A, \psi)} \eta = D_{(A, \psi)}^* D_{(A, \psi)} \eta + D_{(A, \psi)} D_{(A, \psi)}^* \eta$ for any form $\eta \in \Omega^p(\End(E))$. The equation $d_A \psi = 0$ implies that the curvature of the complex connection is $D_{(A, \psi)} D_{(A, \psi)} \eta = [F_A + \psi \wedge \psi, \eta]$.

% The rough Laplacian at a point $p \in M$ is $\nabla_{(A, \psi)}^* \nabla_A = -\sum_{j=1}^n \left( \nabla_{A, e_j, e_j}^2 \right)$ with respect to an orthonormal basis $\{ e_1, \ldots, e_n \}$ of $T_p M$. The curvature operator associated to the pair $(d_A, \psi)$ is (cf. {\bf cite Bourguignon Lawson})
%\begin{align*}
%\mathcal{R}^{(A, \psi)}(\varphi)_X & = \sum_{j=1}^n \left[ (F_A + \psi \wedge \psi)_{e_j, X}, \varphi_{e_j} \right] \quad \text{for all $\varphi \in %\Omega^1(\End(E))$, $X \in T_p M$} \\
%\mathcal{R}^{(A, \psi)}(\varphi)_{X,Y} & = \sum_{j=1}^n \left( \left[ (F_A + \psi \wedge \psi)_{e_j, X}, \varphi_{e_j,Y} \right] - \left[ (F_A + \psi \wedge \psi)_{e_j, Y}, \varphi_{e_j, X} \right] \right) \\
% & \quad \quad \quad  \text{for all $\varphi \in \Omega^2(\End(E))$, $X,Y \in T_p M$} 
%\end{align*}

We have the following identities which will be useful in what follows. The notation $a \times b$ is used to denote various bilinear expressions with constant coefficients.

\begin{align}
0 & = d_A(F_A + \psi \wedge \psi) , \quad 0 = [\psi, F_A + \psi \wedge \psi]  \label{eqn:Higgs-Bianchi} \\
\Delta_{(A, \psi)} \eta & = \nabla_A^* \nabla_A \eta + (F_A + \psi \wedge \psi) \times \eta + R_M \times \eta + \psi \times \psi \times \eta + \nabla_A \psi \times \psi \times \eta \quad \label{eqn:Higgs-Weitzenbock} \\
0 & = D_{(A, \psi)}^* D_{(A, \psi)}^* (F_A + \psi \wedge \psi) \label{eqn:compose-adjoint}
\end{align}
The first identity follows from the Bianchi identity and the equation $d_A \psi = 0$. Equation \eqref{eqn:Higgs-Weitzenbock} is the Weitzenbock identity for a Higgs pair which follows from the usual identity for $\nabla_A$ (see for example \cite{BourguignonLawson81}) together with the fact that $(\psi \wedge \psi) \times \eta$ and the remaining terms in the Laplacian are of the form $\psi \times \psi \times \eta + \nabla_A \psi \times \psi \times \eta$. To see the identity \eqref{eqn:compose-adjoint}, take the inner product of the right hand side with an arbitrary $\eta \in \Omega^0(\End(E))$. We have (cf. \cite[(2.2)]{Rade92} for the case $\psi = 0$)
\begin{align*}
\left< D_{(A, \psi)}^* D_{(A, \psi)}^*  (F_A + \psi \wedge \psi), \eta  \right> & = \left< F_A + \psi \wedge \psi, D_{(A, \psi)} D_{(A, \psi)} \eta \right>  \\
 & = \left< F_A + \psi \wedge \psi, [F_A + \psi \wedge \psi, \eta] \right> = 0
\end{align*}

Consider the Yang-Mills-Higgs flow equations
\begin{equation}\label{eqn:YMH-flow-general}
\frac{\partial A}{\partial t} = - d_A^* (F_A + \psi \wedge \psi), \quad \frac{\partial \psi}{\partial t} = * [\psi, *(F_A + \psi \wedge \psi)]
\end{equation}
After using the metric to decompose $\Omega^1(\End(E)) \cong \Omega^1(\ad(E)) \oplus \Omega^1(i \ad(E))$, the flow equation can be written more compactly as
\begin{equation*}
\frac{\partial}{\partial t} (d_A + \psi) = - D_{(A, \psi)}^* (F_A + \psi \wedge \psi)
\end{equation*}
We then have
\begin{align*}
\frac{\partial}{\partial t} (F_A + \psi \wedge \psi) & = d_A \left( \frac{\partial A}{\partial t} \right) + \frac{\partial \psi}{\partial t} \wedge \psi + \psi \wedge \frac{\partial \psi}{\partial t} \\
 & = - d_A d_A^* (F_A + \psi \wedge \psi) + \left[ \psi, *[\psi, *(F_A + \psi \wedge \psi)] \right] \\
 & = - \Delta_{(A, \psi)} (F_A + \psi \wedge \psi) - d_A*[\psi, *(F_A + \psi \wedge \psi)] + [\psi, d_A^*(F_A + \psi \wedge \psi)]
\end{align*}
where in the last step we use the Bianchi identity \eqref{eqn:Higgs-Bianchi}. We also have
\begin{align*}
\frac{\partial}{\partial t} \left( d_A^* (F_A + \psi \wedge \psi) \right) & = -*\left[ \frac{\partial A}{\partial t}, *(F_A + \psi \wedge \psi) \right] + d_A^* \left( \frac{\partial}{\partial t} (F_A + \psi \wedge \psi) \right) \\
 & = * \left[ d_A^*(F_A + \psi \wedge \psi), F_A + \psi \wedge \psi \right] - d_A^* d_A d_A^* (F_A + \psi \wedge \psi) \\
 & \quad \quad + d_A^* [\psi, *[\psi, *(F_A + \psi \wedge \psi)]]
\end{align*}
and
\begin{align*}
\frac{\partial}{\partial t} \left( -*[\psi, *(F_A + \psi \wedge \psi)] \right) & = -* \left[ \frac{\partial \psi}{\partial t}, *(F_A + \psi \wedge \psi) \right] - *\left[ \psi, \frac{\partial}{\partial t} *(F_A + \psi \wedge \psi) \right] \\
 & = * \left[ - *[\psi, *(F_A + \psi \wedge \psi)], *(F_A + \psi \wedge \psi) \right] + * \left[ \psi, * d_A d_A^* (F_A + \psi \wedge \psi) \right] \\
 & \quad \quad - * \left[ \psi, *[\psi, *[\psi, *(F_A + \psi \wedge \psi)]] \right]
\end{align*}
Adding these two results gives us
\begin{align*}
\frac{\partial}{\partial t} \left( D_{(A, \psi)}^* (F_A + \psi \wedge \psi) \right) & = * \left[ D_{(A, \psi)}^* (F_A + \psi \wedge \psi), F_A + \psi \wedge \psi \right]  - D_{(A, \psi)}^* D_{(A, \psi)} D_{(A, \psi)}^* (F_A + \psi \wedge \psi) \\
 & = * \left[ D_{(A, \psi)}^* (F_A + \psi \wedge \psi), F_A + \psi \wedge \psi \right] - \Delta_{(A, \psi)} D_{(A, \psi)}^* (F_A + \psi \wedge \psi)
\end{align*}
where the last step uses \eqref{eqn:compose-adjoint}. Let $\mu_{(A, \psi)} = F_A + \psi \wedge \psi$ and $\nu_{(A, \psi)} = D_{(A, \psi)}^* (F_A + \psi \wedge \psi)$. The above equations become
\begin{align}
\left( \frac{\partial}{\partial t} + \Delta_{(A, \psi)} \right) \mu_{(A, \psi)} & =  -d_A*[\psi, *(F_A + \psi \wedge \psi)] + [\psi, d_A^*(F_A + \psi \wedge \psi)] \label{eqn:mu-evolution} \\
\left( \frac{\partial}{\partial t} + \Delta_{(A, \psi)} \right) \nu_{(A, \psi)} & = *[\nu_{(A, \psi)}, *\mu_{(A, \psi)}] \label{eqn:nu-evolution}
\end{align}

Now consider two solutions $(A_1, \psi_1)(t)$ and $(A_2, \psi_2)(t)$ to the Yang-Mills-Higgs flow equations \eqref{eqn:YMH-flow-general} on the time interval $[0,T]$ such that $(A_1, \psi_1)(T) = (A_2, \psi_2)(T)$. We will show below that this implies $(A_1, \psi_1)(0) = (A_2, \psi_2)(0)$.

Define $(a_t, \varphi_t) = (A_2, \psi_2)(t) - (A_1, \psi_1)(t)$, $m_t =  \mu_{(A_2, \psi_2)} - \mu_{(A_1, \psi_1)}$ and $n_t = \nu_{(A_2, \psi_2)} - \nu_{(A_1, \psi_1)}$. In terms of $(a_t, \varphi_t)$ we can write
\begin{align*}
m_t = \mu_{(A_2, \psi_2)} - \mu_{(A_1, \psi_1)} = d_{A_1} a_t + a_t \wedge a_t + [\psi, \varphi_t] + \varphi_t \wedge \varphi_t
\end{align*}
and for any $\eta \in \Omega^p(\End(E))$ the difference of the associated Laplacians has the form
\begin{equation}\label{eqn:laplacian-difference}
\left(\Delta_{(A_2, \psi_2)} - \Delta_{(A_1, \psi_1)} \right) \eta = \nabla_A a \times \eta + a \times \nabla_A \eta + a \times a \times \eta + \psi \times \varphi \times \eta + \varphi \times \varphi \times \eta
\end{equation}
where again $\omega_1 \times \omega_2$ is used to denote a bilinear expression in $\omega_1$ and $\omega_2$ with constant coefficients. By definition of $\nu_{(A, \psi)}$ as the gradient of the Yang-Mills-Higgs functional at $(d_A, \psi)$ we immediately have 
\begin{equation*}
\frac{\partial}{\partial t} (a_t + \varphi_t) = n_t , \quad \text{and} \quad \frac{\partial}{\partial t} (\nabla_A a_t + \nabla_A \varphi_t) = \left( \frac{\partial A}{\partial t} \times a_t, \frac{\partial A}{\partial t} \times \varphi_t \right) + \nabla_A n_t
\end{equation*}
Equation \eqref{eqn:mu-evolution} then becomes
\begin{align*}
\left( \frac{\partial}{\partial t} + \Delta_{(A_1, \psi_1)} \right) m_t & = - \left( \Delta_{(A_2, \psi_2)} - \Delta_{(A_1, \psi_1)} \right) \mu_{(A_2, \psi_2)} \\
 & \quad \quad + a_t \times \psi_1 \times (F_{A_1} + \psi_1 \wedge \psi_1) + \nabla_{A_1} \varphi_t \times (F_{A_1} + \psi_1 \wedge \psi_1) \\
 & \quad \quad + \nabla_{A_1} \psi_1 \times m_t + \psi_1 \times n_t
\end{align*}
and equation \eqref{eqn:nu-evolution} becomes
\begin{align*}
\left( \frac{\partial}{\partial t} + \Delta_{(A_1, \psi_1)} \right) n_t & = *[\nu_{(A_2, \psi_2)}, * \mu_{(A_2, \psi_2)}] - *[\nu_{(A_1, \psi_1)}, *\mu_{(A_1, \psi_1)}] - \left( \Delta_{(A_2, \psi_2)} - \Delta_{(A_1, \psi_1)} \right) \nu_{(A_2, \psi_2)} \\
 & = *[n_t, *\mu_{(A_2, \psi_2)}] +*[\nu_{(A_1, \psi_1)}, *m_t] - \left( \Delta_{(A_2, \psi_2)} - \Delta_{(A_1, \psi_1)} \right) \nu_{(A_2, \psi_2)} \\
\end{align*}
Using \eqref{eqn:laplacian-difference} and the Weitzenbock formula \eqref{eqn:Higgs-Weitzenbock}, we then have the following inequalities. In the case where $X$ is a compact Riemann surface then the estimates of \cite[Sec. 2.2]{Wilkin08} show that all of the derivatives of the connection, the Higgs field and the curvature $F_A$ are uniformly bounded along the flow and so the constant can be chosen uniformly on the interval $[0,T]$.
\begin{lemma}\label{lem:heat-inequalities}
For any pair of solutions $(d_{A_1}, \psi_1)(t)$ and $(d_{A_2}, \psi_2)(t)$ to the Yang-Mills-Higgs flow \eqref{eqn:YMH-flow-general} there exists a positive constant $C$ (possibly depending on $t$) such that the following inequalities hold 
\begin{align}
\left| \left( \frac{\partial}{\partial t} + \nabla_{A_1}^* \nabla_{A_1} \right) m_t \right| & \leq C \left( |a_t| + |\varphi_t| + | \nabla_{A_1} a_t| +| \nabla_{A_1} \varphi_t |+ | m_t | + |n_t| \right) \label{eqn:m-evolution} \\
\left| \left( \frac{\partial}{\partial t} + \nabla_{A_1}^* \nabla_{A_1} \right) n_t \right| & \leq C \left( |a_t| + |\varphi_t| + | \nabla_{A_1} a_t| +| \nabla_{A_1} \varphi_t |+ | m_t | + |n_t| \right)  \label{eqn:n-evolution} \\
\left| \frac{\partial}{\partial t} (a_t + \varphi_t) \right| & = | n_t | \label{eqn:a-evolution} \\
\left| \frac{\partial}{\partial t} (\nabla_A a_t + \nabla_A \varphi_t) \right| & \leq C \left( |a_t| + |\varphi_t| + | \nabla_A n_t | \right) \label{eqn:nabla-a-evolution}
\end{align}
Moreover, if $X$ is a compact Riemann surface then the constant $C$ can be chosen uniformly on any finite time interval $[0, T]$.
\end{lemma}

For simplicity of notation, in the following we use $\nabla := \nabla_{A_1}$ and $\square := \nabla_{A_1}^* \nabla_{A_1}$. Let $X := (m_t, n_t)$ and $Y := (a_t, \varphi_t, \nabla a_t, \nabla \varphi_t)$. The previous lemma implies that there exists a positive constant $C$ such that the following inequalities hold
\begin{align}\label{eqn:coupled-system}
\begin{split}
\left| \frac{\partial X}{\partial t} + \square X \right| & \leq C \left( | X | + | \nabla X| + | Y | \right) \\
\left| \frac{\partial Y}{\partial t} \right| & \leq C \left( |X| + |\nabla X| + |Y| \right)
\end{split}
\end{align}
A general result of Kotschwar in \cite[Thm 3]{kotschwar-uniqueness} shows that any system satisfying \eqref{eqn:coupled-system} on the time interval $[0,T]$ for which $X(T) = 0$, $Y(T)=0$, must also satisfy $X(t) = 0$, $Y(t) = 0$ for all $t \in [0, T]$. In the context of the Yang-Mills-Higgs flow \eqref{eqn:YMH-flow-general}, this gives us the proof of Proposition \ref{prop:backwards-uniqueness}.

%\bibliographystyle{plain}
%\bibliography{../ref}

\begin{thebibliography}{10}

\bibitem{AtiyahBott83}
M.~F. Atiyah and R.~Bott.
\newblock The {Y}ang-{M}ills equations over {R}iemann surfaces.
\newblock {\em Philos. Trans. Roy. Soc. London Ser. A}, 308(1505):523--615,
  1983.

\bibitem{BiswasWilkin10}
Indranil Biswas and Graeme Wilkin.
\newblock Morse theory for the space of {H}iggs {$G$}-bundles.
\newblock {\em Geom. Dedicata}, 149:189--203, 2010.

\bibitem{BourguignonLawson81}
Jean-Pierre Bourguignon and H.~Blaine Lawson, Jr.
\newblock Stability and isolation phenomena for {Y}ang-{M}ills fields.
\newblock {\em Comm. Math. Phys.}, 79(2):189--230, 1981.

\bibitem{ChoeHitching10}
Insong Choe and George~H. Hitching.
\newblock Secant varieties and {H}irschowitz bound on vector bundles over a
  curve.
\newblock {\em Manuscripta Math.}, 133(3-4):465--477, 2010.

\bibitem{Daskal92}
Georgios~D. Daskalopoulos.
\newblock The topology of the space of stable bundles on a compact {R}iemann
  surface.
\newblock {\em J. Differential Geom.}, 36(3):699--746, 1992.

\bibitem{DaskalWentworth04}
Georgios~D. Daskalopoulos and Richard~A. Wentworth.
\newblock Convergence properties of the {Y}ang-{M}ills flow on {K}\"ahler
  surfaces.
\newblock {\em J. Reine Angew. Math.}, 575:69--99, 2004.

\bibitem{Donaldson83}
S.~K. Donaldson.
\newblock A new proof of a theorem of {N}arasimhan and {S}eshadri.
\newblock {\em J. Differential Geom.}, 18(2):269--277, 1983.

\bibitem{Donaldson85}
S.~K. Donaldson.
\newblock Anti self-dual {Y}ang-{M}ills connections over complex algebraic
  surfaces and stable vector bundles.
\newblock {\em Proc. London Math. Soc. (3)}, 50(1):1--26, 1985.

\bibitem{Donaldson87-2}
S.~K. Donaldson.
\newblock Infinite determinants, stable bundles and curvature.
\newblock {\em Duke Math. J.}, 54(1):231--247, 1987.

\bibitem{Hitchin87}
N.~J. Hitchin.
\newblock The self-duality equations on a {R}iemann surface.
\newblock {\em Proc. London Math. Soc. (3)}, 55(1):59--126, 1987.

\bibitem{Hitchin87-2}
N.~J. Hitchin.
\newblock Stable bundles and integrable systems.
\newblock {\em Duke Math. J.}, 54(1):91--114, 1987.

\bibitem{Hitching13}
George~H. Hitching.
\newblock Geometry of vector bundle extensions and applications to a
  generalised theta divisor.
\newblock {\em Math. Scand.}, 112(1):61--77, 2013.

\bibitem{Hong01}
Min-Chun Hong.
\newblock Heat flow for the {Y}ang-{M}ills-{H}iggs field and the {H}ermitian
  {Y}ang-{M}ills-{H}iggs metric.
\newblock {\em Ann. Global Anal. Geom.}, 20(1):23--46, 2001.

\bibitem{Hubbard05}
John~H. Hubbard.
\newblock Parametrizing unstable and very unstable manifolds.
\newblock {\em Mosc. Math. J.}, 5(1):105--124, 2005.

\bibitem{HuybrechtsLehn97}
Daniel Huybrechts and Manfred Lehn.
\newblock {\em The geometry of moduli spaces of sheaves}.
\newblock Aspects of Mathematics, E31. Friedr. Vieweg \& Sohn, Braunschweig,
  1997.

\bibitem{HwangRamanan04}
Jun-Muk Hwang and S.~Ramanan.
\newblock Hecke curves and {H}itchin discriminant.
\newblock {\em Ann. Sci. \'Ecole Norm. Sup. (4)}, 37(5):801--817, 2004.

\bibitem{Jacob15}
Adam Jacob.
\newblock The limit of the {Y}ang-{M}ills flow on semi-stable bundles.
\newblock {\em J. Reine Angew. Math.}, 709:1--13, 2015.

\bibitem{JannerSwoboda15}
R{\'e}mi Janner and Jan Swoboda.
\newblock Elliptic {Y}ang-{M}ills flow theory.
\newblock {\em Math. Nachr.}, 288(8-9):935--967, 2015.

\bibitem{Jost05}
J{\"u}rgen Jost.
\newblock {\em Riemannian geometry and geometric analysis}.
\newblock Universitext. Springer-Verlag, Berlin, fourth edition, 2005.

\bibitem{Kirwan84}
Frances~Clare Kirwan.
\newblock {\em Cohomology of quotients in symplectic and algebraic geometry},
  volume~31 of {\em Mathematical Notes}.
\newblock Princeton University Press, Princeton, NJ, 1984.

\bibitem{Kobayashi87}
Shoshichi Kobayashi.
\newblock {\em Differential geometry of complex vector bundles}, volume~15 of
  {\em Publications of the Mathematical Society of Japan}.
\newblock Princeton University Press, Princeton, NJ, 1987.
\newblock , Kano Memorial Lectures, 5.

\bibitem{kotschwar-uniqueness}
Brett Kotschwar.
\newblock A short proof of backward uniqueness for some geometric evolution
  equations.
\newblock arxiv:1501.00946.

\bibitem{LangeNarasimhan83}
H.~Lange and M.~S. Narasimhan.
\newblock Maximal subbundles of rank two vector bundles on curves.
\newblock {\em Math. Ann.}, 266(1):55--72, 1983.

\bibitem{Lange83}
Herbert Lange.
\newblock Zur {K}lassifikation von {R}egelmannigfaltigkeiten.
\newblock {\em Math. Ann.}, 262(4):447--459, 1983.

\bibitem{LiZhang11}
Jiayu Li and Xi~Zhang.
\newblock The gradient flow of {H}iggs pairs.
\newblock {\em J. Eur. Math. Soc. (JEMS)}, 13(5):1373--1422, 2011.

\bibitem{NaitoKozonoMaeda90}
Hisashi Naito, Hideo Kozono, and Yoshiaki Maeda.
\newblock A stable manifold theorem for the {Y}ang-{M}ills gradient flow.
\newblock {\em Tohoku Math. J. (2)}, 42(1):45--66, 1990.

\bibitem{NarasimhanRamanan69}
M.~S. Narasimhan and S.~Ramanan.
\newblock Moduli of vector bundles on a compact {R}iemann surface.
\newblock {\em Ann. of Math. (2)}, 89:14--51, 1969.

\bibitem{NarasimhanSeshadri65}
M.~S. Narasimhan and C.~S. Seshadri.
\newblock Stable and unitary vector bundles on a compact {R}iemann surface.
\newblock {\em Ann. of Math. (2)}, 82:540--567, 1965.

\bibitem{Nelson69}
Edward Nelson.
\newblock {\em Topics in dynamics. {I}: {F}lows}.
\newblock Mathematical Notes. Princeton University Press, Princeton, N.J.;
  University of Tokyo Press, Tokyo, 1969.

\bibitem{Rade92}
Johan R{\aa}de.
\newblock On the {Y}ang-{M}ills heat equation in two and three dimensions.
\newblock {\em J. Reine Angew. Math.}, 431:123--163, 1992.

\bibitem{ReedSimonVol3}
Michael Reed and Barry Simon.
\newblock {\em Methods of modern mathematical physics. {III}. {S}cattering
  theory}.
\newblock Academic Press [Harcourt Brace Jovanovich, Publishers], New
  York-London, 1979.

\bibitem{Sibley15}
Benjamin Sibley.
\newblock Asymptotics of the {Y}ang-{M}ills flow for holomorphic vector bundles
  over {K}\"ahler manifolds: the canonical structure of the limit.
\newblock {\em J. Reine Angew. Math.}, 706:123--191, 2015.

\bibitem{Simon83}
Leon Simon.
\newblock Asymptotics for a class of nonlinear evolution equations, with
  applications to geometric problems.
\newblock {\em Ann. of Math. (2)}, 118(3):525--571, 1983.

\bibitem{Simpson88}
Carlos~T. Simpson.
\newblock Constructing variations of {H}odge structure using {Y}ang-{M}ills
  theory and applications to uniformization.
\newblock {\em J. Amer. Math. Soc.}, 1(4):867--918, 1988.

\bibitem{Simpson92}
Carlos~T. Simpson.
\newblock Higgs bundles and local systems.
\newblock {\em Inst. Hautes \'Etudes Sci. Publ. Math.}, (75):5--95, 1992.

\bibitem{Sjamaar95}
Reyer Sjamaar.
\newblock Holomorphic slices, symplectic reduction and multiplicities of
  representations.
\newblock {\em Ann. of Math. (2)}, 141(1):87--129, 1995.

\bibitem{Swoboda12}
Jan Swoboda.
\newblock Morse homology for the {Y}ang-{M}ills gradient flow.
\newblock {\em J. Math. Pures Appl. (9)}, 98(2):160--210, 2012.

\bibitem{UhlenbeckYau86}
K.~Uhlenbeck and S.-T. Yau.
\newblock On the existence of {H}ermitian-{Y}ang-{M}ills connections in stable
  vector bundles.
\newblock {\em Comm. Pure Appl. Math.}, 39(S, suppl.):S257--S293, 1986.
\newblock Frontiers of the mathematical sciences: 1985 (New York, 1985).

\bibitem{Wilkin08}
Graeme Wilkin.
\newblock Morse theory for the space of {H}iggs bundles.
\newblock {\em Comm. Anal.Geom.}, 16(2):283--332, 2008.

\bibitem{witten-hecke}
Edward Witten.
\newblock More on {G}auge {T}heory and {G}eometric {L}anglands.
\newblock arxiv:1506.04293.

\end{thebibliography}

\end{document}